\documentclass[11pt]{article}
\usepackage{amsmath}
\usepackage{amsfonts}
\usepackage{amsthm}
\usepackage{graphicx}
\usepackage{overpic}
\usepackage{a4wide}
\usepackage{amssymb} 
\usepackage{pictex}
\usepackage{rotating}  
\usepackage{color}
\usepackage{comment}

\numberwithin{equation}{section}

\newtheorem{theorem}{Theorem}[section]
\newtheorem{proposition}[theorem]{Proposition}
\newtheorem{lemma}[theorem]{Lemma}

\newtheorem{Definition}[theorem]{Definition}
\newenvironment{definition}{\begin{Definition}\rm}{\end{Definition}}
\newtheorem{Remark}[theorem]{Remark}
\newenvironment{remark}{\begin{Remark}\rm}{\end{Remark}}
\newtheorem{RHproblem}[theorem]{RH problem}
\newenvironment{rhproblem}{\begin{RHproblem}\rm}{\end{RHproblem}}

\newcommand{\C}{\mathbb{C}}

\newcommand{\R}{\mathbb{R}}

\newcommand{\ds}{\displaystyle}

\newcommand{\RR}{\mathcal{R}}
\renewcommand{\O}{\mathcal{O}}
\renewcommand{\Re}{{\rm Re} \,}
\renewcommand{\Im}{{\rm Im} \,}

\def\diag{\mathop{\mathrm{diag}}\nolimits}

\def\det{\mathop{\mathrm{det}}\nolimits}

\begin{document}
\title{Non-intersecting squared Bessel paths: \\ critical time and
double scaling limit}
\author{A.B.J.\ Kuijlaars, A.\ Mart\'{\i}nez-Finkelshtein, and F.\ Wielonsky}
%\date{August 17, 2008. Version 1}

\maketitle 

\begin{abstract}
We consider the double scaling limit for a model of $n$
non-intersecting squared 
Bessel processes in the confluent case: all paths start at time $t=0$
at the same positive value $x=a$, remain positive, and are conditioned
to end at time $t=1$ at $x=0$. After appropriate rescaling, the paths
fill a region in the $tx$--plane as $n\to \infty$ that intersects the
hard edge at $x=0$ at a critical time $t=t^{*}$. In a previous paper,
the scaling limits for the positions of the paths at time $t\neq
t^{*}$ were shown to be the usual scaling limits from random matrix
theory. Here, we describe the limit as $n\to \infty$ of the
correlation kernel at critical time $t^{*}$ and in the double scaling
regime. We derive an integral representation for the limit kernel
which bears some connections with the Pearcey kernel. The analysis is
based on the study of a $3\times 3$ matrix valued Riemann-Hilbert
problem by the Deift-Zhou steepest descent method. The main ingredient
is the construction of a local parametrix at the origin, out of the
solutions of a particular third-order linear differential equation,
and its matching with a global parametrix.
\end{abstract}

\tableofcontents

\section{Introduction and main results}

\subsection{Introduction}

We considered in \cite{KMW} a model of $n$ non-intersecting squared
Bessel paths in the confluent case. In this model, all paths start
at time $t=0$ at the same positive value $x=a > 0$ and end at time $t=1$ at $x=0$. Our
aim was to study the asymptotic behavior of the model as
$n\to\infty$. 

The positions of the squared Bessel paths at any given time $t \in (0,1)$
are a determinantal point process with a correlation kernel  that
is built out of the transition probability density function of the
squared Bessel process. In \cite{KMW} we found that, after appropriate 
scaling, the paths fill out a region in the $tx$ plane that we described explicitly. 
Initially, the paths stay away from the hard edge at $x=0$. At a certain critical time
$t^*$ the smallest paths come to the hard edge and then remain close to
it, as can be seen in Figure~\ref{fig:SqBessel-50paths1}.

%%%%%%%%%%%%%%%%%%%%%%%%%%%%%%%%%%%%%%%%%%%%%%%%%%%%%%%%%%
\begin{figure}[t]
\centering \begin{overpic}[scale=0.63]%
{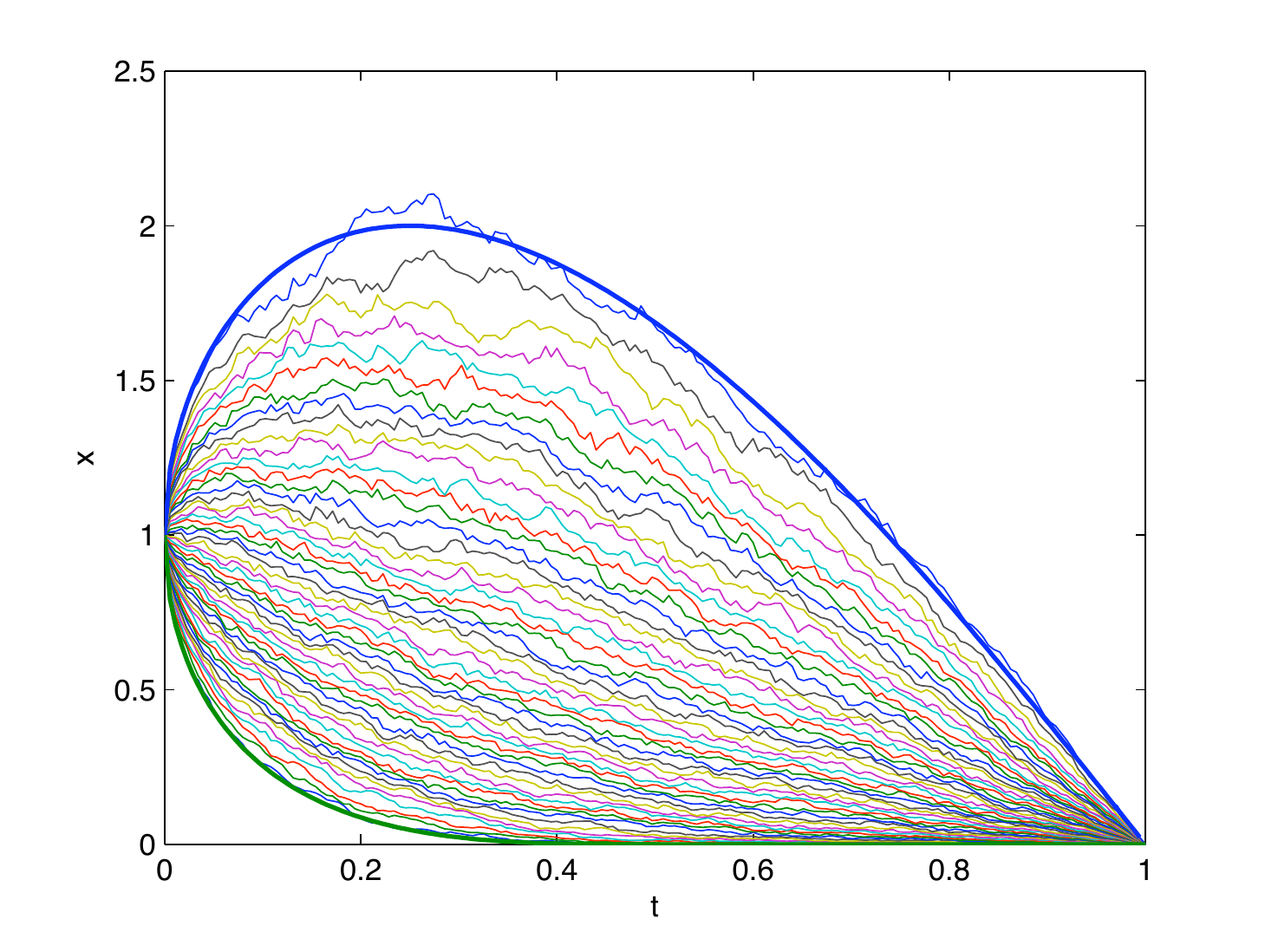}%
\end{overpic}
\caption{Numerical simulation of $50$ rescaled non-intersecting squared Bessel
paths with $a=1$.
Bold lines are the boundaries of the domain filled out by the paths as
their number increases.}\label{fig:SqBessel-50paths1}
\end{figure}
%%%%%%%%%%%%%%%%%%%%%%%%%%%%%%%%%%%%%%%%%%%%%%%%%%%%%%%%%%%%%%%%%%%%%%%

% \begin{figure}[htb]
% \centering 
% \begin{overpic}[scale=0.63]%
% {SqBessel-50paths5}
% \end{overpic}
% \caption{Numerical simulation of $50$ rescaled non-intersecting squared Bessel
% paths with $a=5$.}
% Bold line is the boundary of the domain filled out by the paths as
% their number increases}
% \label{fig:SqBessel-50paths5}
% \end{figure}
%%%%%%%%%%%%%%%%%%%%%%%%%%%%%%%%%%%%

In \cite{KMW} we also proved the local scaling limits of 
the correlation kernel as $n \to \infty$, that are typical from random matrix theory.
Thus we find the sine kernel in the bulk and the Airy kernel at the soft edges,
which includes the lower boundary of the limiting domain for $t < t^*$. 
For $t > t^*$, we find the Bessel kernel at the hard edge $0$,
see \cite[Theorems 2.7-2.9]{KMW} . 

In this paper we consider the critical time $t = t^*$. We describe 
the transition from the Airy kernel to the Bessel kernel by means of a new one-parameter family
of limiting kernels that arise as limiting kernels around the critical time.
This soft-to-hard edge transition is different from
previously studied ones in \cite{BF} or \cite{CK2}, but is related to the
one in \cite{BoKu}. 

We consider the squared Bessel process with parameter $\alpha > -1$,
with transition probabality density $p_{t}^{\alpha}$ given by, see \cite{BS, KatTan1, KO},
\begin{equation} \label{ptalpha}
\begin{aligned} 
   p_{t}^{\alpha} (x,y) & = \frac{1}{2t}\left(\frac{y}{x} \right)^{\alpha/2}
e^{- (x+y)/(2t)}
I_{\alpha}\left(\frac{\sqrt{xy}}{t} \right), &\quad x,y>0, \\
  p_{t}^{\alpha} (0,y) &= \frac{y^{\alpha}}{(2t)^{\alpha+1}\Gamma (\alpha+1) }e^{-y/(2t)}, & \quad  y>0,
\end{aligned}
\end{equation}
where $I_{\alpha}$ denotes the modified Bessel function of the
first kind of order $\alpha$,
\begin{equation} \label{Ialpha}
	I_{\alpha} (z)=
    \sum_{k=0}^{\infty}\frac{(z/2)^{2k+\alpha}}{k! \, \Gamma(k+\alpha+1)}.
\end{equation}

A remarkable theorem of Karlin and McGregor \cite{KM} describes the
distribution of $n$ independent non-intersecting copies of a one-dimensional
diffusion process at any given
time $t$ in terms of its transition probabilities. In the case of the squared
Bessel process, with all starting points at time $0$ in $a > 0$ and all ending points at a later
time $T > 0$ in $0$, the theorem implies
that the positions of the paths at time $t \in (0,T)$ have
the joint probability density
\begin{equation} \label{biorthogonal}
    \mathcal P(x_1, \ldots, x_n) = \frac{1}{Z_n} \det[ f_j(x_k)]_{j,k=1, \ldots,n} \det[g_j(x_k)]_{j,k=1, \ldots, n}
    \end{equation}
 on $(\mathbb R^+)^n$,
with  functions   
    \begin{align} \label{fjcoalesce1}
       f_{2j-1}(x) & = x^{j-1} p_t^{\alpha}(a,x), && \qquad j =1, \ldots, n_1 := \lceil{n/2}\rceil, \\
       \label{fjcoalesce2}
       f_{2j}(x) & = x^{j-1} p_t^{\alpha+1}(a,x), && \qquad j=1, \ldots, n_2 := n-n_1, \\
       \label{gjcoalesce}
       g_j(x) & = x^{j-1} e^{-\frac{x}{2(T-t)}}, && \qquad j = 1, \ldots, n,
    \end{align}
see \cite[Proposition 2.1]{KMW}. The constant $Z_n$ is a normalizing constant
which is taken so that \eqref{biorthogonal} defines a probability density function
on $(\mathbb R^+)^n$.

Formula \eqref{biorthogonal} is characteristic of a biorthogonal
ensemble \cite{Bo}. It is known that \eqref{biorthogonal} defines
a  determinantal point with correlation kernel $\widehat K_n$ 
\begin{equation} \label{corkernelbiorthogonal} 
	\widehat K_n(x,y) = \widehat K_n(x,y;t, T) = \sum_{j,k=1}^n f_j(x) \left[ A^{-1} \right]_{k,j} g_k(y) 
	\end{equation}
where $\left[ A^{-1} \right]_{k,j}$ is the $(k,j)$th entry of the inverse of
the matrix
\[ A = \left[ \int_0^{\infty} f_j(x) g_k(x) dx \right]_{j,k=1, \ldots, n}. \]
This means that
\begin{equation} \label{corkernproperty1} 
	\mathcal P(x_1, \ldots, x_n) = \frac{1}{n!} \det \left[ \widehat K_n(x_i,x_j) \right]_{i,j=1, \ldots,n} 
	\end{equation}
and for each $m = 1, \ldots, n-1$,
\begin{equation} \label{corkernproperty2} 
	\frac{n!}{(n-m)!} \int_0^{\infty} \cdots \int_0^{\infty} \mathcal P(x_1, \ldots, x_n) dx_{m+1} \cdots dx_n
	= \det\left[ \widehat K_n(x_i,x_j) \right]_{i,j=1, \ldots,m}. 
	\end{equation}
Determinantal processes arise naturally in probability theory, see
e.g.\ \cite{Jo2, So}. The connection with models of non-intersecting paths is well-known
see  \cite[Chapter 10]{For} and references therein. 
Non-intersecting squared Bessel paths and related continuous models with a wall are studied in  \cite{KIK,KatTan1,KatTan2,KO,TraWid}.
Non-intersecting discrete random walks with a wall are considered in the recent papers \cite{BFPSW,BFS,BoKu,WarWin}.

As in \cite{KMW} we introduce a time rescaling
\[ t \mapsto \frac{t}{2n}, \qquad T \mapsto \frac{1}{2n} \]
and we consider the rescaled kernels
\begin{equation} \label{kernelKnt} K_n(x,y;t) = e^{-n(x-y)/(1-t)} \widehat K_n\left(x,y; \frac{t}{2n}, \frac{1}{2n} \right),
	\qquad x,y > 0, \qquad 0 < t < 1,
\end{equation}
that depend on the variable $t$. The prefactor $e^{-n(x-y)/(1-t)}$ does not 
affect the correlation functions \eqref{corkernproperty2}.
We define  $w_{1,n}$, $w_{2,n}$ on $[0,\infty)$ by 
\begin{equation}
\begin{aligned} \label{weights1}
w_{1,n}(x) &= x^{\alpha/2} \exp\left(-\frac{nx}{t(1-t)}\right) I_{\alpha}\left(\frac{2n\sqrt{ax}}{t}
\right), \\
w_{2,n}(x) &=  x^{(\alpha+1)/2} \exp\left(-\frac{nx}{t(1-t)}\right) I_{\alpha+1}\left(\frac{2n\sqrt{ax}}{t}
\right),
\end{aligned}
\end{equation}
as in \cite[equation (2.20)]{KMW}. 

Then the kernel \eqref{kernelKnt} is expressed
in terms of a RH problem. Indeed we have
\begin{equation}\label{defK} 
	K_n(x,y;t) = \frac{1}{2\pi i(x-y)} \begin{pmatrix} 0 & w_{1,n}(y) &
	w_{2,n}(y) \end{pmatrix} 
    Y_+^{-1}(y) Y_+(x) \begin{pmatrix} 1 \\ 0 \\ 0 \end{pmatrix} 
\end{equation}
where $Y$ is the solution of the following matrix valued
Riemann-Hilbert problem, see \cite{KMW}:

\begin{rhproblem} \label{rhpforY}
Find $Y:\C\setminus\R\to\C^{3\times3}$ such that
\begin{enumerate}
\item $Y$ is analytic in $\mathbb{C} \setminus [0,\infty)$.
\item On the positive real axis, $Y$ possesses continuous boundary values $Y_+$
(from the upper half plane) 
and $Y_-$ (from the lower half plane), and
\begin{equation}  \label{Yjump}
   Y_+(x) = Y_-(x) \begin{pmatrix}
                   1 & w_{1,n} (x)& w_{2,n} (x)\\
                   0 & 1 & 0 \\
                   0 & 0 & 1
                   \end{pmatrix}, \qquad x > 0,
\end{equation}
\item $Y(z)$ has the following behavior at infinity:
\begin{equation}  \label{Yasym}
    Y(z) = \left( I + \mathcal{O}\left(\frac1z\right) \right)\begin{pmatrix}
            z^{n} & 0 & 0 \\
            0 & z^{-n_1} & 0 \\
            0 & 0 & z^{-n_2}
            \end{pmatrix} , \quad z
        \to \infty, \quad z\in \mathbb{C} \setminus \R,
\end{equation}
where $n_1 = \lceil{n/2}\rceil$ and $n_2 = \lfloor n/2 \rfloor$.
\item $Y(z)$ has the following  behavior near the origin, as $z\to 0$,
$z\in \C\setminus [0,\infty)$,
\begin{equation}  \label{Yedge}
    Y(z) = \mathcal{O}
        \begin{pmatrix}
            1 & h(z) & 1 \\
            1 & h(z) & 1 \\
            1 & h(z) & 1
            \end{pmatrix},
        \mbox{ with }
h(z)=\left\{
        \begin{array}{ccc}
|z|^{\alpha}, & \mbox{ if } & -1<\alpha<0,\\
\log|z|, & \mbox{ if } & \alpha=0,\\
1, & \mbox{ if } & 0<\alpha.\\
            \end{array}\right.
\end{equation}
The $\mathcal{O}$ condition in \eqref{Yedge} is to
be taken entry-wise.
\end{enumerate}
\end{rhproblem}

This RH problem has a unique solution given in terms of multiple
orthogonal polynomials for the modified Bessel weights
\eqref{weights1}.

It was proven in \cite[Proposition 2.3 and Theorem~2.4]{KMW} that in this scaling there
is a critical time 
\begin{equation} \label{tstar}
	 t^{*}= \frac{a}{a+1} 
	 \end{equation}
depending only on the starting value $a$. For every $t \in (0,1)$, we have that
\[ \lim_{n \to \infty} \frac{1}{n} K_n(x,x; t) = \rho(x; t) \]
exists, where the limiting density $\rho(x;t)$ is supported on
an interval $[p(t), q(t)]$ with $p(t) > 0$ if $t < t^*$ and
$p(t) = 0$ if $t > t^*$. The results of \cite{KMW} were obtained from a
steepest descent analysis of the above RH problem for values of $t \neq t^*$.
In this paper we develop the steepest descent analysis at the critical time.

\subsection{Statement of results}

%%%%%%%%%%%%%%%%%%%%%%%%%%%%%%%%%%%%%%%%%%%%%%%%%%%%%%%%%%
\begin{figure}[t]
\centering 
\begin{overpic}[scale=1.6]% 
{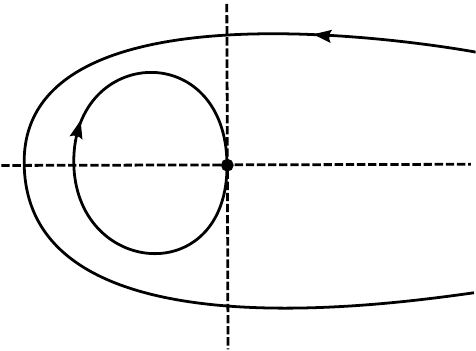}%
%{contourfct}%
\put(49,34){$0 $}
  \put(19,47){$\Gamma  $}
  \put(90,58){$\Sigma  $}
    \end{overpic}
\caption{The contours of integration $\Gamma$ and $\Sigma$ used 
in the definition of the critical kernel \eqref{Kcritical}.
The contour $\Gamma$ consists of a closed loop in the left half-plane
tangent to the origin and is oriented clockwise. The contour $\Sigma$ is
an unbounded loop oriented counterclockwise and encircling $\Gamma$. 
}
\label{fig:pathsforkernel}
\end{figure}
%%%%%%%%%%%%%%%%%%%%%%%%%%%%%%%%%%%%

The main result of our paper is the following theorem.
\begin{theorem}\label{corrkernel}
Let $K_n$ be the correlation kernel \eqref{defK} for the positions
of the rescaled non-intersecting squared Bessel paths starting at $a > 0$
at time $0$ and ending at zero at time $1$. 
Let $t^* = a/(a+1)$ as in \eqref{tstar} and
\begin{equation} \label{cstar}
	c^{*} = t^{*} (1-t^{*}) = \frac{a}{(a+1)^2}.
\end{equation}

Then we have, for every fixed $\tau\in\R$, and $x,y>0$,
\begin{equation} \label{Kcritaslimit}
	\lim_{n\to\infty}\frac{c^{*}}{n^{3/2}}K_{n}\left(\frac{c^{*}x}{n^{3/2}},
	\frac{c^{*}y}{n^{3/2}}; t^{*} - \frac{c^* \tau}{\sqrt{n}} 	\right) = K^{\text{crit}}_{\alpha} (x,y;\tau),
\end{equation}
where $K^{\text{crit}}_{\alpha}$ is the kernel 
\begin{equation}
\label{Kcritical}
	K^{\text{crit}}_{\alpha} (x,y;\tau) = \frac{1}{(2\pi i)^2}\, \int_{t\in \Gamma} \int_{s\in \Sigma} \frac{t^{\alpha}}{s^{\alpha}} 
	e^{\tau/t+1/(2t^{2})-\tau/s-1/(2s^{2})}e^{xt-ys}\, \frac{dt ds}{s-t}.
\end{equation}
The contours $\Gamma$ and $\Sigma$ in \eqref{Kcritical} are as in Figure~\ref{fig:pathsforkernel}.
The fractional powers $s^{\alpha}$ and $t^{\alpha}$ in \eqref{Kcritical} are defined with a 
branch cut on the positive semi-axis, i.e., $0 < \arg s, \arg t < 2\pi$.
\end{theorem}

We prove Theorem~\ref{corrkernel} by an asymptotic analysis
of the RH problem \ref{rhpforY} by means of the steepest descent analysis
of Deift and Zhou, as we did in  \cite{KMW} for the non-critical times.

At a certain stage in the analysis we have to construct a local parametrix
at the origin $x=0$ (the hard edge). This was done in \cite{KMW} with 
the Bessel parametrix. We also
had to construct an Airy parametrix at another position (a soft edge). In the critical
case that we are considering in this paper this other position coincides with
the origin. The coalescing of the soft edge with the hard edge leads to
the construction of a new local parametrix at the origin. The construction
uses a new model Riemann-Hilbert problem that we describe in the next
subsection.
The functions that appear in the model RH problem ultimately lead to
the expression \eqref{Kcritical} for the limiting kernels.

\subsection{Riemann-Hilbert problem}

The model RH problem is defined on the contour $\Sigma_{\Phi}$ shown in Figure~\ref{fig:localanalysis2bis}.
It consists of the six rays $\arg z = 0, \pm \pi/4, \pm 3 \pi/4$, oriented from left to right.

%%%%%%%%%%%%%%%%%%%%%%%%%%%%%%%%%%%%%%%%%%%%%%%%%%%%%%%%%%
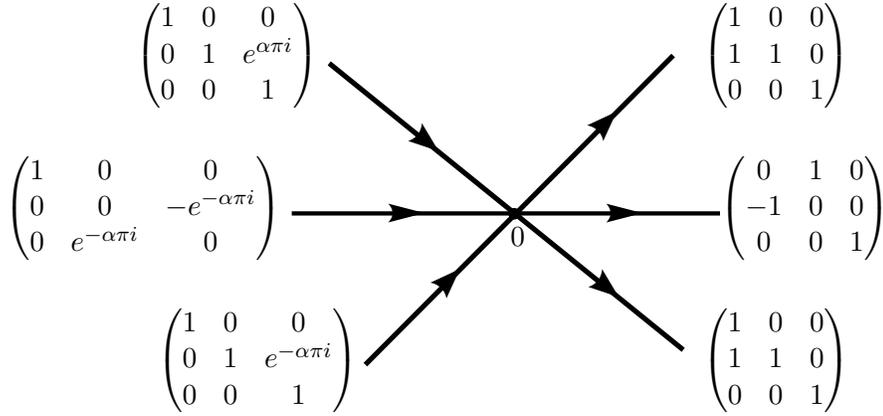
\begin{figure}[th]
\centering \vspace*{1cm}

\begin{overpic}[scale=1.2]%
{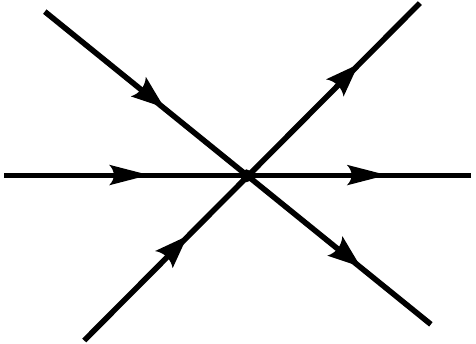}%
   \put(51,28){$0$}
   \put(99,35){$\begin{pmatrix} 0 & 1 & 0 \\ -1 & 0 & 0 \\ 0 & 0 & 1 \end{pmatrix}$}
   \put(95,70){$\begin{pmatrix} 1 & 0 & 0 \\ 1 & 1 & 0 \\ 0  & 0 & 1 \end{pmatrix}$}
   \put(95,0){$\begin{pmatrix} 1 & 0 & 0 \\ 1 & 1 & 0 \\ 0  & 0 & 1 \end{pmatrix}$}
   
   \put(-65,35){$\begin{pmatrix} 1 & 0 & 0 \\ 0 & 0 & -e^{-\alpha \pi i} \\ 0 & e^{-\alpha \pi i} & 0 \end{pmatrix}$}
   \put(-35,70){$\begin{pmatrix} 1 & 0 & 0 \\ 0 & 1 & e^{\alpha \pi i} \\ 0 & 0 & 1 \end{pmatrix}$}
   \put(-30,0){$\begin{pmatrix} 1 & 0 & 0 \\ 0 & 1 & e^{-\alpha \pi i} \\ 0 & 0 & 1 \end{pmatrix}$}
   
   \end{overpic}
   \vspace*{5mm}
   
\caption{Contour $\Sigma_{\Phi}$ and jump matrices $J_{\Phi_{\alpha}}$ in the RH problem for $\Phi_{\alpha}$.}
\label{fig:localanalysis2bis}
\end{figure}
%%%%%%%%%%%%%%%%%%%%%%%%%%%%%%%%%%%%

\begin{rhproblem} \label{rhpforPhi}
Let $\alpha > -1$ and $\tau \in \mathbb C$.
The RH problem is to find $\Phi_{\alpha} = \Phi_{\alpha}(\cdot; \tau): 
\C \setminus\Sigma_{\Phi} \to \C^{3\times3}$ such that
\begin{enumerate}
\item $\Phi_{\alpha}$ is analytic in $\mathbb{C} \setminus \Sigma_{\Phi}$.
\item $\Phi_{\alpha}$ has boundary values on each part of  $\Sigma_{\Phi} \setminus \{ 0 \}$
satisfying
\begin{equation} \label{Phijumps}
	 \Phi_{\alpha,+}(z;\tau) = \Phi_{\alpha,-}(z;\tau) J_{\Phi_{\alpha}}(z), \qquad z \in \Sigma_{\Phi} 
	 \end{equation}
	where the jump matrices $J_{\Phi_{\alpha}}$ are shown in Figure~\ref{fig:localanalysis2bis}.
\item 
Let $\omega = e^{2 \pi i/3}$ and
\begin{equation} \label{thetak}
	\theta_k(z) = \theta_k(z; \tau) = \frac{3}{2} \omega^{2k} z^{2/3}
		+ \omega^k \tau z^{1/3}, \qquad k = 1,2,3. \end{equation}
Then as $z \to \infty$, we have
\begin{multline} \label{Phiasymptotics1} 
	\Phi_{\alpha}(z; \tau) =
    % \sqrt{\frac{2\pi}{3}} i e^{-\tau^2/6} z^{-\alpha/3}
    \frac{i z^{-\alpha/3}}{\sqrt{3}}
    	\begin{pmatrix} z^{1/3} & 0 & 0 \\ 0 & 1 & 0 \\ 0 & 0 & z^{-1/3} \end{pmatrix}
    \begin{pmatrix} \omega & \omega^2 & 1 \\ 1 & 1 & 1 \\ \omega^2 & \omega & 1 \end{pmatrix}
    \begin{pmatrix} e^{\alpha \pi i/3} & 0 & 0 \\ 0 & e^{-\alpha \pi i/3} & 0 \\ 0 & 0 & 1 \end{pmatrix} \\
    \left(I +  \O(z^{-1/3})\right)
    \begin{pmatrix} e^{\theta_1(z;\tau)} & 0 & 0 \\ 0 & e^{\theta_2(z; \tau)}
    & 0 \\ 0 & 0 & e^{\theta_3(z;\tau)} \end{pmatrix}, \quad \Im z > 0,
    \end{multline}
and 
\begin{multline} \label{Phiasymptotics2} 
	\Phi_{\alpha}(z; \tau) =
    % \sqrt{\frac{2\pi}{3}} i e^{-\tau^2/6}  z^{-\alpha/3}
    \frac{i z^{-\alpha/3}}{\sqrt{3}}
    	\begin{pmatrix} z^{1/3} & 0 & 0 \\ 0 & 1 & 0 \\ 0 & 0 & z^{-1/3} \end{pmatrix}
    \begin{pmatrix} \omega^2 & -\omega & 1 \\ 1 & -1 & 1 \\ \omega & -\omega^2 & 1 \end{pmatrix}
    \begin{pmatrix} e^{-\alpha \pi i/3} & 0 & 0 \\ 0 & e^{\alpha \pi i/3} & 0 \\ 0 & 0 & 1 \end{pmatrix} \\
    \left(I +  \O(z^{-1/3})\right)
    \begin{pmatrix} e^{\theta_2(z;\tau)} & 0 & 0 \\ 0 & e^{\theta_1(z; \tau)}
    & 0 \\ 0 & 0 & e^{\theta_3(z;\tau)} \end{pmatrix}, \quad \Im z < 0,
    \end{multline}
\item As $z \to 0$ we have
\begin{align} \label{Phiat01}
 \Phi_{\alpha}(z;\tau) \begin{pmatrix} z^{\alpha} & 0 & 0 \\ 0 & z^{\alpha} & 0 \\ 0 & 0 & 1 \end{pmatrix} & = \O(1),
 	\qquad 0 < |\arg z| < \pi/4, \\ \label{Phiat02}
 \Phi_{\alpha}(z;\tau) \begin{pmatrix} 1 & 0 & 0 \\ 0 & z^{\alpha} & 0 \\ 0 & 0 & 1 \end{pmatrix} & = \O(1),
  \qquad \pi/4 < |\arg z| < 3 \pi/4, \\ \label{Phiat03}
 \Phi_{\alpha}(z;\tau) \begin{pmatrix} 1 & 0 & 0 \\ 0 & z^{\alpha} & 0 \\ 0 & 0 & z^{\alpha} \end{pmatrix} & = \O(1),
 	\qquad 3 \pi/4 < |\arg z| < \pi.
	\end{align}
\end{enumerate}
\end{rhproblem}

Note that the parameter $\tau$ appears in \eqref{thetak}
and in the asymptotic conditions \eqref{Phiasymptotics1} and \eqref{Phiasymptotics2}
of the RH problem.

We prove the following.

\begin{theorem} \label{thm:PhiRHP} 
Let $\alpha > -1$ and $\tau \in \mathbb C$.
The RH problem \ref{rhpforPhi} for $\Phi_{\alpha}$ has a unique solution with 
\begin{equation} \label{detPhi} 
	\det \Phi_{\alpha}(z; \tau) =  z^{-\alpha}, \qquad z \in \mathbb C \setminus \Sigma_{\Phi}.
	\end{equation}
The critical kernel \eqref{Kcritical} satisfies
\begin{equation} \label{KcritRHP}
K^{crit}_{\alpha}(x,y;\tau)
    = \frac{1}{2\pi i(x-y)} \begin{pmatrix} -1 & 1 & 0
    \end{pmatrix} \Phi_{\alpha,+}^{-1}(y;\tau) \Phi_{\alpha,+}(x;\tau) \begin{pmatrix} 1 \\ 1
    \\ 0 \end{pmatrix},
\end{equation}
for $x, y > 0$ and $\tau \in \mathbb R$.
\end{theorem}

The uniqueness statement in Theorem~\ref{thm:PhiRHP} follows from standard arguments
where one first proves \eqref{detPhi}. The existence of a solution follows
from an explicit construction of $\Psi_{\alpha}$, given in Proposition 5.2, in terms of solutions of the third order ODE
\begin{equation} \label{ODEforp} 
	z p''' + \alpha p'' - \tau p' - p = 0. 
	\end{equation}
A particular solution of this equation is given by
\begin{equation} \label{defp} p(z) = \int_{\Gamma} t^{\alpha-3} e^{\tau/t} e^{1/(2t^2)} e^{zt} dt
\end{equation}
where $\Gamma$ is the closed contour shown in Figure~\ref{fig:pathsforkernel}.

The inverse matrix $\Phi_{\alpha}^{-1}$ 
is built out of solutions of the adjoint equation
\begin{equation} \label{ODEforq} 
	z q''' + (3-\alpha) q'' - \tau q' + q = 0
	\end{equation}
which has the special solution
\begin{equation} \label{defq} 
	q(z) = \int_{\Sigma} s^{-\alpha} e^{-\tau/s} e^{-1/(2s^2)} e^{-zs} ds 
	\end{equation}
where $\Sigma$ is also shown in Figure~\ref{fig:pathsforkernel}.

In terms of these functions the kernel \eqref{Kcritical}, \eqref{KcritRHP}, can
also be written as
\begin{multline}  \label{bilinear}
	2 \pi i (x-y) K^{crit}_{\alpha}(x,y; \tau) \\ 
	= \left[q'' (y) -(\alpha-2) q' (y) -\tau q (y) \right] p (x) 
		+ \left[ -yq' (y) + (\alpha-1)q (y) \right] p'(x) \\
	+ yq (y)p'' (x).
	\end{multline}
	
For $y=x$ the right-hand side of \eqref{bilinear} is the bilinear concomitant 
\cite{Bert,Ince} which is constant for
any two solutions $p$ and $q$ of the differential equations \eqref{ODEforp}, \eqref{ODEforq}, 
and which turns out to be zero for the two particular solutions \eqref{defp} and \eqref{defq}.

\begin{remark}
There are  solutions of the differential equations \eqref{ODEforp} and \eqref{ODEforq} that can
be written as integrals of Bessel functions. In particular, we have that  
\begin{equation} \label{altpandq}
\begin{aligned}
\mathfrak p(z) & = z^{-\alpha/2} \int_0^{+\infty} u^{\alpha/2} e^{-\tau u - u^2/2} J_\alpha(2\sqrt{z u} )\, du \quad \text{and} \\ 
\mathfrak q(z)& = z^{\alpha/2} \int_{-i\infty}^{+i \infty} v^{-\alpha/2} e^{\tau v + v^2/2} J_\alpha(2\sqrt{z v} )\, dv,
\end{aligned}
\end{equation}
are solutions of  \eqref{ODEforp} and \eqref{ODEforq}, respectively, where $J_\alpha$ is the Bessel function of the first kind of order $\alpha$.

Based on a similarity with formulas by Desrosiers and Forrester \cite[Proposition 5]{DesFor} for a perturbed chiral GUE,  
we suspect that it should be possible to write an alternative expression for the 
critical kernel in \eqref{Kcritical} in terms of  the functions \eqref{altpandq}, namely 
$$
\int_{u\in \R_+} \int_{v\in  i\R} \left(\frac{u}{v}\right)^{\alpha/2} e^{\frac{v^2}{2}-\frac{u^2}{2}+\tau v-\tau u} J_\alpha(2\sqrt{x u} ) J_\alpha(2\sqrt{y v} )\, \frac{du d v}{u-v}.
$$ 
Unfortunately, we have not been able to make this identification. 
Observe however that for $\alpha=-1/2$ the double integral above reduces to the so-called symmetric Pearcey kernel
 $\mathcal K(\sigma_1; \sigma_2; \eta)$, with
$\sigma_1=x^2/\sqrt{2}$,  $\sigma_2=y^2/\sqrt{2}$,  $\eta=\sqrt{2}\tau$. The correlation kernel
$$
\mathcal K(\sigma_1; \sigma_2; \eta)=\frac{2}{\pi^2 i}\, \int_{u\in C} \int_{x\in \R_+} e^{-\eta x^2+\eta u^2 -x^4 + u^4} \cos(\sigma_1 x) \cos(\sigma_2 u) \frac{u dx du}{u^2-x^2},
$$
where $C$ is the contour in $\C$ consisting of rays from $\infty e^{i \pi/4}$ to $0$ to $\infty e^{-i\pi /4}$,
was introduced by Borodin and Kuan in \cite{BoKu};  the authors point out the possible connection with the 
non-intersecting Bessel paths in the critical regime, as it seems to be the case.
\end{remark}

%\subsection{Outline of the paper}

\section{First and second transformation} \label{sec2}

The steepest descent analysis consists of a sequence of transformations
\[ Y \mapsto X \mapsto U \mapsto T \mapsto S \mapsto R \]
which leads to a RH problem for $R$, normalized at infinity
and with jump matrices that are close to the identity matrix if $n$ is large.

We start from the RH problem \ref{rhpforY} for $Y$, 
stated in the introduction. The RH problem depends on the
parameters $n$ and $t$. We assume that $n$ is large, and $t$ is close to
the critical value $t^*$. Eventually we will take
the double scaling limit 
\begin{equation} \label{doublescaling} 
	n \to \infty, \qquad t \to t^*, \quad \textrm{ such that } \quad \sqrt{n} (t^* - t) = c^* \tau \quad  \text{ remains fixed}. 
	\end{equation} 
But throughout the transformations in Sections \ref{sec2}--\ref{final},
we assume that $n$ and $t$ are finite and fixed.

The first transformation is the same as in \cite{KMW}.

\subsection{The first transformation}

The first transformation $Y \mapsto X$ is based on special properties of the 
modified Bessel functions that appear in the jump matrix \eqref{Yjump} via the
two weights \eqref{weights1}. The result of the first transformation will be that
the jump matrix on $[0,\infty)$ is simplified at the expense of
introducing jumps on $(-\infty,0)$ and on  two unbounded contours $\Delta_2^{\pm}$
that are shown in Figure~\ref{fig:first_transformation}. The contours $\Delta_2^{\pm}$
are the boundaries of an unbounded lense around the negative real axis.

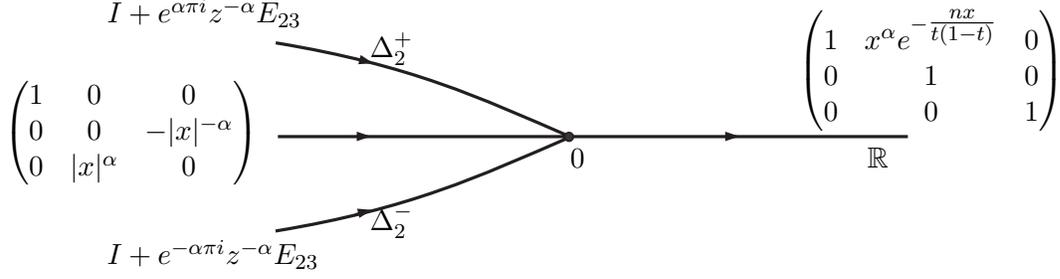
\begin{figure}[t]
\centering \begin{overpic}[scale=1.2]%
{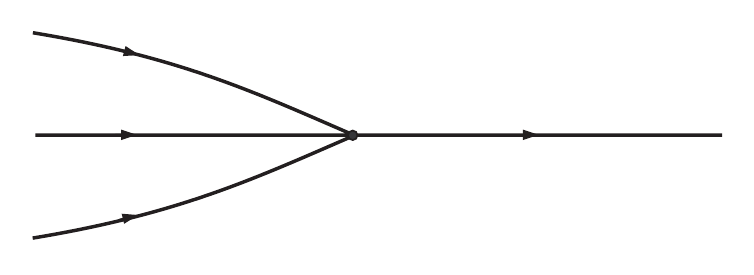}%
%\centering \begin{overpic}[scale=1.2]{first_contour}
      \put(90,14){$\R $}
       \put(18,30){$\Delta_2^+ $}
       \put(18,5){$\Delta_2^- $}
       \put(47,14){$0$}
       \put(-35,18){$\begin{pmatrix} 1 & 0 & 0 \\ 0 & 0 & -|x|^{-\alpha} \\
        0 & |x|^{\alpha} & 0 \end{pmatrix}$}
        \put(-20,35){$I + e^{\alpha \pi i} z^{-\alpha} E_{23}$}
        \put(-20,0){$I + e^{- \alpha \pi i} z^{-\alpha} E_{23}$}
				\put(80,27){$\begin{pmatrix} 1 & x^{\alpha} e^{-\tfrac{nx}{t(1-t)}} & 0 \\
					0 & 1& 0 \\ 0 & 0 & 1 \end{pmatrix}$}
\end{overpic}
\caption{Contour $\Sigma_X = \mathbb R \cup \Delta_2^{\pm}$ and jump matrices $J_X$ 
in the RH problem for $X$.}
\label{fig:first_transformation}
\end{figure}

Here and in the sequel, $E_{ij}$ denotes
the $3 \times 3$ elementary matrix whose entries are all $0$, except for
the $(i,j)$-th entry, which is $1$.

\begin{definition}
We let $	y_1(z) = z^{(\alpha+1)/2} I_{\alpha + 1}(2 \sqrt{z})$ and 
$y_2(z) = z^{(\alpha+1)/2} K_{\alpha + 1}(2 \sqrt{z})$ 
where $K_{\alpha+1}$ is the modified Bessel function of second kind of order $\alpha + 1$.
Then we define $X$ in terms of $Y$ as follows
\begin{multline} \label{Xdef}
   X (z) = C_1
        Y (z)
        \begin{pmatrix}
        1 & 0 & 0 \\
        0 & 1 & 0 \\
        0 & 0 & \frac{n \sqrt{a}}{t} \end{pmatrix} \\
        \times
        \begin{pmatrix}
        1 & 0 & 0\\
        0 & 2y_{2} \left( \frac{n^2 a z}{t^2} \right) & - z^{-\alpha}y_{1}  \left( \frac{n^2 a z}{t^2} \right)\\
        0 & -2y_{2}' \left( \frac{n^2 a z}{t^2} \right) & z^{-\alpha} y_{1}' \left( \frac{n^2 a z}{t^2} \right)
        \end{pmatrix} 
    \begin{pmatrix}
    1 &0&0\\
    0 &  \left(\frac{t}{n \sqrt{a}}\right)^{\alpha} & 0\\
    0 & 0 & -2\pi i \left(\frac{t}{n \sqrt{a}}\right)^{\alpha}
    \end{pmatrix} \\
    \times 
    \begin{cases} 
    		I -  e^{\alpha \pi i} z^{-\alpha} E_{23}, & \text{for $z$ in the upper part of the lens,}  \\
    		I + e^{-\alpha \pi i} z^{-\alpha} E_{23}, & \text{for $z$ in the lower part of the lens,} \\
    		I & \text{elsewhere.}
    		\end{cases}
\end{multline}
where $C_1$ is some constant matrix, see \cite[Equation (3.12)]{KMW}
for its definition.
\end{definition}

Then $X$ is the unique solution of the following RH problem, see \cite[Section 3]{KMW} for details.
\begin{rhproblem} \label{rhpforX} \

\begin{enumerate}
\item $X$ is defined and analytic in $\mathbb{C} \setminus \Sigma_X$ where $\Sigma_X = \R \cup \Delta_2^{\pm}$.
\item On $\Sigma_X$ we have the jump 
\begin{equation} \label{Xjump}
	X_+ = X_- J_X 
	\end{equation}
where the jump matrices $J_X$ are as in Figure~\ref{fig:first_transformation}.
\item As $z \to \infty$ we have
\begin{multline} \label{Xasymptotics}
   X(z) =
   \left(I + \mathcal{O}\left(\frac{1}{z}\right) \right)
    \begin{pmatrix}
    1 & 0 & 0\\
    0 & z^{(-1)^n/4} & 0 \\
    0 & 0 &  z^{- (-1)^n/4}
    \end{pmatrix}
     \begin{pmatrix}
    1 & 0 & 0 \\
    0 & \frac{1}{\sqrt{2}} & \frac{1}{\sqrt{2}} i \\
    0 & \frac{1}{\sqrt{2}} i & \frac{1}{\sqrt{2}}
    \end{pmatrix} \\
    \begin{pmatrix}
    1 & 0 & 0 \\
    0 & z^{\alpha/2} & 0 \\
    0 & 0 & z^{-\alpha/2}
    \end{pmatrix}
\begin{pmatrix}
            z^{n}  & 0 & 0 \\
            0 & z^{-n/2} e^{-2n\sqrt{az}/t} & 0 \\
            0 & 0 & z^{-n/2} e^{2n\sqrt{az}/t}
            \end{pmatrix}.
\end{multline}
\item $X(z)$ has the same behavior as $Y(z)$ near
the origin, see \eqref{Yedge}, as $z\to 0$ from outside
the lens around $(-\infty, 0]$. If $z \to 0$ within the lens around
$(-\infty, 0]$, then 
\begin{equation} \label{Xedge2}
    X(z) = \left\{ \begin{array}{cl}
    \O \begin{pmatrix} 1 & |z|^{\alpha} & 1 \\  1 & |z|^{\alpha} & 1 \\ 1 & |z|^{\alpha} & 1 \end{pmatrix}
    & \textrm{if } \alpha < 0, \\
    \O \begin{pmatrix} 1 & \log |z| & \log |z| \\ 1 & \log |z| & \log |z| \\ 1 & \log |z| & \log |z| \end{pmatrix}
    & \textrm{if } \alpha = 0, \\
    \O\begin{pmatrix} 1 & 1 & |z|^{-\alpha} \\ 1 & 1 & |z|^{-\alpha} \\1 & 1 & |z|^{-\alpha} \end{pmatrix}
    & \textrm{if } \alpha > 0.
    \end{array} \right. \end{equation}
\end{enumerate}
\end{rhproblem}

\subsection{The Riemann surface}

In the second transformation we are going to use certain functions that come
from a Riemann surface. In  \cite[Section 4]{KMW} we used the  Riemann surface 
associated with the algebraic equation
\begin{equation} \label{RSequation1}
    z = \frac{1-k \zeta}{\zeta(1-t(1-t)\zeta)^2}, \qquad k = (1-t)(t-a(1-t)).
    \end{equation}
This equation was derived from a formal WKB analysis of the differential equation
\begin{multline} \label{DEforMOP2}
  zy'''(z) + \left((2+\alpha) - \frac{2nz}{t(1-t)}\right) y''(z) \\
  +
   \left(\frac{n^2z}{t^2(1-t)^2}  + \frac{n(n-\alpha-2)}{t(1-t)} - \frac{an^2}{t^2}\right)y'(z)
   - \frac{n^3}{t^2(1-t)^2} y (z)=0,
\end{multline}
see \cite{CV2} and \cite[Equation (2.21)]{KMW}, that is satisfied by 
the multiple orthogonal polynomials associated with the weights
\eqref{weights1}.

There are three inverse functions to \eqref{RSequation1}, which behave
as $z\to\infty$ as
\begin{align}
\zeta_{1} (z) & = \frac{1}{z}+ \mathcal{O} \left(\frac{1}{z^{2}} \right), \nonumber\\
%+ \frac{(1-t)(t+a(1-t))}{z^2} + \mathcal{O} \left(\frac{1}{z^{3}} \right), \nonumber\\
\zeta_{2} (z) & = \frac{1}{t
(1-t)}-\frac{\sqrt{a}}{tz^{1/2}}-\frac{1}{2z} + \mathcal{O}
\left(\frac{1}{z^{3/2}} \right), \label{zeta}\\
%-\frac{t+4a(1-t)}{8\sqrt{a}z^{3/2}}  - \frac{(1-t)(t+a(1-t))}{2z^2}+\mathcal{O} \left(\frac{1}{z^{5/2}} \right), \label{zeta}\\
\zeta_{3} (z) & = \frac{1}{t
(1-t)}+\frac{\sqrt{a}}{tz^{1/2}}-\frac{1}{2z}+ \mathcal{O}
\left(\frac{1}{z^{3/2}} \right). \nonumber
%+\frac{t+4a(1-t)}{8\sqrt{a}z^{3/2}}-\frac{(1-t)(t+a(1-t))}{2z^2} +\mathcal{O} \left(\frac{1}{z^{5/2}} \right).\nonumber
\end{align}
At critical time
$t=t^{*}$, we have $k=0$ and equation \eqref{RSequation1} reduces to
\begin{equation} \label{RSequation3}
    z = \frac{1}{\zeta(1- c^* \zeta)^2}, \qquad c^* = t^*(1-t^*).
    \end{equation}
Then, the corresponding Riemann surface has two real branch points, $0$
and $q^{*} = 27c^*/4 >0$, $0$ being degenerate (of order 2), and 
$q^{*}$ being simple. The point at infinity is also a simple branch point
of the Riemann surface. 

%%%%%%%%%%%%%%%%%%%%%%%%%%%%%%%%%%%%%%%%%%%%%%%%%%%%%%%%%%
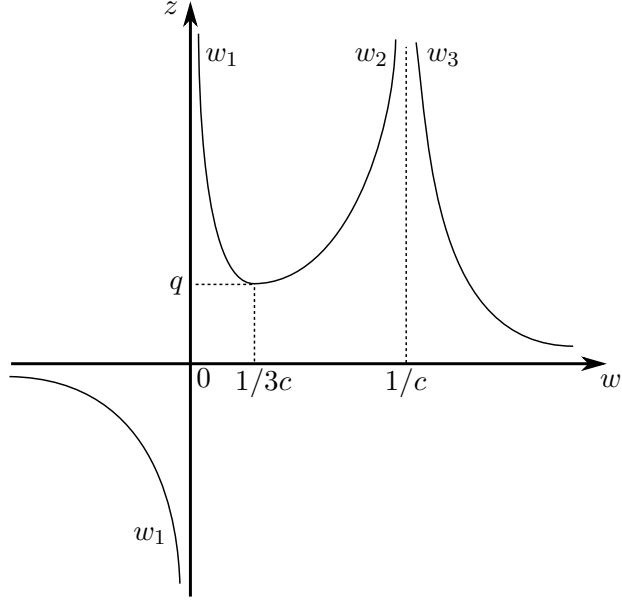
\begin{figure}[t]
\centering 
\begin{overpic}[scale=0.7]%
{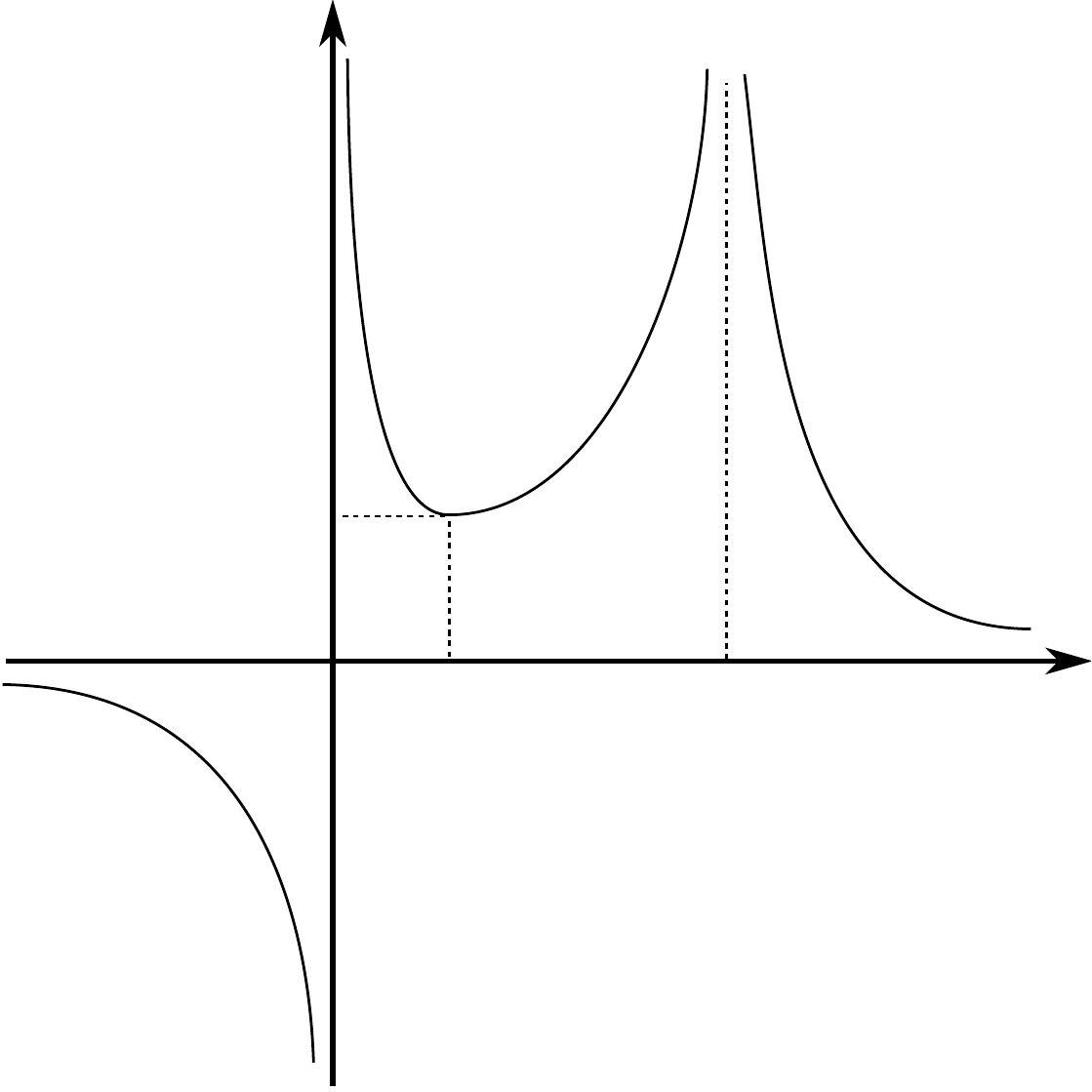}%
  \put(33,90){$w_{1} $}
   \put(21,10){$w_{1} $}
  \put(58,90){$w_{2} $}
   \put(70,90){$w_{3} $}
   \put(31.5,35.5){$0 $}
   \put(27,52){$q $} 
   \put(63,35){$1/c $}
    \put(38,35){$1/3c $}
     \put(99,35.5){$w $}
     \put(26,98){$z $}
 \end{overpic}
\caption{Plot of function $z=z(w)$ given by \eqref{eqmodif} for $w\in \R$.}
\label{fig:plot}
\end{figure}
%%%%%%%%%%%%%%%%%%%%%%%%%%%%%%%%%%%%

In the present paper, we want to work with a Riemann surface $\RR$ with
a double branch point, even if $t \neq t^*$. Following the approach
of \cite{BK4} we do not consider \eqref{RSequation1} if $t \neq t^*$
but instead consider a modified equation
\begin{equation}\label{eqmodif}
	z=\frac{1}{w (1-cw)^{2}},
\end{equation}
with $c$ some positive number and $w$ a new auxiliary variable. 

%%%%%%%%%%%%%%%%%%%%%%%%%%%%%%%%%%%%%%%%%%%%%%%%%%%%%%%%%%
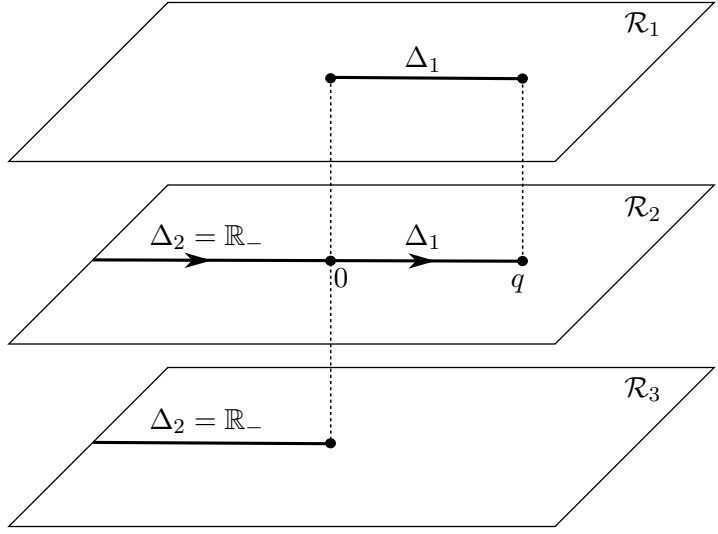
\begin{figure}[t]
\centering 
\begin{overpic}[scale=0.73]%
{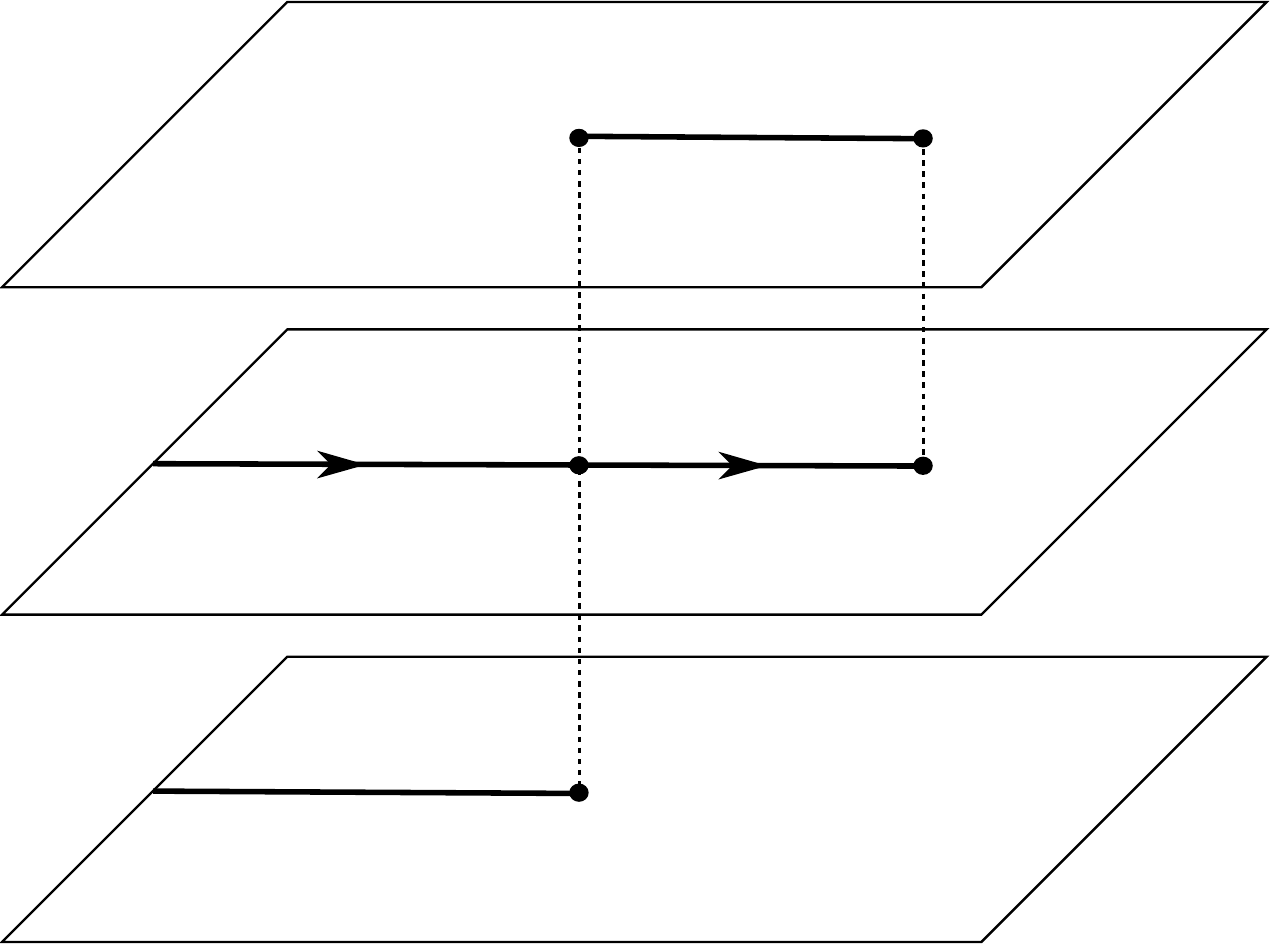}%
  \put(87,70){$\mathcal R_{1} $}
   \put(87,44){$\mathcal R_{2} $}
    \put(87,18.5){$\mathcal R_{3} $}
    \put(46,34){$0 $}
     \put(71,34){$q $}
      \put(20,40){$\Delta_{2}=\R_{-} $}
       \put(56,40){$\Delta_{1} $}
       \put(20,14){$\Delta_{2}=\R_{-} $}
       \put(56,65){$\Delta_{1} $} 
\end{overpic}
\caption{The Riemann surface $\RR$ used in the steepest descent  analysis. 
The origin is a double branch point for every $t$.}
\label{Riemann1}
\end{figure}
%%%%%%%%%%%%%%%%%%%%%%%%%%%%%%%%%%%%

The Riemann surface has simple branch points at 
\begin{equation} \label{qdef} 
	q = 27c/4 
\end{equation}
 and at infinity, and it has a double branch point
at $0$. The sheet structure of $\RR$ can be readily visualized from Figure~\ref{fig:plot} 
and is shown in Figure~\ref{Riemann1}.

The sheets $\RR_{1}$ and $\RR_{2}$ are glued together along the cut
$\Delta_{1}= [0,q]$ and the sheets $\RR_{3}$ and $\RR_{2}$ are
glued together along the cut 
$\Delta_{2}= (0,\infty]$ in the usual crosswise manner. 
There are three inverse functions $w_{k}$, $k=1,2,3$, that behave as
$z\to\infty$ as
\begin{align} %\label{asymptoticsofw}
w_{1} (z) & = \frac{1}{z}+ \mathcal{O} \left(\frac{1}{z^{2}} 
\right), 
\nonumber\\
w_{2} (z) & = \frac{1}{c}-\frac{1}{\sqrt{c}z^{1/2}}-\frac{1}{2z} 
+\mathcal{O} \left(\frac{1}{z^{3/2}} \right), \label{behaviorw}\\
w_{3} (z) & = \frac{1}{c}+\frac{1}{\sqrt{c}z^{1/2}}-\frac{1}{2z}
+\mathcal{O} \left(\frac{1}{z^{3/2}} \right),\nonumber
\end{align}
and which are defined and analytic on $\C\setminus\Delta_{1}$, $\C
\setminus (\Delta_{1}\cup\Delta_{2})$ and $\C \setminus \Delta_{2}$
respectively.

The algebraic function $w=w(z)$ in \eqref{eqmodif} gives a bijection between the 
Riemann surface $\mathcal R$ and the extended complex $w$-plane. Figure~\ref{fig:mapping} 
represents this mapping, along with the domains $\widetilde{\mathcal R}_j = w_j(\mathcal R_j$), 
$j=1,2,3$ (the images of the corresponding sheets of $\mathcal R$) and the points
\[
	w_{q}=w_{2} (q)=\frac{1}{3c}>0,\qquad w_{\infty}=w_{2} (\infty)=\frac{1}{c}>0.
\]
We point out that $w_{2+}(\Delta_{j})$, $j=1,2$, are analytic arcs that extend to 
infinity in the upper half plane, while $w_{2-}(\Delta_{j})$, $j=1,2$, are in the lower half plane.

\subsection{Modified $\zeta$ functions}

We next define modified $\zeta$-functions with the same asymptotic
behavior as $z\to\infty$ as given in \eqref{zeta} up to order $\O(z^{-3/2})$. 

\begin{definition}
For $k=1,2,3$ we define 
\begin{equation} \label{mod-zeta}
	\zeta_{k}=w_{k}+pw_{k}^{2}, 
\end{equation}
where $w_k$ is given by \eqref{behaviorw} and
\begin{equation}\label{defcp}
c=\frac{(1-t)^{2}}{4}\left(\frac{\sqrt{a}}{2}+\sqrt{\frac{a}{4}+\frac{2t}{1-t}}
\right)^{2} \quad \text{and}\quad p=\frac{c^{2}}{t (1-t)}-c.
\end{equation}
\end{definition}

Note that for $t=t^{*}$, the critical time, we have $c=c^{*}=t^{*}(1-t^{*})$ 
and $p=0$, so that we recover in this case the equation
\eqref{RSequation3} and the  $\zeta$-functions defined in \eqref{zeta}.

\begin{lemma}\label{behavior-modified}
For $c$ and $p$ given by \eqref{defcp}, the asymptotic behavior of functions $\zeta_{k}$, 
$k=1, 2, 3$, defined in \eqref{mod-zeta}, as $z\to \infty$, is given by \eqref{zeta}.  
\end{lemma}
\begin{proof}
This follows from direct computations using \eqref{behaviorw} and \eqref{mod-zeta}. 
\end{proof}

%%%%%%%%%%%%%%%%%%%%%%%%%%%%%%%%%%%%%%%%%%%%%%%%%%%%%%%%%%
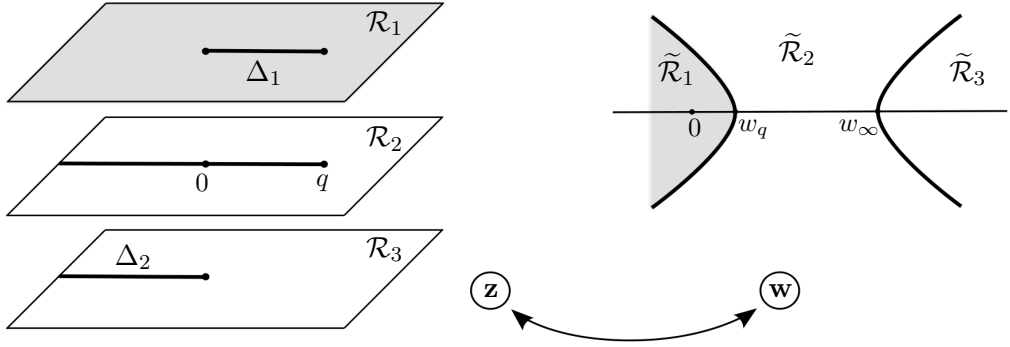
\begin{figure}[t]
\centering 
\begin{overpic}[scale=0.9]%
{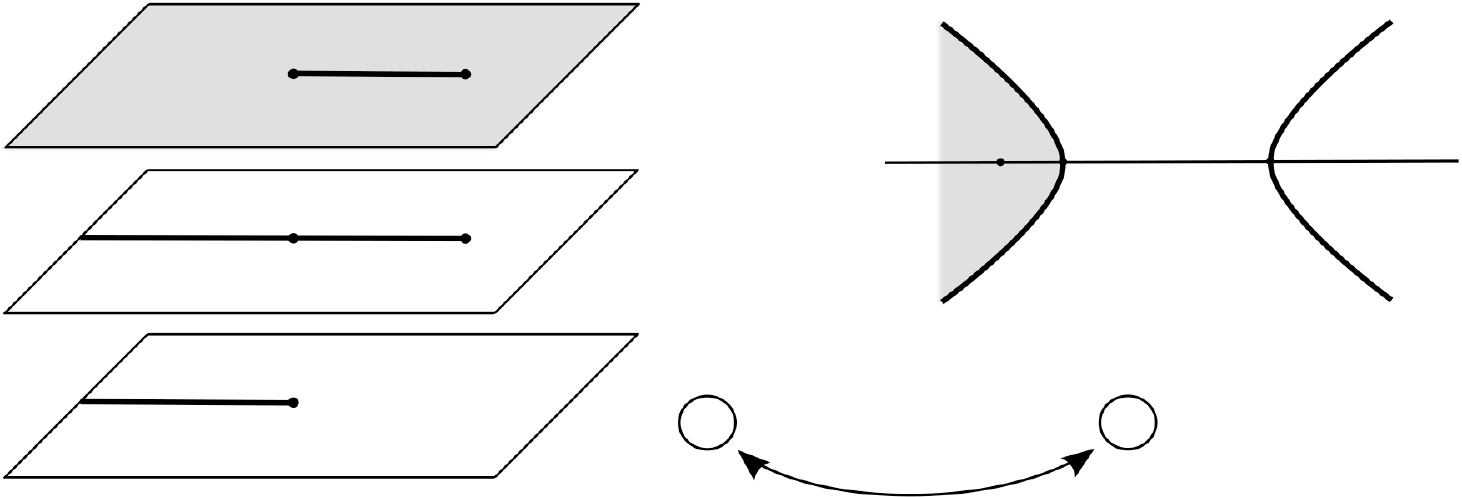}%
  \put(36,31){$\mathcal R_{1} $}
   \put(36,19.5){$\mathcal R_{2} $}
    \put(36,8.5){$\mathcal R_{3} $}
    \put(19,15){\small $0 $}
     \put(31,15.5){\small $q $}
      \put(11,7.7){$\Delta_{2} $} 
       \put(24,26){$\Delta_{1} $} 
       \put(47.7,4.5){$\mathbf z $}    
        \put(76,4.5){$\mathbf w $}  
        \put(68,20.5){\small $0 $}    
        \put(73,21){\small $w_{q} $}
        \put(83,21){\small $w_{\infty} $}
        \put(65,26){$\widetilde{\mathcal R}_{1} $}
   \put(77,28){$\widetilde{\mathcal R}_{2} $}
    \put(94,26){$\widetilde{\mathcal R}_{3} $}
       \end{overpic}
\caption{Bijection between the Riemann surface $\mathcal R$ and the extended $w$-plane.}
\label{fig:mapping}
\end{figure}
%%%%%%%%%%%%%%%%%%%%%%%%%%%%%%%%%%%%

In what follows we need the behavior of the $\zeta$-functions \eqref{mod-zeta} 
near the origin. The following lemmas are analogous to Lemmas 3.2 and 3.3 in \cite{BK4}.
We put $\omega=e^{2\pi i/3}$, as before.
\begin{lemma} \label{wat0}
There exist analytic functions $f_{1}$ and $g_{1}$ defined in a
neighborhood $U_{1}$ of the origin such that with  $z\in U_1$ and $k=1,2,3$,
\begin{equation}\label{expwat0}
w_{k} (z)=%\left\{
\begin{cases}%{ll}
\omega^{2k}z^{-1/3}f_{1} (z)+\omega^{k}z^{1/3}g_{1}
(z)+\frac{2}{3c}\qquad \text{for }\Im z>0,\\
\omega^{k}z^{-1/3}f_{1} (z)+\omega^{2k}z^{1/3}g_{1}
(z)+\frac{2}{3c}\qquad \text{for }\Im z<0,
\end{cases} %\right.
\end{equation}
In addition, we have $f_{1} (0)=c^{-2/3}$, and $f_{1} (z)$ and $g_{1}
(z)$ are real for real $z\in U_1$.
\end{lemma}

\begin{proof}
We put $z=x^{3}$ and $w= (xy)^{-1}$ in \eqref{eqmodif} and obtain
\begin{equation}\label{eqxy}
(xy-c)^{2}-y^{3}=0.
\end{equation}
It has a solution $y=y (x)$ that is analytic in a neighborhood $U_{1}$
of $0$ and $y (0)=c^{2/3}>0$, $y' (0)=-\frac{2}{3}c^{1/3}<0$. Then, we can write
\begin{equation}\label{eqxw}
xw (x)=1/y (x)=f_{1} (x^{3})+x^{2}g_{1} (x^{3})+xh_{1} (x^{3}),
\end{equation}
with
$f_{1}$, $g_{1}$ and $h_{1}$ analytic in $U_{1}$ and $f_{1}
(0)=c^{-2/3}>0$.
Plugging this into \eqref{eqmodif}, we find after some calculations that
\begin{equation}\label{relatf1g1}
f_{1} (z)g_{1} (z)=\frac{1}{9c^{2}},\quad c^{2}f_{1}^{3}
(z)-1+\frac{2}{27c}z+ c^{2}g_{1}^{3} (z) z^{2}=0,
\end{equation}
and $h_{1} (z)=2/(3c)$. Hence, there is a solution $w=w (z)$ of
\eqref{eqmodif} with
\[
w (z)=z^{-1/3}f_{1} (z)+z^{1/3}g_{1} (z)+\frac{2}{3c},\qquad
\text{for }z\in U_{1}\setminus(-\infty,0],
\]
where we take the principal branches for the fractional powers. This
solution is real for $z$ real and positive, thus it coincides with
$w_{3} (z)$, which proves \eqref{expwat0} for 
$k=3$. Considering the two others solutions of \eqref{eqxy}, we get
the expressions \eqref{expwat0} for
$k=1,2$ by analytic continuation. Since $y (x)$ is real for real $x$,
\eqref{eqxw} implies that $f_{1} (z)$ and $g_{1} (z)$ are also real
when $z$ is real.
\end{proof}

The next lemma describes the behavior of the $\zeta$-functions at the origin.
\begin{lemma}\label{zetaat0}
There exist analytic functions $f_{2}$ and $g_{2}$ defined in a
neighborhood $U_{2}$ of the origin such that for $z\in U_2$ and
$k=1,2,3$,
\begin{equation}\label{expzetaat0}
\zeta_{k} (z)=%\left\{
\begin{cases}%{ll}
\omega^{2k}z^{-1/3}f_{2} (z)+\omega^{k}z^{-2/3}g_{2}
(z)+\frac{2}{3t (1-t)}\qquad \text{for }\Im z>0,\\
\omega^{k}z^{-1/3}f_{2} (z)+\omega^{2k}z^{-2/3}g_{2}
(z)+\frac{2}{3t (1-t)}\qquad \text{for }\Im z<0.
\end{cases} %\right.
\end{equation}
In addition, we have 
\begin{equation}\label{f2g2at0}
f_{2} (0)=c^{-2/3}\left(1+\frac{4p}{3c} \right),\qquad g_{2} (0)=pc^{-4/3},
\end{equation}
and $f_{2} (z)$ and $g_{2}
(z)$ are real for real $z\in U_2$.
\end{lemma}
\begin{proof}
The proof follows from \eqref{mod-zeta} and the previous lemma. It
suffices to compute $w_{k}^{2} (z)$ by using \eqref{expwat0} and the
first identity in
\eqref{relatf1g1}. Then, \eqref{expzetaat0} follows from
\eqref{mod-zeta} if we set
\begin{equation}\label{f12g12}
f_{2} (z)=\left(1+\frac{4p}{3c} \right)f_{1} (z)+pzg_{1}^{2} (z),\qquad 
g_{2} (z)=\left(1+\frac{4p}{3c} \right)zg_{1} (z)+pf_{1}^{2} (z),
\end{equation}
and \eqref{f2g2at0} follows from the value of $f_{1}
(0)$ given in Lemma \ref{wat0}. The functions $f_{2} (z)$ and $g_{2}
(z)$ are real for real $z\in U_2$ since 
$f_{1} (z)$ and $g_{1} 
(z)$ are real for real $z\in U_1$.
\end{proof}

\subsection{The $\lambda$-functions} \label{section4}

We next define the $\lambda$-functions as anti-derivatives of the modified $\zeta$-functions
\eqref{mod-zeta}.

\begin{definition}
We define for $k=1,2,3$,
\begin{equation}\label{deflambda}
\lambda_{k} (z)=\int_{0_{+}}^{z}\zeta_{k} (s)ds,
\end{equation}
where the path of integration starts at $0$ on the upper half-plane (which is denoted by $0_{+}$)
and is contained in $\C \setminus (-\infty,q]$ for $k=1,2$, and
in $\C \setminus (-\infty,0]$ for $k=3$. 
\end{definition}

By construction the functions
$\lambda_{1}$ and $\lambda_{2}$ are analytic in $\C \setminus
(-\infty,q]$ while $\lambda_{3}$ is analytic in $\C \setminus
(-\infty,0]$. From Lemma \ref{behavior-modified} and
\eqref{deflambda}, it follows 
that, as $z\to\infty$,
\begin{align}
\lambda_{1} (z) & = \log{z}+ \ell_{1} + \mathcal{O} \left(\frac{1}{z} \right), \nonumber\\
%+ \frac{(1-t)(t+a(1-t))}{z^2} + \mathcal{O} \left(\frac{1}{z^{3}} \right), \nonumber\\
\lambda_{2} (z) & = \frac{z}{t
(1-t)}-\frac{2\sqrt{a}z^{1/2}}{t}-\frac{1}{2}\log z +\ell_{2}+ \mathcal{O}
\left(\frac{1}{z^{1/2}} \right), \nonumber\\
\lambda_{3} (z) & = \frac{z}{t
(1-t)}+\frac{2\sqrt{a}z^{1/2}}{t}-\frac{1}{2}\log z+\ell_{3}+ \mathcal{O}
\left(\frac{1}{z^{1/2}} \right), \nonumber
\end{align}
where $\ell_{k}$, $k=1,2,3$, are certain constants of integration, and $\log z$ is the principal branch
of the logarithm which is real on $(0,+\infty)$. 
Using the structure of the Riemann surface $\mathcal R$ and the residue calculation based on the expansion of $\zeta_{1}$ at infinity,
see \eqref{zeta}, we conclude that
\begin{equation}\label{valueLambdaAt0}
\lambda_{1-} (0)=-2\pi i, \quad \lambda_{2-} (0)=2\pi i, \quad  \lambda_{3-} (0)=0.
\end{equation}
Moreover, the following jump relations hold true:
\begin{equation}
\begin{aligned}\label{jumplamb}
	\lambda_{1\pm} (x) & =\lambda_{2\mp} (x)-2\pi i, \quad \lambda_{3+} (x) =\lambda_{3-} (x),  
	& & x\in \Delta_{1}= (0,q), \\
 \lambda_{1+}(x) & = \lambda_{1-} (x) +2\pi i,  & & x\in \Delta_{2}=(-\infty,0),\\
	\lambda_{2+}(x) & =\lambda_{3-}(x),  \quad
\lambda_{2-}(x)=\lambda_{3+}(x)+2\pi i, & & x\in\Delta_{2}.
\end{aligned}
\end{equation}

The behavior of the $\lambda$-functions at the
origin follows from Lemma \ref{zetaat0} and \eqref{deflambda}.

\begin{lemma} \label{lambdaat0}
There exist analytic functions $f_{3}$ and $g_{3}$ defined in a
neighborhood $U_{3}$ of the origin such that for $z\in U_3$ and $k=1,2,3$,
\begin{equation}\label{explambdaat0}
	\lambda_{k} (z)=%\left\{
	\begin{cases}%{ll}
	\frac{3}{2}\omega^{2k}z^{2/3}f_{3} (z)+\omega^{k}z^{1/3}g_{3}(z)+\frac{2z}{3t (1-t)}\quad & \text{for }\Im z>0,\\
	\lambda_{k-} (0) +\frac{3}{2}\omega^{k}z^{2/3}f_{3}(z)+\omega^{2k}z^{1/3}g_{3} 
(z)+\frac{2z}{3t (1-t)}\quad & \text{for }\Im z<0,
\end{cases} %\right.
\end{equation}
with $\lambda_{k-} (0)$ given by \eqref{valueLambdaAt0}. In addition, we have 
\begin{equation}\label{f3g3at0}
	f_{3} (0)= f_{2} (0)=c^{-2/3}\left(1+\frac{4p}{3c} \right),\qquad
	g_{3} (0)= 3g_{2} (0)=3pc^{-4/3},
\end{equation}
and both $f_{3} (z)$ and $g_{3}(z)$ are real for real $z\in U_3$.
\end{lemma}
\begin{proof}
Integrating \eqref{expzetaat0} and taking into account that $\lambda_{k+}(0)=0$, we 
get \eqref{explambdaat0} and \eqref{f3g3at0}. The fact that $f_{3} (z)$ and $g_{3}
(z)$ are real for real $z\in U_3$ also follows from Lemma \ref{zetaat0}.
\end{proof}

The functions $f_3$ and $g_3$ depend on $t$. We write
$ f_3(z; t)$ and $g_3(z; t)$ if we want to emphasize the dependence on $t$.

%%%%%%%%%%%%%%%%%%%%%%%%%%%%%%%%
\begin{figure}[t]
\hspace{1.2cm} {\includegraphics[scale=0.8]{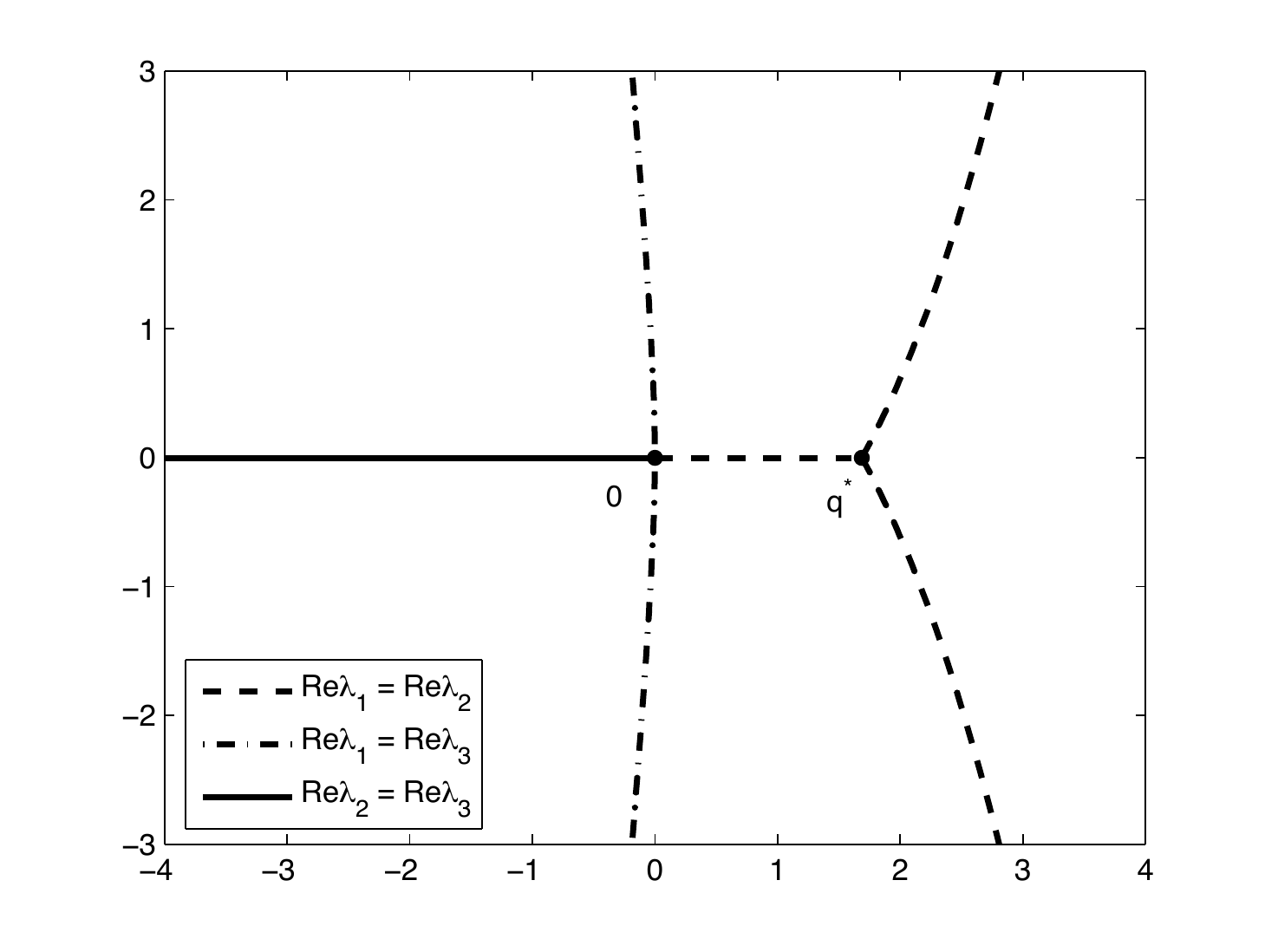}}
\caption{Curves $\Re \lambda_{j}=\Re \lambda_{k}$ at the critical time
(here $a=1$ and $t=t^{*}=0.5$).}
\label{lambda-at-crit}
\end{figure}
%%%%%%%%%%%%%%%%%%%%%%%%%%%%%%%%

In order to control the jumps in the different RH problems that we are going to
consider in the sequel, we need to compare the real parts of the 
$\lambda$-functions. Figure~\ref{lambda-at-crit} shows, at the critical time
$t=t^{*}$, the curves in the complex plane where the 
real parts of the $\lambda$-functions are equal. These curves are
critical trajectories of the quadratic differentials 
$(\zeta_{j} (z)-\zeta_{k} (z))^{2}dz^{2}$. The curve $\Re\lambda_{2} (z)=\Re
\lambda_{3} (z)$ consists of the negative real axis. The curve
$\Re\lambda_{1} (z)=\Re \lambda_{3} (z)$ consists of two trajectories emanating from the
origin, symmetric with respect to the real axis, and going to
infinity. Finally, the curve $\Re\lambda_{1} (z)=\Re
\lambda_{2} (z)$ consists of $\Delta_{1}$ along with two symmetric
trajectories emanating from the branch point $q^{*}$.

At a non-critical
time $t\neq t^{*}$, the configuration of the curves remains the same,
except in a small neighborhood of the origin. Figures~\ref{lambda-bef-crit}
and \ref{lambda-after-crit} show the new configurations in such a
neighborhood. When $t<t^{*}$, the function $\zeta_{1} 
(z)-\zeta_{2}(z)$ has an additional zero on $\Delta_{1}$, which causes
the appearance of a new loop around the origin where $\Re \lambda_{1}
(z)=\Re \lambda_{2} (z)$. Also, the curve $\Re\lambda_{1} (z)=\Re
\lambda_{3} (z)$ is shifted to the left and becomes a continuation of
the loop as it intersects the negative real axis. Similarly, when
$t>t^{*}$, the function $\zeta_{2} 
(z)-\zeta_{3}(z)$ has an additional zero on $\Delta_{2}$, which causes
the appearance of a new loop around the origin where $\Re \lambda_{2}
(z)=\Re \lambda_{3} (z)$. Now, the curve $\Re\lambda_{1} (z)=\Re
\lambda_{3} (z)$ is shifted to the right and becomes a continuation of
the loop as it intersects the positive real axis. In both cases, when
$t$ tends to $t^{*}$, the loop shrinks to the origin.
%%%%%%%%%%%%%%%%%%%%%%%%%%%%%%
\begin{figure}[t]
\hspace{1.2cm} {\includegraphics[scale=0.8]{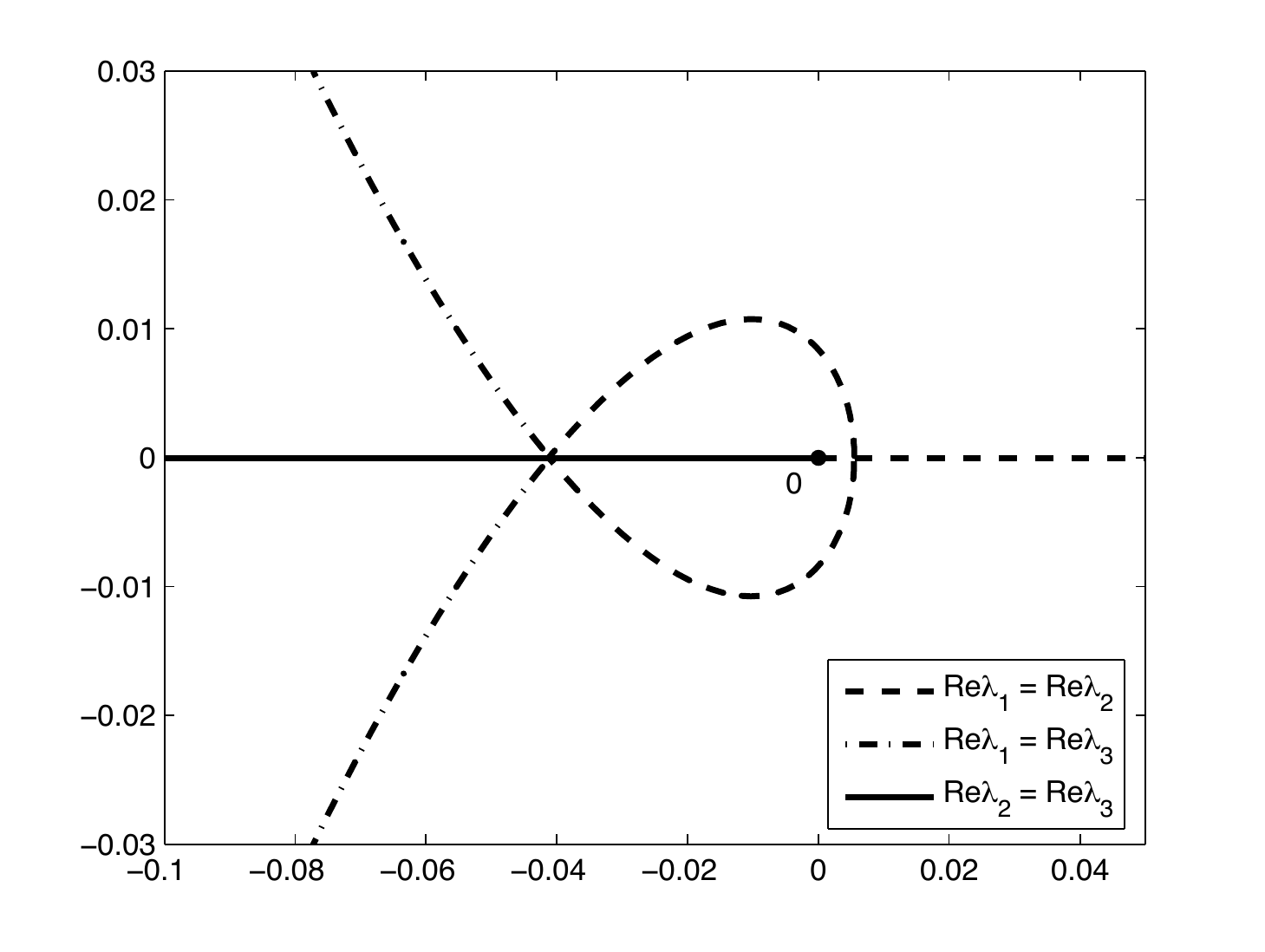}}
\caption{Curves $\Re \lambda_{j}=\Re \lambda_{k}$ near the origin, before the critical time (here $a=1$, $t=0.3<t^{*}=0.5$).}
\label{lambda-bef-crit}
\end{figure}
%%%%%%%%%%%%%%%%%%%%%%%%%%%%%%%
\begin{figure}[t]
\hspace{1.2cm} {\includegraphics[scale=0.8]{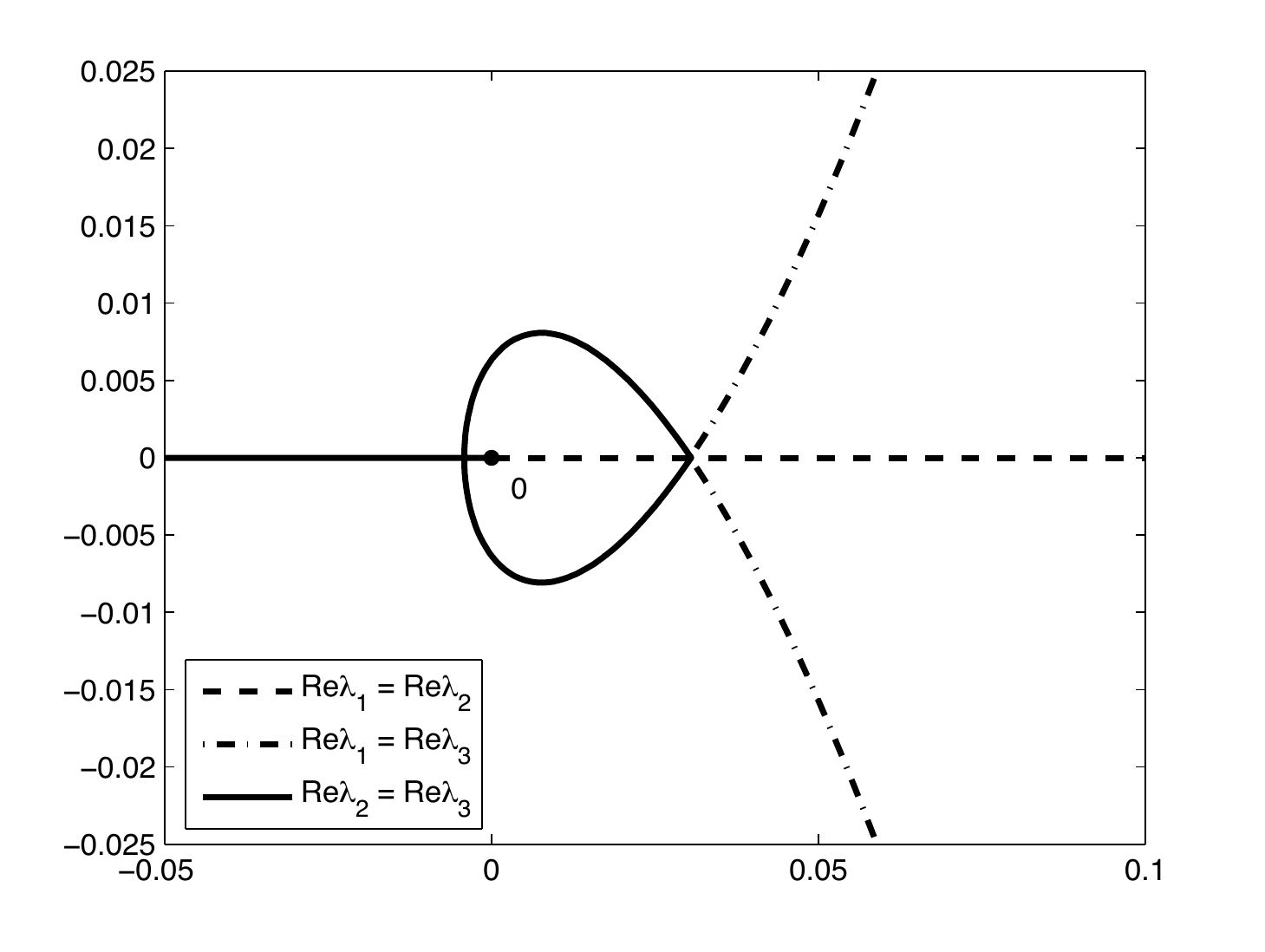}}
\caption{Curves $\Re \lambda_{j}=\Re \lambda_{k}$ near the origin, after the critical time (here $a=1$,  
$t=0.7>t^{*}=0.5$).}
\label{lambda-after-crit}
\end{figure}
%%%%%%%%%%%%%%%%%%%%%%%%%%%%%%%

Concerning the relative ordering of the real parts in a neighborhood
of the real axis, 
the following lemma holds true.
\begin{lemma}\label{lambda-sign}
\begin{enumerate}
\item[\rm (a)] For $z\in(q,+\infty)$, we have $\Re \lambda_{1} (z) <\Re \lambda_{2} (z)$.
\item[\rm (b)] The open interval $(0,q)$ has a neighborhood $U_{1}$ in the complex plane 
such that for $z\in U_{1}\setminus (0,q)$ and outside of the additional loop around $0$
when $t\neq t^{*}$, we have
$\Re \lambda_{2} (z) < \Re \lambda_{1} (z)$.
\item[\rm (c)] The open interval $(-\infty,0)$ has a neighborhood
$U_{2}$ in the complex plane such that for $z\in
U_{2}\setminus (-\infty,0)$ and outside of the additional loop around $0$
when $t\neq t^{*}$, we have
%\begin{equation*} 
$\Re \lambda_{2} (z) < \Re \lambda_{3} (z)$.
%\end{equation*}
The neighborhood $U_2$ is unbounded and contains a full neighborhood
of infinity.
\end{enumerate}
\end{lemma}
\begin{proof}
This is similar to the proof of \cite[Lemma 4.3]{KMW}.
When $t\neq t^{*}$, only the ordering of the
real parts inside
the additional loop is modified.
See also \cite[Lemma 4.2]{BK4}. 
\end{proof}

In the double scaling limit \eqref{doublescaling} that we are going to consider, 
we have that $t-t^{*}$ is of order $n^{-1/2}$ as $n\to\infty$. Then, the 
special ordering of the real parts of the $\lambda$-functions inside the loop
will not cause a problem because the (shrinking) disk around the
origin where we  are going to 
construct the local parametrix will be big enough to
contain the loop for $n$ large, see the proof of Lemma
\ref{lem:JumpsS} below.

\subsection{Second transformation of the RH problem}

The goal of the second transformation $X \mapsto U$ is to normalize the RH problem at
infinity using the functions $\lambda _j$ from Section \ref{section4}.
\begin{definition}
Given $X$ as in \eqref{Xdef} we define
\begin{equation}\label{Udef}
	U(z)= C_2   X(z)
\begin{pmatrix}
e^{-n \lambda_{1} (z)} & 0 & 0\\
0 & e^{-n (\lambda_{2} (z)- \frac{z}{t (1-t)})} & 0\\
0 & 0 & e^{-n \left(\lambda_{3} (z)- \frac{z}{t (1-t)}\right)}
\end{pmatrix},
\end{equation}
where $C_{2}$ is some explicit constant matrix, see \cite[Equation (4.19)]{KMW}.
\end{definition}

As a consequence of the assertion (c) in Lemma \ref{lambda-sign} 
we may (and do) 
assume that the contours $\Delta_2^\pm$,
defined above (and depicted in Figure~\ref{fig:first_transformation})
lie in the neighborhood $U_2$ of $\Delta_2$
where $\Re (\lambda _2-\lambda _3)<0$ (except when it intersects the small
loop near $0$ when $t\neq t^{*}$, see Figures~\ref{lambda-bef-crit} and \ref{lambda-after-crit}).

Using the jump relations \eqref{jumplamb} and other properties of
the $\lambda$ functions we find that $U$ solves the following RH problem, see \cite[Section 4]{KMW}
for details.
\begin{rhproblem} \label{rhpforU}
The matrix-valued function $U(z)$ defined by \eqref{Udef}
is the unique solution of the following RH problem.
\begin{enumerate}
\item $U(z)$ is analytic in $\mathbb{C} \setminus \Sigma_U$ where $\Sigma_U = \Sigma_X = \R \cup \Delta_2^{\pm}$.
\item On $\Sigma_U$ we have the jump
\begin{equation} \label{Ujump}
	U_+(z) = U_-(z) J_U(z), \qquad z \in \Sigma_U
	\end{equation}
with jump matrices $J_U(z)$ given by
\begin{align}
	J_U(x) & = \label{Ujump1}
    \begin{pmatrix}
	1 & 0 & 0\\
	0 & 0 & - |x|^{-\alpha} \\
	0 & |x|^{\alpha}    & 0
    \end{pmatrix},  && x \in\Delta_{2}= (-\infty,0), \\
  J_U(x) & =  \label{Ujump3}
    \begin{pmatrix}
    e^{n (\lambda_{2}-\lambda_{1})_{+} (x)} & x^\alpha & 0 \\
        0 & e^{n (\lambda_{2}-\lambda_{1})_{-} (x)} & 0 \\
        0 & 0                    & 1
    \end{pmatrix}, &&  x\in\Delta_{1}= (0,q), \\
    \label{Ujump4}
	J_U(x) & = 
	\left( I + x^\alpha e^{n (\lambda_{1}-\lambda_{2}) (x)} E_{1 2}\right),
	&&  x\in (q,\infty), \\
	 \label{Ujump5}
	J_U(z) & = \left( I + e^{\pm \alpha \pi i} z^{-\alpha} e^{n(\lambda_2-\lambda_3)(z)} E_{2 3}\right),  
	&&  z \in \Delta_2^{\pm}.
\end{align}
\item As $z\to \infty$ we have
\begin{multline} \label{Uasymptotics}
U(z) =
   \left(I + \mathcal{O}\left(\frac{1}{z}\right) \right)
    \begin{pmatrix}
    1 & 0 & 0\\
    0 & z^{1/4} & 0 \\
    0 & 0 &  z^{-1/4}
    \end{pmatrix}
     \begin{pmatrix}
    1 & 0 & 0 \\
    0 & \frac{1}{\sqrt{2}} & \frac{1}{\sqrt{2}}i \\[2mm]
    0 & \frac{1}{\sqrt{2}}i & \frac{1}{\sqrt{2}}
    \end{pmatrix}
    \begin{pmatrix}
    1 & 0 & 0 \\
    0 & z^{\alpha/2} & 0 \\
    0 & 0 & z^{-\alpha/2}
    \end{pmatrix},
\end{multline}
\item $U(z)$ has the same behavior as $X(z)$ at
the origin, see \eqref{Yedge} and \eqref{Xedge2}.
\end{enumerate}
\end{rhproblem}

It follows from Lemma \ref{lambda-sign} that in the double scaling limit \eqref{doublescaling}
the jump matrices in \eqref{Ujump4} and \eqref{Ujump5}
tend to the  identity matrix $I$ as $n \to \infty$ at an exponential
rate, except when $z$ lies inside the small loop around the origin, 
see Figures~\ref{lambda-bef-crit} and \ref{lambda-after-crit}.
Moreover, $(\lambda_{2}-\lambda_{1})_{+} =-
(\lambda_{2}-\lambda_{1})_{-}-4\pi i$ is purely imaginary on
$\Delta_{1}$, so that the first two diagonal elements of the jump matrices in
\eqref{Ujump3} are highly oscillatory if $n$ is large.

\section{Third transformation of the RH problem} \label{sec6}

The third transformation $U\mapsto T$ consists in opening a lens around
$\Delta_{1}$, see Figure~\ref{deform3}. It transforms the oscillatory
entries in the jump matrix on $\Delta_{1}$ into exponentially small
off-diagonal terms. Following \cite[Section 5]{KMW}, we define $T$ as follows.

\begin{definition}
We define
\begin{equation} \label{T1def}
	T(z)=U(z) \left( I \mp z^{-\alpha}e^{n (\lambda_{2}-\lambda_{1}) (z)} E_{21} \right),
\end{equation}
for $z$ in the domain bounded by $\Delta_{1}^{\pm}$ and $\Delta_1$ 
(the shaded region in Figure~\ref{deform3}), and we define
\begin{equation} \label{T2def}
    T(z) =U(z) 
    \end{equation}
for $z$  outside of the lens around $\Delta_1$.  
\end{definition}

%%%%%%%%%%%%%%%%%%%%%%%%%%%%%%%%%%%%%%%%%%%%%%%%%%%%%%%%%%
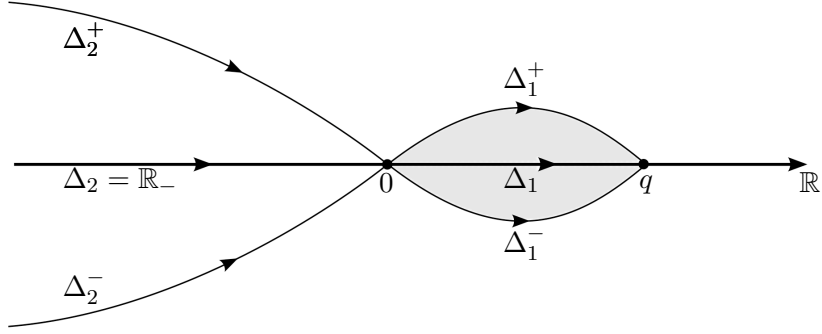
\begin{figure}[t]
\centering 
\begin{overpic}[scale=0.7]%
{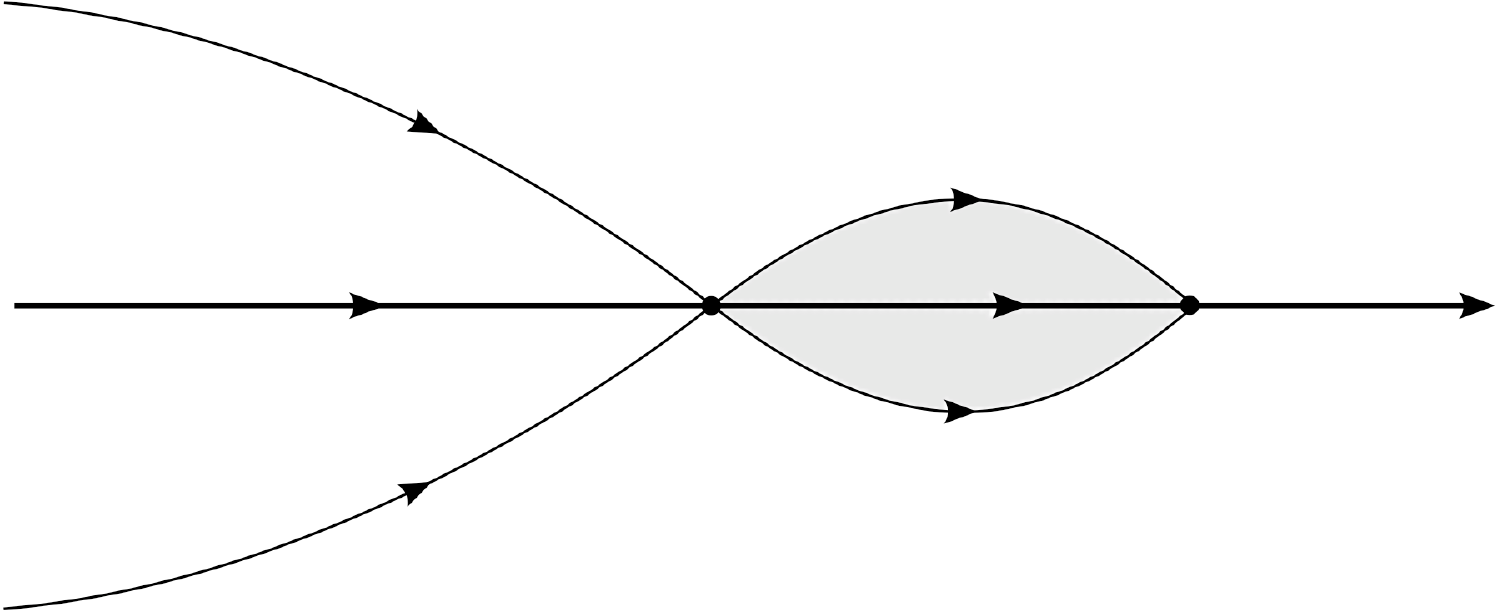}%
  \put(7,35){$\Delta_{2}^+ $}
  \put(7,35){$\Delta_{2}^+ $}
  \put(7,4){$\Delta_{2}^- $}
  \put(62,30){$\Delta_{1}^+ $}
  \put(62,9.5){$\Delta_{1}^- $}
   \put(7,17.5){$\Delta_{2}=\R_{-} $}
  \put(62,17.5){$\Delta_{1}$}
   \put(46.5,17){$0$}
    \put(79,17.5){$q$}
     \put(99,17){$\R$}
  \end{overpic}
\caption{Opening of lens around $\Delta_{1}$. The contour $\Sigma_T = \mathbb R \cup \Delta_1^{\pm} \cup \Delta_2^{\pm}$ 
is the jump contour in the RH problem for $T$}
\label{deform3}
\end{figure}
%%%%%%%%%%%%%%%%%%%%%%%%%%%%%%%%%%%%%%%%%%%%%%%%%%%%%%%%%%

Let $\Sigma_{T} = \mathbb R \cup \Delta_1^{\pm} \cup \Delta_2^{\pm}$ be the 
union of the contours depicted in Figure~\ref{deform3}. 
Then, straightforward calculations yield that the matrix-valued function $T$ is the solution
of the following RH problem.

\begin{rhproblem} \label{rhpforT}
The matrix-valued function $T(z)$ satisfies
\begin{enumerate}
\item $T$ is analytic in $\mathbb C \setminus \Sigma_{T}$.
\item On $\Sigma_T$ we have
\begin{equation} \label{Tjump}
	T_+ = T_- J_T 
	\end{equation}
	where
{\allowdisplaybreaks
\begin{align}
	J_T(x) & =  \begin{pmatrix}
    1 & 0                  & 0 \\
	0 & 0 & -|x|^{-\alpha}\\
	0 & |x|^{\alpha} & 0
    \end{pmatrix}, &&  x\in\Delta_{2}, \label{JT1}\\
	J_T(z) & = 
			I + e^{\pm \alpha \pi i} z^{-\alpha}e^{n (\lambda_{2}-\lambda_{3})(z)} E_{23},
		&& z\in\Delta_{2}^{\pm}, \label{JT2} \\
	J_T(x) & =    \begin{pmatrix}
	0 &  x^{\alpha}          & 0    \\
	-x^{-\alpha} & 0 & 0 \\
	0 &  0 & 1     \end{pmatrix}, 
		&&   x\in\Delta_1,\label{JT3} \\
	J_T(z) & = 
	  I + z^{-\alpha} e^{n (\lambda_{2}-\lambda_{1})(z)} E_{2 1}, 
    && z\in\Delta_{1}^{\pm}, \label{JT4} \\
	J_T(x) & = 
    I + x^\alpha e^{n (\lambda_{1}-\lambda_{2}) (x)} E_{1 2},
    && x\in (q,+\infty).\label{JT5}
\end{align}}
\item As $z \to \infty$, we have
\begin{equation} \label{Tinfty}
    T(z) =
   \left(I + \mathcal{O}\left(\frac{1}{z}\right) \right)
    \begin{pmatrix}
    1 & 0 & 0\\
    0 & z^{1/4} & 0 \\
    0 & 0 &  z^{-1/4}
    \end{pmatrix}
     \begin{pmatrix}
    1 & 0 & 0 \\
    0 & \frac{1}{\sqrt{2}} & \frac{1}{\sqrt{2}}i \\[2mm]
    0 & \frac{1}{\sqrt{2}}i & \frac{1}{\sqrt{2}}
    \end{pmatrix}
    \begin{pmatrix}
    1 & 0 & 0 \\
    0 & z^{\alpha/2} & 0 \\
    0 & 0 & z^{-\alpha/2}
    \end{pmatrix}.
\end{equation}
\item For $-1<\alpha<0$, $T(z)$ behaves near the origin like:
\begin{equation*} %\label{T01}
T(z)=\mathcal{O}\begin{pmatrix}
            1 & |z|^{\alpha} & 1 \\
            1 & |z|^{\alpha} & 1 \\
            1 & |z|^{\alpha} & 1
            \end{pmatrix}, \quad \text{as }z\to 0.
\end{equation*}
For $\alpha=0$, $T(z)$ behaves near the origin like:
\begin{equation*} %\label{T02}
T(z)=\mathcal{O}\begin{pmatrix}
             1 & \log |z| & 1 \\
             1 & \log |z| & 1 \\
             1 & \log |z| & 1
            \end{pmatrix}, \quad \text{as }z\to 0\text{ outside the
lenses around $\Delta_{2}$ and $\Delta_{1}$},
\end{equation*}
and
\begin{equation*} %\label{T02lens}
T(z)=\begin{cases}
 \mathcal{O}\begin{pmatrix}
            1 & \log |z| & \log |z| \\
            1 & \log |z| & \log |z| \\
            1 & \log |z| & \log |z|
            \end{pmatrix}, & \text{as }z\to 0 \text{ inside the lens
around $\Delta_{2}$},\\[5mm]
\mathcal{O}\begin{pmatrix}
            \log |z| & \log |z| & 1 \\
            \log |z| & \log |z| & 1 \\
            \log |z| & \log |z| & 1
            \end{pmatrix}, & \text{as }z\to 0 \text{ inside the lens
around $\Delta_{1}$}.
\end{cases}
\end{equation*}
For $\alpha > 0$, $T(z)$ behaves near the origin like:
\begin{equation*} %\label{T03}
T(z)=\mathcal{O}\begin{pmatrix}
             1 & 1 & 1 \\
             1 & 1 & 1 \\
             1 & 1 & 1
            \end{pmatrix}, \quad \text{as }z\to 0\text{ outside the
lenses around $\Delta_{2}$ and $\Delta_{1}$},
\end{equation*}
and\begin{equation*} %\label{T03bis}
T(z)=\left\{\begin{array}{ll}
\mathcal{O}\begin{pmatrix}
            1 & 1 & |z|^{-\alpha} \\
            1 & 1 & |z|^{-\alpha}\\
            1 & 1 & |z|^{-\alpha}
            \end{pmatrix}, & \text{as }z\to 0 \text{ inside the lens
around $\Delta_{2}$},\\
\mathcal{O}\begin{pmatrix}
             |z|^{-\alpha} & 1 & 1 \\
             |z|^{-\alpha} & 1 & 1 \\
             |z|^{-\alpha} & 1 & 1
            \end{pmatrix}, & \text{as }z\to 0 \text{ inside the lens
around $\Delta_{1}$}.
\end{array}
\right.
\end{equation*}
\item
$T(z)$ remains bounded as $z \to q$.
\end{enumerate}
\end{rhproblem}

\section{Global parametrix} \label{sec:N}

In the next step we ignore the jumps on $\Delta_1^{\pm}$ and $\Delta_2^{\pm}$ in the RH problem for $T$ and
we consider the following RH problem.
\begin{rhproblem} \label{rhpforN}
	Find $N_{\alpha}:\C\setminus (-\infty,q]\mapsto \C^{3\times 3}$ such that
\begin{enumerate}
\item $N_{\alpha}$ is analytic in $\C\setminus(-\infty, q]$.
\item $N_{\alpha}$ has continuous boundary values on $(-\infty,0)$
and $(0,q)$, satisfying the following jump relations:
\begin{align}
\label{eq:Njump1}
N_{\alpha +}(x) & =N_{\alpha -}(x)
    \begin{pmatrix}
    0            & x^\alpha & 0 \\
        -x^{-\alpha} & 0        & 0 \\
        0            & 0        & 1
    \end{pmatrix}, &&  x\in (0,q), \\
\label{eq:Njump2}
N_{\alpha +}(x) & =N_{\alpha -}(x)
\begin{pmatrix}
    1            & 0 & 0 \\
        0 & 0        & -|x|^{-\alpha} \\
        0            & |x|^{\alpha}        & 0
    \end{pmatrix}, && x\in (-\infty,0).
    \end{align}
\item As $z \to \infty$, 
\begin{equation} \label{Ninfty}
    N_{\alpha}(z) =
   \left(I + \mathcal{O}\left(\frac{1}{z}\right) \right)
    \begin{pmatrix}
    1 & 0 & 0\\
    0 & z^{1/4} & 0 \\
    0 & 0 &  z^{-1/4}
    \end{pmatrix}
     \begin{pmatrix}
    1 & 0 & 0 \\
    0 & \frac{1}{\sqrt{2}} & \frac{1}{\sqrt{2}}i \\[2mm]
    0 & \frac{1}{\sqrt{2}}i & \frac{1}{\sqrt{2}}
    \end{pmatrix}
    \begin{pmatrix}
    1 & 0 & 0 \\
    0 & z^{\alpha/2} & 0 \\
    0 & 0 & z^{-\alpha/2}
    \end{pmatrix}.
\end{equation}
\item As $z \to q$ we have
\begin{equation} \label{Nalphaatq} 
	N_{\alpha}(z)=\O \begin{pmatrix}
    |z-q|^{-1/4} & |z-q|^{-1/4} & 1 \\
    |z-q|^{-1/4} & |z-q|^{-1/4} & 1 \\
    |z-q|^{-1/4} & |z-q|^{-1/4} & 1 \end{pmatrix}.
\end{equation}
\item As $z \to 0$ we have 
\begin{equation} \label{Nalphaat0}
	z^{\alpha/3} N_{\alpha}(z) \begin{pmatrix} 1 & 0 & 0 \\ 0 & z^{-\alpha} & 0 \\ 0 & 0 & 1 \end{pmatrix}
		= M_{\alpha}^{\pm} z^{-1/3} + \O(1), \qquad \pm \Im z > 0,
		\end{equation}
		where $M_{\alpha}^{\pm}$ is a rank one matrix.
\end{enumerate}
\end{rhproblem}

This RH problem is solved as in \cite[Section 6]{KMW} in terms of
the branches $w_{k}$, $k=1,2,3$, of the algebraic function $w$, defined by \eqref{eqmodif},
\eqref{behaviorw}. 

\begin{proposition}
\begin{enumerate}
\item[\rm (a)] 
The solution of the RH problem \ref{rhpforN} for $\alpha = 0$ is given by
\begin{equation} \label{N0def}
 N_0(z)  =\begin{pmatrix}
    F_1(w_1(z)) & F_1(w_2(z))  & F_1(w_3(z)) \\
        F_2(w_1(z)) & F_2(w_2(z))  & F_2(w_3(z)) \\
        F_3(w_1(z)) & F_3(w_2(z))  & F_3(w_3(z))
    \end{pmatrix},
\end{equation}
where 
\begin{equation}  \label{eq:F1F2F3}
    F_1(w) = K_1 \frac{(w - w_{\infty})^{2}}{D(w)^{1/2}}, \quad
        F_2(w) = K_{2} \frac{w(w-w^*)}{D(w)^{1/2}}, \quad
        F_3(w) =  K_{3}  \frac{w(w - w_{\infty})}{D(w)^{1/2}}.
\end{equation}
with $w^* \neq w_{\infty}$, $K_1$, $K_2$, $K_3$ certain non-zero constants that depend on $a$ and $t$,
and
\begin{equation} \label{defD}
 	D(w) = (w-w_{q})(w-w_{\infty}), 
\end{equation}
The square root $D (w)^{1/2}$ in \eqref{eq:F1F2F3} is defined with a cut along
$w_{2-} (\Delta_1)\cup w_{2-}(\Delta_2)$, such that it is positive for real $w > w_{\infty}$.
\item[\rm (b)] The solution of the RH problem \ref{rhpforN} for general $\alpha$ is given by
\begin{equation} \label{Nalpha} 
    N_{\alpha}(z) = C_{\alpha} N_0(z) \begin{pmatrix}
    e^{\alpha G_1(z)} & 0 & 0 \\
    0 & e^{\alpha G_2(z)} & 0 \\
    0 & 0 & e^{\alpha G_3(z)}
    \end{pmatrix},  
    \end{equation}
    where $N_0(z)$ is given by \eqref{N0def}, $C_{\alpha}$ is a constant matrix
that depends on $a$, $t$ and $\alpha$, such that $\det C_\alpha=1$ (see \cite[Equation (6.14)]{KMW} for details),
and the functions $G_j(z)$ are given by
\begin{equation} \label{G1G2G3def}
    G_j(z) = r_j(w_j(z)), \quad  z \in \mathcal R_j, \quad j=1,2,3,
    \end{equation}
with 
\begin{equation} \label{r123def}
    \begin{aligned}
    r_1(w) & = \log(1-cw), & w \in \widetilde{\mathcal R}_1, \\
    r_2(w) & = - \log w - \log(1-cw), & w \in \widetilde{\mathcal R}_2, \\
    r_3(w) & = \log(1-cw) + i \pi, & w \in \widetilde{\mathcal R}_3.
    \end{aligned}
\end{equation}
The branches of the logarithms in \eqref{r123def} are chosen so that
$\log(1-cw)$ vanishes for $w=0$ and has a branch cut along $w_{2-}(\Delta_2)$
(cf.~Figure~\ref{fig:mapping}), and $\log w$ is the principal branch with a cut along $(-\infty,0]$. 
\end{enumerate}
\end{proposition}
\begin{proof}
The fact that $N_{\alpha}$ satisfies items 1., 2., 3., and 4.\ of the RH problem \ref{rhpforN}
is proved as in \cite[Section 6]{KMW}. 

From \eqref{N0def}, \eqref{eq:F1F2F3}, and the behavior of the functions $w_k$ at $0$ 
as given in Lemma \ref{wat0} we obtain 
\begin{equation}\label{asymptoticsN0at0}
	N_0(z) = \begin{cases}
	c^{-2/3} \begin{pmatrix} K_1 \\ K_2 \\ K_3 	\end{pmatrix} 
	\begin{pmatrix} \omega^2 & \omega & 1 \end{pmatrix} z^{-1/3} + \O(1), \quad z\to 0, \; \Im z>0, \\
	c^{-2/3} \begin{pmatrix} K_1 \\ K_2 \\ K_3 \end{pmatrix} 
	\begin{pmatrix} \omega &  -\omega^2 & 1 \end{pmatrix} z^{-1/3} + \O(1), \quad z\to 0, \; \Im z<0,
\end{cases}
\end{equation}
where, for $\Im z<0$, there is a minus sign in the second entry of the
rowvector $\begin{pmatrix} \omega &  -\omega^2 & 1\end{pmatrix}$ because of the choice for
the  branch of the square root $D (w)^{1/2}$ used in \eqref{eq:F1F2F3}.
This proves \eqref{Nalphaat0} for the case $\alpha = 0$.

For general $\alpha$, we use that by \eqref{G1G2G3def} and 
\eqref{r123def} we have that 
\[ e^{G_{1}(z)} =\O(z^{-1/3}),  \qquad e^{G_{2}(z)}=\O(z^{2/3}), \qquad
	e^{G_3(z)} = \O(z^{-1/3}) \]
as $z\to 0$. Then using \eqref{Nalpha} we find \eqref{Nalphaat0} also for this case. 
\end{proof}

It will be convenient in what follows to consider besides $N_{\alpha}$ also the
matrix valued function
\begin{equation} \label{tildeNalpha}
 \widetilde N_{\alpha}(z) =  z^{\alpha/3}N_{\alpha}(z)
\begin{pmatrix} 1 & 0 & 0 \\
    0 & z^{-\alpha} & 0 \\ 0 & 0 & 1  \end{pmatrix},
\end{equation}
which already appeared in \eqref{Nalphaat0}.

\begin{lemma}
\label{lemma:localN}
The solution $N_{\alpha}$ of the RH problem \ref{rhpforN} is unique and
satisfies 
\[ \det N_{\alpha} (z) = \det  \widetilde N_{\alpha}(z) = 1, \qquad z \in \C \setminus (-\infty,q]. \]
where $\widetilde N_{\alpha}$ is given by \eqref{tildeNalpha}.

As $z \to 0$ we have both
\begin{equation}\label{asymptotics0}  
	\widetilde N_{\alpha}(z)= \O \left( |z|^{-1/3} \right), \quad \text{and} \quad
	 \widetilde N_{\alpha}(z) ^{-1}= \O \left( |z|^{-1/3} \right).
\end{equation}
\end{lemma}
\begin{proof}
From the jump condition in the RH problem \ref{rhpforN} it follows that $\det N_\alpha (z)$ 
has an analytic continuation to $\C \setminus \{0, q\}$. The isolated
singularities are removable by \eqref{Nalphaatq} and \eqref{Nalphaat0},
where it is important that $M_{\alpha}^{\pm}$ in \eqref{Nalphaat0} is of rank one.
Thus $\det N_\alpha (z)$ is an entire function, and since it tends to $1$ as $z \to \infty$
by \eqref{Ninfty} we  conclude from Liouville's theorem that $\det N_\alpha (z)=1$.
The uniqueness of the solution of the RH problem \ref{rhpforN} now follows with
similar arguments in a standard way. 

From the definition \eqref{tildeNalpha} we then also find that 
$\det \widetilde N_{\alpha}(z) = 1$. The behavior \eqref{asymptotics0} for
both $\widetilde N_{\alpha}$ and its inverse finally follows from the
condition \eqref{Nalphaat0} in the RH problem, and the fact that $\widetilde N_{\alpha}$
has determinant one. 
\end{proof}

We need two more results that will be used later in Section \ref{final}.

\begin{lemma}\label{lemma:inverseN_0}
For $N_0$ defined in \eqref{Nalpha}, we have
\begin{equation} \label{N0transpose}
N_0^{-1} = N_0^T  \begin{pmatrix}
1 & 0 & 0 \\
0 & 0 & -i \\
0 & -i & 0
\end{pmatrix}, \quad z\in \C \setminus (-\infty,q],
\end{equation}
where the superscript $T$ denotes the matrix transpose.
\end{lemma}
\begin{proof}
Observe from \eqref{eq:Njump1}--\eqref{eq:Njump2} that $N_0$ and $N_0^{-T}$ 
have the same jumps on $\Delta_1$ and $\Delta_2$, so that $N_0 N_0^T$ is analytic
in $\C \setminus \{0,q\}$. The singularies at $0$ and $q$ are removable because
of \eqref{Nalphaatq} and \eqref{Nalphaat0}, so that $N_0 N_0^T$ is entire.  
By \eqref{Ninfty},
\[
N_0(z) N_0^T(z) =  \begin{pmatrix}
1 & 0 & 0 \\
0 & 0 & i \\
0 & i & 0
\end{pmatrix} + \mathcal{O}\left(\frac{1}{z}\right), \quad z\to \infty,
\]
and the assertion \eqref{N0transpose} follows by  Liouville's theorem.
\end{proof}

As a consequence, we obtain the following corollary.
\begin{lemma} \label{lemma:K}
The constants $K_j$ from \eqref{eq:F1F2F3} satisfy the
relation
\begin{equation} \label{relationK}
	  K_1^2 - 2 i K_2 K_3  = 0.
\end{equation}
\end{lemma}
\begin{proof}
From  \eqref{N0transpose} we obtain
\[  N_0^{T}(z)\begin{pmatrix} 1 & 0 & 0 \\ 0 & 0 & -i \\ 0 & -i & 0 \end{pmatrix} N_0(z)= I. \]
Then insert the behavior \eqref{asymptoticsN0at0}  for both $N_0^T(z)$ and $N_0(z)$
and observe that the coefficient of $z^{-2/3}$ must vanish. This yields
\[ 	\begin{pmatrix} K_1 & K_2 & K_3 \end{pmatrix} 
	 \begin{pmatrix} 1 & 0 & 0 \\ 0 & 0 & -i \\ 0 & -i & 0 \end{pmatrix} 
	  \begin{pmatrix} K_1 \\ K_2 \\ K_3 \end{pmatrix} = 0,
	  \]
	  which is  \eqref{relationK}.
\end{proof}

\section{Local parametrices} \label{sec:parametrixQ}

The next step is the construction of local parametrices around
the branch points $q$ and $0$. We shall be brief about the local
parametrix $P$ around $q$. The main issue will be the 
construction of the local parametrix $Q$
around the origin.

\subsection{Parametrix $P$ around $q$}
We build $P$ in a fixed disk $D(q^*, r_q)$ around $q^* = 27 c^*/4$
with some (small) radius $r_q > 0$. For $t$ sufficiently
close to $t^*$ we then have that $q$ is in this neighborhood,
and we ask that $P$ should satisfy.

\begin{enumerate}
\item $P$ is analytic on $D (q^{*},r_q)\setminus \Sigma_{T}$;
\item  $P$ has the same jumps as $T$ has on $D (q^*,r_q)\cap \Sigma_{T}$, see
	\eqref{JT3}--\eqref{JT5};
\item as $n\to \infty$, 
	\begin{equation}\label{matchP}
	P(z)=N_{\alpha} (z)\left(I+\O\left(n^{-1} \right) \right)\quad
\text{uniformly for }|z-q^{*}| = r_q. 
\end{equation}
\end{enumerate}

The construction of  $P$ is done in a
standard way by means of Airy functions, see
\cite{De,DKMVZ1,DKMVZ2}. We will not give any details.

\subsection{Parametrix $Q$ around $0$: required properties}\label{local0}

The construction of the parametrix at the origin is the main novel
ingredient in the present RH analysis. A similar problem 
has been previously solved in \cite[Section 8]{BK4}, which serves as an 
inspiration for the approach we follow here. 

We want to define a matrix $Q$ in a neighborhood
$D (0,r_0)$ of the origin such that
\begin{enumerate}
\item $Q$ is analytic on $D (0,r_0)\setminus \Sigma_{T}$, where $ \Sigma_{T}$ has been 
defined in Section \ref{sec6}, see also Figure~\ref{deform3};
\item $Q$ has the same jumps as $T$ has on  $\Sigma_{T}  \cap D (0,r_0)$, 
see \eqref{JT1}--\eqref{JT4}. That is, we have
\begin{equation}\label{jumpQ}
	Q_{+} (z)=Q_{-} (z) J_{Q} (z), \qquad z \in \Sigma_T \cap D(0,r),
\end{equation}
where $J_{Q}$ is given by 
{\allowdisplaybreaks
\begin{align}
J_Q(x)&=
    \begin{pmatrix}
    1 & 0                  & 0 \\
	0 & 0 & -|x|^{-\alpha}\\
	0 & |x|^{\alpha} & 0
    \end{pmatrix}, &&  x\in\Delta_{2} \cap D(0,r), \label{JQ1}\\
J_Q(z)&=
   I + e^{\pm \alpha \pi i} z^{-\alpha}e^{n (\lambda_{2}-\lambda_{3})(z)} E_{23}, 
   && z\in\Delta_{2}^{\pm} \cap D(0,r), 
   \label{jQ2}\\
J_Q(x) &=
    \begin{pmatrix}
0 &  x^{\alpha}          & 0    \\
-x^{-\alpha} & 0 & 0 \\
0 &  0 & 1\\
    \end{pmatrix}, &&  x\in\Delta_1 \cap D(0,r), \label{JQ3}
\\
J_Q(z)&=
    I + z^{-\alpha} e^{n (\lambda_{2}-\lambda_{1})(z)} E_{2 1}, \quad z\in\Delta_{1}^{\pm} \cap D(0,r), \label{JQ4}
\end{align}}
\item As $z \to 0$, $Q$ has the same behavior as $T$ has, see item 4.\ in the RH problem \ref{rhpforT}.
\end{enumerate}

As we will see in Section \ref{final}, the radius $r_0$ will actually
depend on $n$, namely $r_0=n^{-1/2}$, so that the parametrix $Q$ will be
defined in a disk shrinking neighborhood as $n\to \infty$. 

Note that we did not state a matching condition for $Q$. 
Usually one asks for a matching condition of the type 
\[ Q (z)=N_{\alpha} (z) (I+ \O(1/n^{\kappa})), \]
as $n\to \infty$,  uniformly for $z$ on the circle $|z|= r_0$, 
with some $\kappa > 0$. In the present situation we are not able to get such a matching condition.
We will only be able to match $Q (z)$ with $N_{\alpha}(z)$ up to a bounded factor.
Hence, it will be necessary to introduce an additional
transformation, defined globally in the complex plane, as a last step
of the Riemann-Hilbert analysis, see Section \ref{final}.

\subsection{Reduction to constant jumps}

The jump condition \eqref{jumpQ} can be reduced to a condition
with constant jump matrices as follows. 
We put
\begin{equation}\label{defLambda2}
	\Lambda_n(z) =\diag (1,z^{\alpha},1) \, \diag \left(e^{-n \lambda_{1} (z)},e^{-n \lambda_{2} (z)},e^{-n \lambda_{3} (z)}\right),
\end{equation}
with $z^{\alpha} = |z|^{\alpha} e^{i \alpha \arg(z)}$ and $\arg z \in (0, 2\pi)$ is defined with the branch cut $[0, +\infty)$.
Then the jump matrices $J_Q$ from \eqref{JQ1}--\eqref{JQ4} factorize as 
\begin{equation}\label{defjQ0}
J_Q (z) =\Lambda_{n,-}^{-1} (z)\, J_Q^{0} (z) \, \Lambda_{n,+} (z), 
\end{equation}
where
{\allowdisplaybreaks
\begin{align}
J_Q^{0}(x)&=
    \begin{pmatrix}
    1 & 0                  & 0 \\
0 & 0 & -e^{-\alpha\pi i}\\
0 & e^{-\alpha\pi i} & 0
    \end{pmatrix}, \quad x\in\Delta_{2}, \label{jQ01}\\
J_Q^{0}(z)&=
   I + e^{\pm \alpha \pi i} E_{23}, \quad z\in\Delta_{2}^{\pm}, \label{jQ02}\\
J_Q^{0} (x) &=
    \begin{pmatrix}
0 &  1         & 0    \\
-1 & 0 & 0 \\
0 &  0 & 1\\
    \end{pmatrix}, \quad  x\in\Delta_1,\label{jQ03}
\\
J_Q^{0}(z)&=
    I + E_{2 1}, \quad z\in\Delta_{1}^{\pm}, \label{jQ04}
\end{align}}

Observe that the jump matrices $J_Q^{0}$ agree with the jump matrices $J_{\Phi}$ in the RH problem \ref{rhpforPhi}
for $\Phi_{\alpha}$, see Figure~\ref{fig:localanalysis2bis}, except that the jump matrices $J_{\Phi}$ are on six infinite rays 
$\Sigma_{\Phi}$ emanating from the origin.
We will therefore look for $Q$ in the form
\begin{equation} \label{Qform} 
	Q(z) = E_n(z) \Phi_{\alpha}(f_n(z); \tau_n(z)) \Lambda_n(z) 
	\end{equation}
where $E_n(z)$ and $\tau_n(z)$ are analytic in $D(0,r_0)$ and 
where $f_n$ is a conformal map on $D(0,r_0)$ that maps the contours $\Sigma_T \cap D(0,r_0)$
into the six rays $\Sigma_{\Phi}$ so that $[0,r_0)$ is mapped into the positive real axis.

For any choice of conformal $f_n$, and analytic $E_n$ and $\tau_n$, the matrix valued $Q$
defined by \eqref{Qform} will then satisfy the required jumps \eqref{jumpQ}.

\subsection{Definition of $\Phi_{\alpha}(z; \tau)$} \label{subs:defPhi}

We next construct the matrix valued function $\Phi_{\alpha}(z) = \Phi_{\alpha}(z;\tau)$
that solves the RH problem \ref{rhpforPhi} stated in the introduction.

As already mentioned in the introduction we use the following third order 
linear differential equation
\begin{equation} \label{ODEnew}
	zp''' + \alpha p'' - \tau p' - p = 0,
\end{equation}  
with $\alpha>-1$ and $\tau\in \C$.
Then $z=0$ is a regular singular point of this ODE with Frobenius indices 
$0$, $1$, and $-\alpha+2$. There are two linearly independent
entire solutions, and one solution that branches at the origin.

Due to the special form of the ODE \eqref{ODEnew} (the coefficients are at most
linear in $z$), it can be solved with Laplace transforms. 
We find solutions with an integral representation
\[ p(z) = C \int_{\Gamma}
    t^{\alpha-3} e^{\tau/t} e^{1/(2t^2)} e^{zt} \, dt, 
\]
where $\Gamma$ is an appropriate contour in the complex $t$-plane and $C$ is a constant.
A basis of solutions of \eqref{ODEnew} can be chosen by
selecting different contours $\Gamma$. We will make use of four
contours $\Gamma_j$, $j=1,2,3,4$, defined as follows, see also Figure~\ref{contourfct}:
%%%%%%%%%%%%%%%%%%%%%%%%%%%%%%%%%%%%%%%%%%%%%%%%%%%%%%%%%%
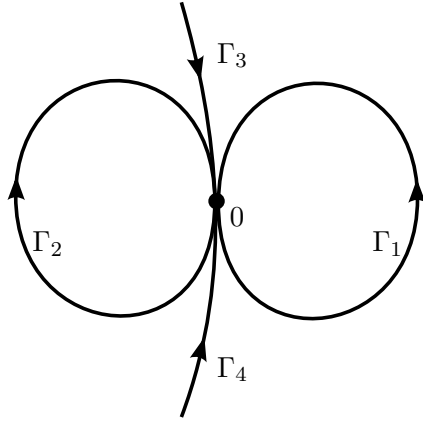
\begin{figure}[t]
\centering 
\begin{overpic}[scale=1.7]%
{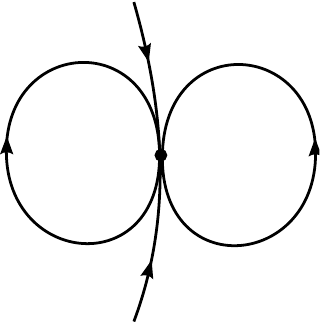}%
%{contourfct}%
\put(53,46){$0 $}
  \put(6,40){$\Gamma_{2} $}
  \put(87,40){$\Gamma_{1}  $}
   \put(50,85){$\Gamma_{3} $}
  \put(50,10){$\Gamma_{4}  $}
    \end{overpic}
\caption{The contours of integration $\Gamma_{j}$, $j=1,\ldots 4$,
used in the definition of the functions $p_{j}$.}
\label{contourfct}
\end{figure}
%%%%%%%%%%%%%%%%%%%%%%%%%%%%%%%%%%%%

\begin{enumerate}
\item[(a)] We let $\Gamma_1$ be a simple closed contour passing
through the origin, but otherwise lying in the right half-plane and
which is tangent to the imaginary axis. $\Gamma_1$ is oriented
counterclockwise and we put
\begin{equation}\label{integralForp1}
	 p_1(z) = \int_{\Gamma_1}
    t^{\alpha-3} e^{\tau/t} e^{1/(2t^2)} e^{zt} \, dt. 
    \end{equation}
We choose the branch of $t^{\alpha-3} = |t|^{\alpha-3} e^{i (\alpha-3) \arg t}$
with $-\pi/2 < \arg t < \pi/2$.  

\item[(b)] $\Gamma_2$ is the reflection of $\Gamma_1$ in the
imaginary axis, oriented clockwise. We put
\begin{equation}\label{integralForp2}
p_2(z) = e^{-\alpha \pi i} \int_{\Gamma_2}
    t^{\alpha-3} e^{\tau/t} e^{1/(2t^2)} e^{zt} \, dt.
\end{equation}
In \eqref{integralForp2} we define the branch of $t^{\alpha-3}$
with $\pi/2 < \arg t < 3 \pi/2$.

\item[(c)] $\Gamma_3$ is an unbounded contour in the upper
half-plane that starts at infinity at an angle where $\Re(zt) < 0$
as $t \to \infty$, and ends at the origin along the positive
imaginary axis. We put
\begin{equation}\label{integralForp3}
	p_3(z) = e^{-\alpha \pi i} \int_{\Gamma_3}
    t^{\alpha-3} e^{\tau/t} e^{1/(2t^2)} e^{zt} \, dt. 
    \end{equation}
In \eqref{integralForp3} we take $t^{\alpha-3}$ 
with $0 < \arg t < \pi$.

The condition $\Re(zt) < 0$ is necessary to have convergence of
the integral. This condition can be met with a contour that is in
the upper half-plane if and only if $-\pi/2 < \arg z < 3\pi/2$.
Therefore $p_3$ is well-defined and analytic in $\mathbb C
\setminus i\mathbb R_-$, that is, with a branch cut along the
negative imaginary axis.

\item[(d)] $\Gamma_4$ is similar to $\Gamma_3$, but in the lower
half-plane. It is an unbounded contour in the lower half-plane
starting at infinity at an angle where $\Re(zt) < 0$ as $t \to
\infty$, and it ends at the origin along the negative imaginary
axis. We put
\begin{equation}\label{integralForp4}
	 p_4(z) = e^{\alpha \pi i} \int_{\Gamma_4}
    t^{\alpha-3} e^{\tau/t} e^{1/(2t^2)} e^{zt} \, dt, 
    \end{equation}
In \eqref{integralForp4} the branch of $t^{\alpha-3}$ is defined with
$-\pi < \arg t < 0$.

Then $p_4$ is well-defined and analytic in $\mathbb C \setminus
i\mathbb R_+$, thus with a branch cut along the positive imaginary
axis.
\end{enumerate}

With these definitions it is clear that $p_1$ and $p_2$ are entire
functions, while $p_3$ and $p_4$ have a branch point at the
origin. The four solutions are not linearly independent, but any
three of them are.

We define $\Phi_{\alpha}$ in each of the sectors determined by the 
six rays: $\arg z = 0, \pm \pi/4,
\pm 3 \pi/4, \pi$, as shown in Figure~\ref{fig:localanalysis2bis}, 
as a Wronskian matrix using three of the functions $p_j$.

\begin{definition} \label{Phidefinition}
We define $\Phi_{\alpha}$ in the six sectors as follows.
{\allowdisplaybreaks
\begin{align} \label{defPhi1}
\Phi_{\alpha}(z;\tau)  & = \frac{e^{\tau^2/6}}{\sqrt{2\pi}} \begin{pmatrix}
-p_{4}(z) & p_{3}(z) & p_{1}(z)\\
-p_{4}'(z) & p_{3}'(z) & p_{1}'(z)\\
-p_{4}''(z) & p_{3}''(z) & p_{1}''(z)
\end{pmatrix} && 0<\arg z<\pi/4, \\ \label{defPhi2}
\Phi_{\alpha}(z;\tau) & = \frac{e^{\tau^2/6}}{\sqrt{2\pi}} \begin{pmatrix}
p_{2}(z) & p_{3}(z) & p_{1}(z)\\
p_{2}'(z) & p_{3}'(z) & p_{1}'(z)\\
p_{2}''(z) & p_{3}''(z) & p_{1}''(z)
\end{pmatrix}  && \pi/4<\arg z<3\pi/4,\\ \label{defPhi3}
\Phi_{\alpha}(z;\tau)  & = \frac{e^{\tau^2/6}}{\sqrt{2\pi}} \begin{pmatrix}
p_{2}(z) & p_{3}(z) & -e^{-\alpha\pi i}p_{4}(z)\\
p_{2}'(z) & p_{3}'(z) & -e^{-\alpha\pi i}p_{4}'(z)\\
p_{2}''(z) & p_{3}''(z) & -e^{-\alpha\pi i}p_{4}''(z)
\end{pmatrix} && 3\pi/4<\arg z<\pi, \\ \label{defPhi4}
\Phi_{\alpha}(z;\tau)  & = \frac{e^{\tau^2/6}}{\sqrt{2\pi}} \begin{pmatrix}
p_{2}(z) & p_{4}(z) & e^{\alpha\pi i}p_{3}(z)\\
p_{2}'(z) & p_{4}'(z) & e^{\alpha\pi i}p_{3}'(z)\\
p_{2}''(z) & p_{4}''(z) & e^{\alpha\pi i}p_{3}''(z)
\end{pmatrix}  && -\pi<\arg z<-3\pi/4,\\ \label{defPhi5}
\Phi_{\alpha}(z;\tau)  & = \frac{e^{\tau^2/6}}{\sqrt{2\pi}} \begin{pmatrix}
p_{2}(z) & p_{4}(z) & p_{1}(z)\\
p_{2}'(z) & p_{4}'(z) & p_{1}'(z)\\
p_{2}''(z) & p_{4}''(z) & p_{1}''(z)
\end{pmatrix} && -3\pi/4<\arg z<-\pi/4,\\ \label{defPhi6}
\Phi_{\alpha}(z;\tau)  & = \frac{e^{\tau^2/6}}{\sqrt{2\pi}} \begin{pmatrix}
p_{3}(z) & p_{4}(z) & p_{1}(z)\\
p_{3}'(z) & p_{4}'(z) & p_{1}'(z)\\
p_{3}''(z) & p_{4}''(z) & p_{1}''(z)
\end{pmatrix} && -\pi/4<\arg z<0.
\end{align}}
\end{definition}

The scalar factor $\frac{e^{\tau^2/6}}{\sqrt{2\pi}}$ is needed 
in \eqref{defPhi1}--\eqref{defPhi6} in order to have the exact asymptotic
behavior \eqref{Phiasymptotics1} in the RH problem \ref{rhpforPhi}.
 
The functions $p_j$ clearly depend on $\alpha$ and $\tau$, although
we did not emphasize it in the notation.

\begin{proposition} \label{prop:propertiesPhi}
Let $\alpha > -1$ and $\tau \in \C$. Then 
	$\Phi_{\alpha}(z;\tau)$ as defined above satisfies the RH problem \ref{rhpforPhi} 
	stated in the introduction.
\end{proposition}
\begin{proof}
It is a tedious but straightforward check based on the integral representations
\eqref{integralForp1}--\eqref{integralForp4} that $\Phi_{\alpha}$ has the
constant jumps on six rays in the complex $z$-plane as given in
\eqref{Phijumps}. 

The asymptotic properties \eqref{Phiasymptotics1}--\eqref{Phiasymptotics2}
follow from a steepest descent analysis for the integrals defining the
functions $p_j$. We give more details about this in the next subsection,
where we also describe the next term in the 
asymptotic expansions of \eqref{Phiasymptotics1}--\eqref{Phiasymptotics2}
since we will need this later on.

The behavior  \eqref{Phiat01}--\eqref{Phiat03} at $0$ follows from the
behavior of the solutions $p_j$ of the ODE \eqref{ODEnew} at $0$. Since $p_1$
and $p_2$ are entire solutions, they are bounded at $0$. The solutions $p_3$ and $p_4$
satisfy
\[ p_j(z) = \O(z^{2-\alpha}), \quad p_j'(z) = \O(z^{1-\alpha}), \quad p_j''(z) = \O(z^{-\alpha}), \qquad j=3,4, \]
as $z \to 0$, which can be found by analyzing the integral  representations \eqref{integralForp3} and \eqref{integralForp4}.
This proves \eqref{Phiat01}--\eqref{Phiat03} in view of the definition of $\Phi_{\alpha}$ in
in Definition \ref{Phidefinition}.
\end{proof}

\subsection{Asymptotics of $\Phi_{\alpha}$}

As before we define $\omega = e^{2 \pi i/3}$ and $\theta_k$ as in \eqref{thetak}.
We also put
\begin{equation} \label{defL}
 L_\alpha(z)  = z^{-\alpha/3}
    \begin{pmatrix} z^{1/3} & 0 & 0 \\ 0 & 1 & 0 \\ 0 & 0 & z^{-1/3}
    \end{pmatrix} \times
    \left\{ \begin{array}{l}
    \begin{pmatrix} \omega & \omega^2 & 1 \\
    1 & 1 & 1 \\ \omega^2 & \omega & 1 \end{pmatrix}
    \begin{pmatrix} e^{\alpha \pi i/3} & 0 & 0 \\ 0 & e^{-\alpha
    \pi i/3} & 0 \\ 0 & 0 & 1 \end{pmatrix}, \\
    \hfill{\text{for } \Im z > 0,} \\[10pt]
    \begin{pmatrix} \omega^2 & -\omega & 1 \\
    1 & -1 & 1 \\ \omega & -\omega^2 & 1 \end{pmatrix}
    \begin{pmatrix} e^{-\alpha \pi i/3} & 0 & 0 \\ 0 & e^{\alpha
    \pi i/3} & 0 \\ 0 & 0 & 1 \end{pmatrix}, \\
    \hfill{\text{for } \Im z < 0,} \end{array} \right.
\end{equation}
where all fractional powers are defined with a branch cut along the
negative real axis. Define also the constant matrices

\begin{align}  \nonumber
M_{\alpha}^+(\tau) & =  \frac{\tau (\tau^{2}+9\alpha -9)}{27}
\diag \left(\omega^{2}, \omega, 1\right) \\  \label{Mplus}
	& + \frac{i\tau}{3\sqrt{3}} 
	\diag\left(\omega^{-\alpha/2}, \omega^{\alpha/2}, 1\right) 
	\begin{pmatrix}	0 & -\omega & 1 \\
	\omega^{2} & 0 & -1 \\
	-\omega^{2} & \omega & 0
	\end{pmatrix} 
	\diag\left(\omega^{\alpha/2}, \omega^{-\alpha/2}, 1\right),\\
	\nonumber 
M_{\alpha}^-(\tau) & =  \frac{\tau (\tau^{2}+9\alpha -9)}{27}
\diag \left(\omega, \omega^2, 1 \right) \\
	& + \frac{i\tau}{3\sqrt{3}} 
	\diag\left(\omega^{\alpha/2}, \omega^{-\alpha/2}, 1\right) 
	\begin{pmatrix}	0 & -\omega^2 & -1 \\ 	\label{Mminus}
	\omega & 0 & -1 \\
	\omega & \omega^2 & 0
	\end{pmatrix} 
	\diag\left(\omega^{-\alpha/2}, \omega^{\alpha/2}, 1\right).
	\end{align}

\begin{lemma} \label{lemma:asymptPhi}
Let $\alpha > -1$ and $\tau \in \mathbb C$. Then we have, as $z \to \infty$,
\begin{align} \label{PhiinUHP}
	\Phi_{\alpha}(z; \tau) =
    \frac{i}{\sqrt{3}} L_\alpha (z)
    \left(I + \frac{M_{\alpha}^+(\tau)}{z^{1/3}}+ \O(z^{-2/3})\right)
    & \begin{pmatrix} e^{\theta_1(z)} & 0 & 0 \\ 0 & e^{\theta_2(z)}
    & 0 \\ 0 & 0 & e^{\theta_3(z)} \end{pmatrix} \\
    &  \text{for } \Im z > 0, \nonumber \\
	\Phi_{\alpha}(z; \tau)  = \label{PhiinLHP}
    \frac{i}{\sqrt{3}} L_\alpha (z)
    \left(I + \frac{M_{\alpha}^-(\tau)}{z^{1/3}}+ \O(z^{-2/3})\right)
    & \begin{pmatrix} e^{\theta_2(z)} & 0 & 0 \\ 0 & e^{\theta_1(z)}
    & 0 \\ 0 & 0 & e^{\theta_3(z)} \end{pmatrix} \\
    & \text{for } \Im z < 0. \nonumber
    \end{align}
The expansions \eqref{PhiinUHP} and \eqref{PhiinLHP} are valid
uniformly for $\tau$ in a compact subset of the complex plane.
\end{lemma}
\begin{proof} We apply the
classical steepest descent analysis to the integral representations
\eqref{integralForp1}--\eqref{integralForp4} of the functions $p_{j}$. 
We set $\sigma (t;z,\tau)=zt+\tau/t+1/ (2t^{2})$. The
saddle points are solutions of
\begin{equation}\label{saddle}
\frac{\partial \sigma}{\partial t}
=z-\frac{\tau}{t^{2}}-\frac{1}{t^{3}}=0.\end{equation}
%For $\tau=0$, there are three solutions
%$t_{k}^{0}=\omega^{k}z^{-1/3}$, $k=1,2,3$. 
As $z\to\infty$, while
$\tau$ remains bounded, the three solutions to \eqref{saddle} have the following expansion:
\begin{equation}\label{estimt}
	t_{k}=t_{k} (z;\tau) = 
	\omega^{2k} z^{-1/3} + \omega^{k} \frac{\tau}{3}z^{-2/3}
	+\O(z^{-4/3}),\qquad k = 1,2,3,
\end{equation}
and the corresponding values at the saddles are
\begin{align} \nonumber
	\sigma (t_{k} (z;\tau);z,\tau) & = 
	\frac{3}{2}\omega^{2k} z^{2/3} + \tau\omega^{k} z^{1/3} -
	\frac{\tau^{2}}{6} + 	\O(z^{-1/3}) \\
	& \label{estimtheta} =
		\theta_k(z; \tau) - \frac{\tau^{2}}{6}  +
	\O (z^{-1/3}), \qquad \text{as } z\to\infty,
\end{align}
with $\theta_{k}$ introduced in \eqref{thetak}.

If $C_{k}$ is the steepest descent path through the saddle point $t_{k}$, 
we obtain from \eqref{estimtheta} and standard steepest descent arguments that
\begin{equation} \label{pathCkestimate}
	\int_{C_{k}} t^{\alpha-3} e^{\tau/t}e^{1/(2t^{2})} e^{zt} dt = 
	\pm \sqrt{\frac{2\pi}{-\frac{\partial^{2}\sigma}{\partial t^{2}}
	(t_{k} ;z,\tau)}}\, t_{k}^{\alpha-3} e^{-\tau^2/6}
	e^{\theta_k(z;\tau)} (1+\O (z^{-1/3}))
\end{equation}
as $z \to \infty$, 
where the $\pm$ sign depends on the orientation of the steepest
descent path. Plugging \eqref{estimt} into \eqref{pathCkestimate}, and using the fact that
\[
	\frac{\partial^{2} \sigma}{\partial t^{2}}
		=\frac{3}{t^{4}}\left(1+\frac{2}{3}\tau t \right),
\]
we obtain as $z \to \infty$
\begin{equation}\label{steep}
	\int_{C_{k}} t^{\alpha-3} e^{\tau/t}e^{1/(2t^{2})} e^{zt} dt = 
	\pm \sqrt{\frac{-2\pi}{3}}\,e^{- \tau^{2}/6}
	\left(\omega^{k}z^{-1/3}\right)^{\alpha - 1} 
	e^{\theta_{k}(z; \tau)} \left(1+\O (z^{-1/3})\right).
\end{equation}

%%%%%%%%%%%%%%%%%%%%%%%%%%%%%%%%%%%%%%%%%%%%%%%%%%%%%%%%%%
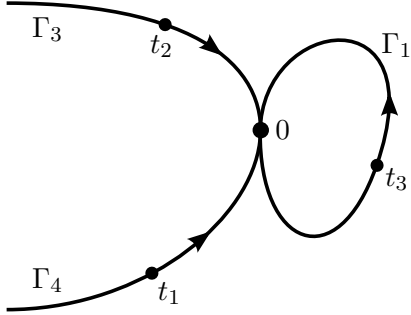
\begin{figure}[t]
\centering 
\begin{overpic}[scale=1.7]%
{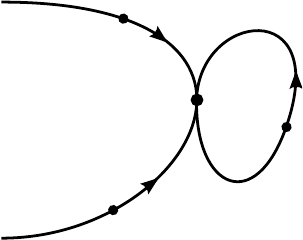}%
%{contourfct}%
\put(69,44){$0 $}
  %\put(5,40){$\Gamma_{2} $}
  \put(96,66){$\Gamma_{1}  $} 
   \put(7,70){$\Gamma_{3} $}
  \put(7,6){$\Gamma_{4}  $}
   \put(37,66){$t_2  $}
    \put(39,3){$t_1  $}
     \put(97,32){$t_3  $}
    \end{overpic}
\caption{Deformation of the contours of integration $\Gamma_1$, $\Gamma_3$ and $\Gamma_4$ 
into the steepest descent paths for  the integrals defining $p_1$, $p_3$ and $p_4$ in the
case where $0<\arg z<\pi/4$.}
\label{fig: pathsforpDeform}
\end{figure}
%%%%%%%%%%%%%%%%%%%%%%%%%%%%%%%%%%%%

Consider now $\Phi_{\alpha}(z)$ as defined in \eqref{defPhi1} for $0<\arg z<\pi/4$. We see
that only $p_1$, $p_3$ and $p_4$ play a role in this sector. The corresponding contours
of integration can be deformed into the steepest descent paths
through one of the saddle points as shown in Figure\ref{fig: pathsforpDeform}. 

Hence, in the given sector, the functions $p_1$, $p_3$ and $p_4$ have an 
asymptotic behavior of the form \eqref{steep} for some particular $k$
and some choice of the $\pm$ sign, and multiplied by $e^{\pm \alpha \pi i}$ in case of $p_3$ and $p_4$,
see the formulas \eqref{integralForp1}, \eqref{integralForp3}, and \eqref{integralForp4}.
This will lead to the asymptotic expansion for the first row of \eqref{PhiinUHP} for $0 < z < \pi/4$,
except for the determination of $M_{\alpha}^+(\tau)$.
The second and rows can be dealt with similarly. Here we have to 
consider the first and second derivatives of $p_1$, $p_3$, and $p_4$,
which have similar integral representations \eqref{integralForp1}, \eqref{integralForp3}, and \eqref{integralForp4}, 
but with $\alpha$ replaced by $\alpha+1$ and $\alpha+2$, respectively.

The other sectors can be analyzed in a similar way. Tracing the behavior of the dominant saddle point 
we find that the asymptotic expression just obtained remains valid in the full upper half plane,
while in the lower half plane we find \eqref{PhiinLHP}, again up to the determination of
$M_{\alpha}^-(\tau)$.

What remains is to obtain the constants $M_{\alpha}^{\pm}(\tau)$ that appear in the $\O(z^{-1/3})$ term in
the expansions \eqref{PhiinUHP} and \eqref{PhiinLHP}. This can be done by calculating
the next terms in the asymptotic expansion of the integrals. Alternatively,
we can use the fact that  $\Phi_{\alpha}$ solves the first-order matrix-valued ODE 
\[
	z \Phi_{\alpha}'(z) = \begin{pmatrix}
	0 & z & 0 \\
	0 & 0 & z \\
	1 & \tau & -\alpha
	\end{pmatrix}\Phi_{\alpha}(z).
\]
Substituting into this  the asymptotic expansions for $\Phi_{\alpha}$ and
equating terms on both sides, we find after lengthy calculations (that were actually performed
with the help of Maple)
the formulas for $M_{\alpha}^{\pm}(\tau)$ as given in \eqref{Mplus} and \eqref{Mminus}.
\end{proof}

\subsection{Definition and properties of $f(z)$ and $\tau(z)$} \label{sec:ftau}

We will take the local parametrix $Q$ in the form, see also \eqref{Qform}
\begin{equation} \label{Qform2}
	Q(z) = E_n(z) \Phi_{\alpha}(n^{3/2} f(z); n^{1/2} \tau(z))
    \begin{pmatrix} e^{-n \lambda_1(z)} & 0 & 0 \\
    0 & z^{\alpha} e^{-n \lambda_2(z)} & 0 \\
    0 & 0 & e^{-n \lambda_3(z)}
    \end{pmatrix} e^{\frac{2nz}{3t(1-t)}},
\end{equation} 
where $E_n$ is an analytic prefactor, $f(z)$ is a conformal map defined
in a neighborhood of $z=0$ and $\tau(z)$ is analytic in $z$.
Assuming that $f$ maps the contour $\Sigma_T \cap D(0,r)$ to the six
rays $\Sigma_{\Phi}$ such that $f(z)$ is positive for positive real $z$,
then it follows from the above construction  that $Q$ will satisfy
the required jump condition.

We are going to take $f(z)$ and $\tau(z)$ in such a way that
the exponential factors in \eqref{Qform2} are cancelled. That is,
we want analytic $f(z)$ and $\tau(z)$ such that
\begin{equation}\label{thetaup}
 \theta_k(f(z); \tau(z)) = \lambda_k(z) - \frac{2z}{3t(1-t)},\qquad k=1,2,3,
\end{equation}
    for $\Im z > 0$, while for $\Im z < 0$,
\begin{align}\label{thetalow1}
    \theta_1(f(z); \tau(z)) & = \lambda_2(z) - \frac{2z}{3t(1-t)} -
2\pi i \\
\label{thetalow2}
    \theta_2(f(z); \tau(z)) & = \lambda_1(z) - \frac{2z}{3t(1-t)} + 2\pi i \\
\label{thetalow3}
    \theta_3(f(z); \tau(z)) & = \lambda_3(z) - \frac{2z}{3t(1-t)},
    \end{align}
where the functions $\theta_k$ were defined in \eqref{thetak}.

To define $f(z)$ and $\tau(z)$ we use the functions $f_3(z;t)$ and $g_3(z;t)$
from Lemma \ref{lambdaat0}. 

\begin{definition}
	We put
\begin{equation}\label{defftau}
	f (z) = f (z;t)=z[f_{3} (z;t)]^{3/2},\qquad 
	\tau (z) = \tau (z;t)=\frac{g_{3} (z;t)}{{f_{3} (z;t)}^{1/2}},
\end{equation}
where as usual we take the principal branches of the fractional powers. 
We write $f(z;t)$ and $g (z;t)$ in order to emphasize their dependence on the parameter $t$.
\end{definition}

\begin{lemma}\label{behfandtau} 
There exist $r_0 >0$ and $\delta>0$ such that for each
$t\in(t^*-\delta,t^{*}+\delta)$ we have that $z \mapsto f (z;t)$ is a
conformal mapping on the disk $D (0,r_0)$ and $z \mapsto \tau (z;t)$ is analytic
on $D (0,r_0)$. The map $z \mapsto f (z;t)$ is positive for positive real $z$
and negative for negative real z.

In addition, we have 
\begin{equation}\label{behavtau}
\tau (z;t)=\O (t-t^{*})+\O (z)\qquad \text{as }t\to t^{*}\text{ and } z\to 0.
\end{equation}
\end{lemma}
\begin{proof} 
Because of \eqref{f3g3at0} we have that $f_3(0) > 0$ and so 
$f(z)$ defined by \eqref{defftau} is indeed a conformal map in
a neigborhood of $z=0$ which is positive for positive values of $z$.
Also $\tau(z)$ is analytic in a neighborhood of $z=0$.

We have
\begin{equation}\label{f3g3t*}
f_{3} (z;t)=f_{3} (z,t^{*})+\O (t-t^{*}),\quad 
g_{3} (z;t)=g_{3} (z,t^{*})+\O (t-t^{*}),\quad \text{as }t\to t^{*},
\end{equation}
uniformly for $z$ in a neighborhood of $0$, and
\begin{equation}\label{expf3g3at0}
f_{3} (z;t^{*})= (c^{*})^{-2/3}+\O (z),\qquad g_{3}
(z;t^{*})=\O (z)\quad 
\text{as }z\to 0.
\end{equation}
This follows from the definitions \eqref{defcp} of $c$ and $p$,
equation \eqref{eqxy}, and the definitions of $f_{j}$ and $g_{j}$,
$j=1,2,3$. Expansions \eqref{expf3g3at0} also use \eqref{f3g3at0}
and the fact that $p=0$ 
when $t=t^{*}$. Then \eqref{behavtau} is a consequence of the previous expansions
and the definitions of $f (z)$ and $\tau (z)$ in \eqref{defftau}.
\end{proof}

\subsection{Definition and properties of the prefactor $E_n(z)$}

The prefactor $E_n(z)$ in the definition of $Q(z)$ in
\eqref{Qform2}  should be analytic and chosen so that
$Q$ is close to  $N_{\alpha}$ on $|z|=n^{-1/2}$. In view of the expansion
of $\Phi$ given in Lemma \ref{lemma:asymptPhi}, we set the following
definition. We use $r_0 > 0$ as given by Lemma~\ref{behfandtau}
and we assume $t \in (t^*-\delta, t^*+\delta)$.

\begin{definition} \label{def:E}
We define for $z \in D(0,r_0) \setminus \mathbb R$,
\begin{equation} \label{defE}
	E_n(z) =  - i \sqrt{3}  N_{\alpha}(z)
    \begin{pmatrix} 1 & 0 & 0 \\ 0 & z^{-\alpha} & 0 \\ 0 & 0 & 1 \end{pmatrix} L_\alpha ^{-1}(n^{3/2}f(z)), 
    \end{equation}
where $L_\alpha$ has been introduced in \eqref{defL}, and $N_\alpha$
is described in Section \ref{sec:N}. 
\end{definition}

\begin{lemma} \label{Eanalytic}
$E_n$ and $E_n^{-1}$ have an analytic continuation to $D(0,r_0)$.
\end{lemma}
\begin{proof}
Taking into account \eqref{eq:Njump1}--\eqref{eq:Njump2} we see that for 
$\widetilde N_\alpha$ defined in \eqref{tildeNalpha},
\[ \widetilde{N}_{\alpha,+}(x) = \widetilde{N}_{\alpha,-}(x) \begin{pmatrix} 0 & 1 & 0 \\ -
    1 & 0 & 0 \\ 0 & 0 & 1 \end{pmatrix}
    \qquad \text{for } x \in (0, r_0) \]
and
\[ \widetilde{N}_{\alpha,+}(x) = \widetilde{N}_{\alpha,-}(x)
    \begin{pmatrix} 1 & 0 & 0 \\ 0 & 0 & - e^{-\alpha \pi i}
    \\ 0 & e^{-\alpha \pi i} & 0 \end{pmatrix},
    \qquad \text{for } x \in (-r_0, 0). \]

For $L_\alpha$ we find the same jump matrices. Indeed, for $x > 0$, we
 have by the definition of $L_\alpha$,
{\allowdisplaybreaks
\begin{align*} L_{\alpha,-}(x)^{-1} L_{\alpha,+}(x)
     & = \begin{pmatrix} e^{\alpha \pi i/3} & 0 & 0 \\ 0 & e^{-\alpha
    \pi i/3} & 0 \\ 0 & 0 & 1 \end{pmatrix}
    \begin{pmatrix} \omega^2 & -\omega & 1 \\
    1 & -1 & 1 \\ \omega & -\omega^2 & 1 \end{pmatrix}^{-1} \\
    & \qquad \qquad \times
    \begin{pmatrix} \omega & \omega^2 & 1 \\
    1 & 1 & 1 \\ \omega^2 & \omega & 1 \end{pmatrix}
    \begin{pmatrix} e^{\alpha \pi i/3} & 0 & 0 \\ 0 & e^{-\alpha
    \pi i/3} & 0 \\ 0 & 0 & 1 \end{pmatrix} \\
    & = \begin{pmatrix} e^{\alpha \pi i/3} & 0 & 0 \\ 0 & e^{-\alpha
    \pi i/3} & 0 \\ 0 & 0 & 1 \end{pmatrix}
    \begin{pmatrix} 0 & 1 & 0 \\
    -1 & 0 & 0 \\ 0 & 0 & 1 \end{pmatrix}
    \begin{pmatrix} e^{\alpha \pi i/3} & 0 & 0 \\ 0 & e^{-\alpha
    \pi i/3} & 0 \\ 0 & 0 & 1 \end{pmatrix} \\
    & =  \begin{pmatrix} 0 & 1 & 0 \\
    -1 & 0 & 0 \\ 0 & 0 & 1 \end{pmatrix}
    \end{align*}
}
and for $x < 0$,
{\allowdisplaybreaks
\begin{align*} L_{\alpha,-}(x)^{-1} L_{\alpha,+}(x)
     & = e^{-2\alpha \pi i/3} \begin{pmatrix} e^{\alpha \pi i/3} & 0 & 0 \\ 0 & e^{-\alpha
    \pi i/3} & 0 \\ 0 & 0 & 1 \end{pmatrix}
    \begin{pmatrix} \omega^2 & -\omega & 1 \\
    1 & -1 & 1 \\ \omega & -\omega^2 & 1 \end{pmatrix}^{-1} \\
    & \qquad \times
    \begin{pmatrix} \omega & 0 & 0 \\ 0 & 1 & 0 \\ 0 & 0 & \omega^2 \end{pmatrix}
    \begin{pmatrix} \omega & \omega^2 & 1 \\
    1 & 1 & 1 \\ \omega^2 & \omega & 1 \end{pmatrix}
    \begin{pmatrix} e^{\alpha \pi i/3} & 0 & 0 \\ 0 & e^{-\alpha
    \pi i/3} & 0 \\ 0 & 0 & 1 \end{pmatrix} \\
    & = e^{-2\alpha \pi i/3} \begin{pmatrix} e^{\alpha \pi i/3} & 0 & 0 \\ 0 & e^{-\alpha
    \pi i/3} & 0 \\ 0 & 0 & 1 \end{pmatrix}
    \begin{pmatrix} 1 & 0 & 0 \\ 0 & 0 & -1 \\ 0 & 1 & 0
    \end{pmatrix}
    \begin{pmatrix} e^{\alpha \pi i/3} & 0 & 0 \\ 0 & e^{-\alpha
    \pi i/3} & 0 \\ 0 & 0 & 1 \end{pmatrix} \\
    & = \begin{pmatrix} 1 & 0 & 0 \\ 0 & 0 & - e^{-\alpha \pi i} \\
    0 & e^{-\alpha \pi i} & 0 \end{pmatrix}.
\end{align*}
}
Thus, the jumps for $\widetilde{N}_{\alpha}$ and $L_{\alpha}$ are the same. Since
$f$ is a conformal map on $D(0,r_0)$ with $f(x) > 0$ for $x
\in (0, r_0)$ and $f(x) < 0$ for $x < 0$, it follows that
$\widetilde{N}_{\alpha}(z) L_{\alpha}^{-1}(n^{3/2}f(z))$ has an analytic
continuation to $D(0,r_0) \setminus \{0\}$. As a result, we conclude that $E_n$ has an
analytic continuation to $D(0,r_0) \setminus \{0\}$.

We show that the isolated singularity at the origin is removable. Indeed, by Lemma~\ref{lemma:localN} we have
\[  \widetilde{N}_{\alpha}(z) = \O(z^{-1/3}) \qquad
\text{as } z \to 0, \] and by the definition of $L_{\alpha}$
\[ z^{-\alpha/3} L_{\alpha}^{-1}(z) = \O(z^{-1/3}) \qquad \text{as } z
    \to 0. \]
Thus
\begin{align*} z^{-\alpha/3} \widetilde{N}_{\alpha}(z) L_{\alpha}^{-1}(n^{3/2}f(z))
    & = \left(\frac{n^{3/2} f(z)}{z} \right)^{\alpha/3}  \widetilde{N}_{\alpha}(z) 
    \left(n^{3/2} f(z)^{-\alpha/3}
 L_{\alpha}^{-1}(n^{3/2} f(z)) \right) \\
    & = \O(z^{-2/3}) \qquad \text{as } z \to 0.
    \end{align*}
It follows that the singularity of the left hand side at $z=0$ is removable and thus 
$E_n(z)$ is analytic in $D(0,r_0)$. 

Recall that by Lemma~\ref{lemma:localN} we have that 
$\det N_{\alpha} = \det \widetilde{N}_{\alpha} \equiv 1$. From \eqref{defL} we get that
\[ \det L_{\alpha}(z) = 3 i \sqrt{3} \, z^{-\alpha}. \]
Thus by \eqref{defE} 
\[ \det E_n(z) = \left( \frac{n^{3/2} f(z)}{z} \right)^{\alpha} \]
which is analytic and non-zero in a neighborhood of $z=0$. Thus $E_n^{-1}(z)$ is analytic in the
neighborhood $D(0,r_0)$ as well.

This completes the proof of the lemma.
\end{proof}

Having defined $f(z)$, $\tau(z)$ and $E_n(z)$ we then define the local parametrix $Q$ 
as in formula \eqref{Qform2}.

\section{Fourth transformation of the RH problem}

In the next transformation we define
\begin{equation}\label{defS}
	S (z)=\begin{cases}
	T (z)P (z)^{-1}, &  \text{for }z\in D(q^*,r_q),\\
	T (z)Q (z)^{-1}, & \text{for }z\in D(0,n^{-1/2}),\\
T (z)N_{\alpha}^{-1} (z), &  \text{elsewhere},
\end{cases}
\end{equation}
where we use  the matrix-valued functions $N_{\alpha}$ from \eqref{Nalpha},
$P$ constructed in Section \ref{sec:parametrixQ} in the fixed neighborhood $D(q^{*},r)$ 
of $q^{*}$, and $Q$ given by \eqref{Qform2} 
in the shrinking neighborhood $D(0,n^{-1/2})$ of the origin.

%%%%%%%%%%%%%%%%%%%%%%%%%%%%%%%%%%%%%%%%%%%%%%%%%%%%%%%%%%
\begin{figure}[t]
\centering 
\begin{overpic}[scale=0.9]%
{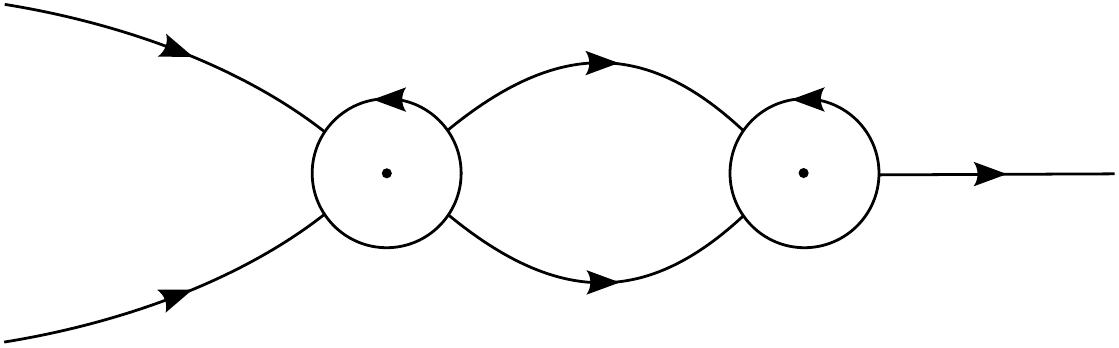}
%{rhpR}%
  \put(33.5,11){$0 $}
  \put(71,12){$q^*  $}
    \end{overpic}
\caption{Jump contour $\Sigma_{S}$ for the RH problem for $S$. The
disk around $0$ is shrinking with radius $n^{-1/2}$ as $n\to\infty$.
The contour $\Sigma_S$ is also the contour for the RH problem for $R$.
}
\label{contourS}
\end{figure}
%%%%%%%%%%%%%%%%%%%%%%%%%%%%%%%%%%%%%%%%%%%%%%%%%%%%%%%%%%

By construction, $S(z)$ is piece-wise analytic and has jumps across the contour $\Sigma_S$
shown in Figure~\ref{contourS}, with a possible isolated singularities at $0$ and $q^*$.
The singularity at $q^*$ is removable which follows from the properties of the Airy parametrix.
We now check that the origin is not a singularity of $S(z)$.

\begin{lemma}
The singularity of $S$ at $z=0$ is removable.
\end{lemma}

\begin{proof}
We give the proof for the case $\alpha > 0$.

Consider $z \to 0$ with $\Im z > 0$ and outside the lenses around $\Delta_{2}$ and $\Delta_{1}$. 
By the RH problem \ref{rhpforT} for $T$, we have that $T$ remains bounded there.
We  show that $Q^{-1}$ remains bounded as well. 

By \eqref{Qform2}, we have
\begin{multline*} 
	Q^{-1}(z) = 
	e^{-\frac{2nz}{3t(1-t)}}
	\begin{pmatrix} e^{n \lambda_1(z)} & 0 & 0 \\ 0 & e^{n \lambda_2(z)} & 0 \\ 0 & 0  & e^{n \lambda_3(z)}
	\end{pmatrix}  \\
	\begin{pmatrix} 1 & 0 & 0 \\ 0 & z^{-\alpha} & 0 \\ 0 & 0 & 1 \end{pmatrix}
	  \Phi_{\alpha}(n^{3/2} f(z); n^{1/2} \tau(z))^{-1}
     E_n(z)^{-1} 
    \end{multline*}
By Lemma~\ref{Eanalytic} we know that $E_n(z)^{-1}$ is analytic and thus bounded as $z \to 0$. Also the 
functions $\lambda_j$ are bounded as $z \to 0$. Also
\[
	\begin{pmatrix} 1 & 0 & 0 \\ 0 & z^{-\alpha} & 0 \\ 0 & 0 & 1 \end{pmatrix}
	  \Phi_{\alpha}(n^{3/2} f(z); n^{1/2} \tau(z))^{-1}
\]
is bounded as $z \to 0$ in the region under consideration because
of the condition 4. in the RH problem for $\Phi_{\alpha}$, see \eqref{Phiat02}
and the fact that $\det \Phi_{\alpha} = z^{-\alpha}$, see \eqref{detPhi}.

Wee conclude that $Q^{-1}$ remains bounded as $z \to 0$ in the region in the upper half-plane
outside of the lenses.

The other regions can be treated in a similar way and the lemma follows.
\end{proof}

We find the following RH problem for $S$
\begin{rhproblem} \label{rhpforS}
\begin{enumerate}
\item $S$ is analytic outside the contour $\Sigma_{S}$ shown in  Figure~\ref{contourS}.
\item On $\Sigma_S$ there is a jump relation
\begin{equation} \label{jumpsS} 
	S_+(z) = S_-(z) J_S(z) 
	\end{equation}
with jump matrix $J_S(z)$ given by
\begin{align}
	J_S(z) & = N_{\alpha}(z) P^{-1}(z), && \text{for $|z- q^*| = r_q$}, \label{JS1} \\
	J_S(z) & = N_{\alpha}(z) Q^{-1}(z), && \text{for $|z| = r_0 = n^{-1/2}$}, \label{JS2} \\
	J_S(z) & = N_{\alpha}(z) J_T(z) N_{\alpha}^{-1}(z), && \text{elsewhere on $\Sigma_S$} \label{JS3} 
\end{align}
\item $S(z) = I + \O(z^{-1})$ as $z \to \infty$.
\end{enumerate}
\end{rhproblem}

Recall now that we are interested in the limit \eqref{doublescaling} where $n \to \infty$,
$t \to t^*$ and $c^* \tau = n^{1/2} (t^* - t)$ remains fixed.
The jump matrices $J_S$ in \eqref{JS1}--\eqref{JS3} depend on $n$ and $t$, and we would
like that they tend to the identity matrix in the double scaling limit \eqref{doublescaling}. 
This turns out to be the case for the jump matrices \eqref{JS1} and \eqref{JS3}. However, this is not
the case for \eqref{JS2} as will be shown later on.

We start with the good jumps.
\begin{lemma} \label{lem:JumpsS}
In the limit \eqref{doublescaling} we have 
\begin{equation} \label{JSgood1}
	J_{S} (z) = I+\O \left(n^{-1} \right)\quad \text{ uniformly for } |z-q^*| = r_q,
\end{equation}
and for some $c>0$ (depending on $\alpha$), 
\begin{equation}\label{JSgood3}
	J_{S} (z) = I+\O \left( \frac{e^{-cn^{2/3}}}{1+|z|}  \right)\quad
 	\text{uniformly for $z \in \Sigma_{S}$ outside of the two circles.}
\end{equation}
\end{lemma}

\begin{proof}
The behavior \eqref{JSgood1} follows from \eqref{JS1} and the matching condition
\eqref{matchP} for $P$.

%We consider the behavior of $J_S$ on the rest of $\Sigma_S$ at a fixed distance from the origin. 
In view of Lemma~\ref{lambda-sign} and the asymptotic behavior of the $\lambda$ functions
we find that for some $c > 0$,
\begin{align} 
	\Re ( \lambda_3- \lambda_2)(z) & \geq c |z|^{1/2},  &&  z \in \Delta_2^{\pm} \setminus D(0,1), \label{boundLambda1} \\
	\Re (\lambda_1 - \lambda_2)(z) & \geq c,  && z \in \Delta_1^{\pm} \setminus (D(0, 1) \cup D(q^*, r_q)) \label{boundLambda2} \\
  \Re (\lambda_2 - \lambda_1)(z) & \geq c |z|, && z \in (q^* +r_q, \infty). \label{boundLambda3} 
\end{align}
According to  \eqref{JT2}, \eqref{JS3} and \eqref{boundLambda1}, for $z\in \Delta_2^{\pm}\setminus D(0,1)$,
$$
J_S(z)-I =  e^{\pm \alpha \pi i} z^{-\alpha}e^{n (\lambda_{2}-\lambda_{3})(z)} N_{\alpha}(z) E_{23} N_{\alpha}^{-1}(z)=\O\left(|z|^{-\alpha } e^{- c n |z|^{1/2}}\right)
$$
for some $c > 0$. 

Analogous considerations on the lips $\Delta_1^{\pm} \setminus (D(0, 1) \cup D(q^*, r_q))$ and on $(q^* +r_q, \infty)$, appealing to  formulas  \eqref{JT4}--\eqref{JT5} and \eqref{boundLambda2}--\eqref{boundLambda3}, show that there exists some $c > 0$, such that
\[ J_S(z) = I + \O\left(|z|^{|\alpha|} e^{- c n |z|^{1/2} } \right),   \qquad z \in \Sigma_S \setminus ( D(0,1) \cup D(q^*, r_q)). \]
%which easily implies \eqref{JSgood3} for $|z| \geq 1$.
What remains is to estimate $J_S(z)$ on the lips of the lenses near $0$ for $n^{-1/2} < |z| < 1$.

For $t=t^{*}$, we obtain from Lemma~\ref{lambdaat0} that there exists a constant $c_{1}>0$ such 
that (recall that $p=0$ when $t=t^{*}$)
\begin{align*} 
	\Re ( \lambda_3- \lambda_2)(z;t^*) & \geq c_1 |z|^{2/3},  && z \in \Delta_2^{\pm} \cap D(0,1), \\
	\Re (\lambda_1 - \lambda_2)(z;t^*) & \geq c_1 |z|^{2/3},  && z \in \Delta_1^{\pm} \cap D(0,1),  
\end{align*}
Moreover, \eqref{explambdaat0} and \eqref{f3g3t*} imply that $\lambda_{j}(z,t)=\lambda_{j} (z;t^{*})+z^{1/3}\, \O (t-t^{*})$ as $t\to
t^{*}$.  Thus,
\begin{align*} 
	\Re ( \lambda_3- \lambda_2)(z;t) & \geq c_1 |z|^{2/3} - c_2 |z|^{1/3} |t-t^*|,  && z \in \Delta_2^{\pm} \cap D(0,1), \\
	\Re (\lambda_1 - \lambda_2)(z;t) & \geq c_1 |z|^{2/3} - c_2 |z|^{1/3} |t-t^*|,  && z \in \Delta_1^{\pm} \cap D(0,1),  
\end{align*}
for some $c_{2}>0$. Since $t-t^{*} = \O(n^{-1/2})$ we conclude that
\begin{align*} 
	\Re ( \lambda_3- \lambda_2)(z;t) & \geq c_3 n^{-1/3},  && z \in \Delta_2^{\pm} \cap D(0,1), \, n^{-1/2} < |z| < 1, \\
	\Re (\lambda_1 - \lambda_2)(z;t) & \geq c_1 n^{-1/3},  && z \in \Delta_1^{\pm} \cap D(0,1), \, n^{-1/2} < |z| < 1,
\end{align*}
for some positive constant $c_{3}>0$ and $n$ large enough. 

Now, for $z\in \Delta_2^{\pm}\cap D(0,1)$, $|z|>n^{-1/2}$, and using \eqref{tildeNalpha}, we get
$$
J_S(z)-I =  e^{\pm \alpha \pi i} e^{n (\lambda_{2}-\lambda_{3})(z)} \widetilde N_{\alpha}(z) E_{23} \widetilde  N_{\alpha}^{-1}(z).
$$
By Lemma~\ref{lemma:localN}, $\widetilde N_{\alpha}(z) = \O \left( |z|^{-1/3} \right)$,  $ \widetilde N_{\alpha}^{-1}(z)=\O \left( |z|^{-1/3} \right)$ as $|z|\to 0$, so that
$$
J_S(z)-I =  \O\left(e^{- cn^{2/3}} \right)
$$
for some $c>0$. Analogous conclusion is obtained on $\Delta_1^{\pm} \cap D(0,1)$, $n^{-1/2} < |z| < 1$.

Gathering both estimates and replacing them by a weaker uniform bound we obtain  \eqref{JSgood3}.
\end{proof}

We next analyze the jump matrix \eqref{JS2} on $|z| = n^{-1/2}$
again in the double scaling limit \eqref{doublescaling}

We have $J_S = N_{\alpha} Q^{-1}$ where $N_{\alpha}$ is given by \eqref{Nalpha}
and $Q$ is given by \eqref{Qform2}. All notions that appear in
these formulas depend on $t$ or $n$ (or both). For example, $N_{\alpha}$ depends on
$t$ since the endpoint $q$ is varying with $t$, and tends to $q^*$ as $t \to t^*$.
Indeed, $q = q^* + \O(t-t^*)$.
Also the matrix $C_{\alpha}$ from \eqref{Nalpha} and the constants $K_1$, $K_2$, $K_3$
from \eqref{eq:F1F2F3} depend on $t$ and tend
to limiting values corresponding to the value $t^*$ at the same rate of $\O(t-t^*) = \O(n^{-1/2})$.
We denote the limiting values with $^*$:
\begin{equation} \label{CalphaKjstar} 
	C_{\alpha} = C_{\alpha}^* + \O(t-t^*), \qquad  K_j = K_j^* + \O(t-t^*), \quad j =1,2,3, 
	 \end{equation}
and these quantities appear in the formula \eqref{lemjumpS0} below.

\begin{proposition} \label{prop:JS}
In the limit \eqref{doublescaling} we have that
\begin{equation} \label{lemjumpS0}
  J_S(z) = I - \frac{h_n(z;t)}{ z} {\mathcal M}_{\alpha}^* 
  			+ \O(n^{-1/6}) 
  			\end{equation}
uniformly for $|z| = n^{-1/2}$, where
\begin{equation} \label{Malphastar}
	 \mathcal M_{\alpha}^* = C_{\alpha}^* \begin{pmatrix} K_1^* \\ K_2^* \\ K_3^* \end{pmatrix}
  		\begin{pmatrix} K_1^* & K_2^* & K_3^* \end{pmatrix} 
  		\begin{pmatrix} 1 & 0 & 0 \\ 0 & 0 & -i \\ 0 & -i & 0 \end{pmatrix}
  		\left(C_{\alpha}^*\right)^{-1}
\end{equation} and
\begin{equation}
\label{def:hn}
h_n(z;t)=\tau(z;t)\, \frac{n \tau(z;t)^2 + 9 \alpha}{9 c^* }.
\end{equation}
\end{proposition}

\begin{proof}
The local parametrix $Q$ from \eqref{Qform2} depends on both $t$ and $n$.
The functions $\lambda_j$ come from the Riemann surface and therefore depend
on $t$, but only in a mild way. Of more importance is the dependence of 
the functions $f(z) = f(z;t)$ and $\tau(z) = \tau(z;t)$ on $t$, see Lemma~\ref{behfandtau}.

From \eqref{behavtau} we have that 
\begin{equation} \label{taubounded}
	 n^{1/2} \tau(z;t) = \O(1),   
	 \end{equation}
in the double scaling limit \eqref{doublescaling} uniformly for $|z| = n^{-1/2}$.
[This is in fact the reason why we need the shrinking disk of radius $n^{-1/2}$.]

We also have that $n^{1/2} f(z;t)$ remains bounded as $|z| = n^{-1/2}$,
However, $n^{3/2} f(z;t)$ is growing in absolute value and is of order $n$
uniformly for $|z| = n^{-1/2}$. Therefore we can apply the asymptotic formulas
from Lemma~\ref{lemma:asymptPhi} to $\Phi_{\alpha}\left(n^{3/2} f(z); n^{1/2} \tau(z) \right)$
and we find for  $\Im z > 0$,
\begin{multline} \label{Phidoublescaling}
	\Phi_{\alpha}\left(n^{3/2} f(z); n^{1/2} \tau(z) \right) \\
	=  \frac{i}{\sqrt{3}} L_{\alpha}(n^{3/2} f(z))
		\left( I + \frac{M_{\alpha}^+(n^{1/2} \tau(z))}{n^{1/2} f(z)^{1/3}} + \O(n^{-2/3})\right) \\
		\times
			\begin{pmatrix} e^{\theta_1(n^{3/2} f(z); n^{1/2} \tau(z))} & 0 & 0 \\ 
			0 &   e^{\theta_2(n^{3/2} f(z); n^{1/2} \tau(z))} & 0 \\ 0 & 0 & e^{\theta_3(n^{3/2} f(z);n^{1/2} \tau(z))}
	\end{pmatrix} 
	\end{multline}
uniformly for $|z| = n^{-1/2}$, where we recall that $\theta_k(z; \tau)$, $k=1,2,3$, also depend on $\tau$,
see \eqref{thetak}.
By \eqref{thetak} and \eqref{thetaup} we actually have
\[ \theta_k\left(n^{3/2} f(z); n^{1/2} \tau(z)\right) = n \theta_k(f(z); \tau(z)) 
= n \lambda_k(z) - \frac{2nz}{3t(1-t)} \]
for $k=1,2,3$, and $\Im z > 0$, by our choice of $f(z)$ and $\tau(z)$.

Thus by \eqref{Qform2} and \eqref{doublescaling} we have
\[ Q(z) = E_n(z) \frac{i}{\sqrt{3}} L_{\alpha}(n^{3/2} f(z))
		\left( I + \frac{M_{\alpha}^+(n^{1/2} \tau(z))}{n^{1/2} f(z)^{1/3}} + \O(n^{-2/3})\right)\begin{pmatrix} 1 & 0 & 0 \\
    0 & z^{\alpha} & 0 \\ 0 & 0 & 1  \end{pmatrix},
		 \]
and then by inserting the definition \eqref{defE} of $E_n(z)$, we obtain
\begin{equation}
\label{eq:asymptQinfty}
Q(z) = N_{\alpha}(z) \begin{pmatrix} 1 & 0 & 0 \\
    0 & z^{-\alpha} & 0 \\ 0 & 0 & 1  \end{pmatrix} 
    \left(I + \frac{M_{\alpha}^+(n^{1/2}\tau(z))}{n^{1/2} f(z)^{1/3}}+ \O(n^{-2/3})\right) \begin{pmatrix} 1 & 0 & 0 \\
    0 & z^{\alpha} & 0 \\ 0 & 0 & 1  \end{pmatrix},
\end{equation}
uniformly for $|z| = n^{-1/2}$ with $\Im z > 0$. For $|z| = n^{-1/2}$ with $\Im z < 0$ we obtain
the same formula \eqref{eq:asymptQinfty} but with $M_{\alpha}^+$ replaced by $M_{\alpha}^-$.

Then for $|z| = n^{-1/2}$,
\begin{equation} \label{JSasymptotics1} 
J_S^{-1}(z) = Q(z) N_{\alpha}^{-1}(z) = 
 I + \widetilde N_{\alpha}(z) \frac{M_{\alpha}^{\pm}(n^{1/2}\tau(z))}{n^{1/2} f(z)^{1/3}} 
 	\widetilde N_{\alpha}^{-1}(z) + \O(n^{-1/3})
 \end{equation}
where we recall the definition \eqref{tildeNalpha} of $\widetilde N_{\alpha}$. The entries
of $\widetilde N_{\alpha}(z)$ and its inverse are of order $|z|^{-1/3} = n^{1/6}$ 
by \eqref{asymptotics0}. Therefore the  error term has gone up from $\O(n^{-2/3})$ in 
\eqref{eq:asymptQinfty} to $\O(n^{-1/3})$ in \eqref{JSasymptotics1}.

In the evaluation of $\widetilde N_{\alpha}(z) M_{\alpha}^{\pm}(n^{1/2}\tau(z)) 
 	\widetilde N_{\alpha}^{-1}(z)$ we encounter the following matrix
\begin{equation} \label{Malpha}
	\mathcal M_{\alpha} = C_{\alpha}  \begin{pmatrix}
K_1 \\ K_2 \\ K_3
\end{pmatrix} \begin{pmatrix} K_1 & K_2 & K_3 \end{pmatrix}  \begin{pmatrix}
1 & 0 & 0 \\
0 & 0 & -i \\
0 & -i & 0
\end{pmatrix} C_\alpha^{-1}
\end{equation}
which is a $3 \times 3$ rank one matrix depending on $t$, but not on $z$.
The matrix $\mathcal M_{\alpha}$ is in fact nilpotent,
\begin{equation} \label{Malpha2}
	\mathcal M_{\alpha}^2 = 0,
	\end{equation}
 which follows from \eqref{Malpha} and the property \eqref{relationK} of the numbers $K_j$.
From \eqref{CalphaKjstar} and \eqref{Malphastar} we also find that
\begin{equation} \label{Malpha3} 
	\mathcal M_{\alpha} = \mathcal M_{\alpha}^* + \O(t-t^*) \qquad \textrm{ as } t \to t^*. 
	\end{equation}

We first prove the following  lemma. 	

\begin{lemma}
We have for $\Im z > 0$
\begin{equation} \label{quant1}
\widetilde N_\alpha(z)  \begin{pmatrix} \omega^2 & 0 & 0 \\ 
0 & \omega & 0 \\
0 & 0 & 1
\end{pmatrix}
\widetilde N_\alpha^{-1}(z) 
 = 3 c^{-4/3} \mathcal M_\alpha  z^{-2/3} + \O\left( z^{-1/3 }\right)
\end{equation}
and
\begin{multline}\label{quant2}
\widetilde N_\alpha(z) \begin{pmatrix} \omega^{-\alpha/2} & 0 & 0 \\ 
0 & \omega^{\alpha/2} & 0 \\
0 & 0 & 1
\end{pmatrix}\begin{pmatrix}  0 & -\omega & 1 \\ 
\omega^2 &  0 & -1 \\
-\omega^2 & \omega & 0
\end{pmatrix} \begin{pmatrix} \omega^{\alpha/2} & 0 & 0 \\ 
0 & \omega^{-\alpha/2} & 0 \\
0 & 0 & 1
\end{pmatrix}
\widetilde N_\alpha^{-1}(z) \\
= -3 \sqrt{3} i c^{-4/3} \mathcal M_\alpha z^{-2/3} + \O\left( z^{-1/3 }\right)
\end{multline}
as $z \to 0$.
\end{lemma}

\begin{proof}
We obtain \eqref{quant1} from \eqref{Nalpha}, \eqref{asymptoticsN0at0}, Lemma~\ref{lemma:inverseN_0},
and the fact that 
\[
	\begin{pmatrix} \omega^2 &  \omega & 1 \end{pmatrix} \begin{pmatrix} \omega^2 & 0 & 0 \\ 
0 & \omega & 0 \\
0 & 0 & 1
\end{pmatrix} \begin{pmatrix} \omega^2  \\ 
  \omega   \\
  1
\end{pmatrix}= 3.
\]

From \eqref{Nalpha}, we obtain that the left-hand side of \eqref{quant2} is equal to
\begin{equation*}
C_\alpha N_0(z) \begin{pmatrix} \xi_1^{\alpha}(z) & 0 & 0 \\ 
0 & \xi_2^{\alpha}(z) & 0 \\
0 & 0 & \xi_3^{\alpha}(z)
\end{pmatrix}\begin{pmatrix}  0 & -\omega & 1 \\ 
\omega^2 &  0 & -1 \\
-\omega^2 & \omega & 0
\end{pmatrix} \begin{pmatrix} \xi_1^{-\alpha}(z) & 0 & 0 \\ 
0 & \xi_2^{-\alpha}(z) & 0 \\
0 & 0 & \xi_3^{-\alpha}(z)
\end{pmatrix}
 N_0^{-1}(z)C_\alpha^{-1},
\end{equation*}
with functions
\[
\xi_1(z) = \omega^{-1/2} e^{G_1(z)}, \quad \xi_2(z)=\omega^{1/2} z^{-1} e^{G_2(z)}, \quad \xi_3(z)= e^{G_3(z)},
\]
with $G_j(z) = r_j(w_j(z))$ defined in \eqref{G1G2G3def}--\eqref{r123def}. Using these
expressions and Lemma~\ref{wat0}, we find the remarkable fact that 
\[
\frac{\xi_j(z)}{\xi_k(z)} =1 + \O\left( z^{1/3}\right), \quad z\to 0.
\]

Then, using \eqref{asymptoticsN0at0} and the fact that 
\[ \begin{pmatrix} \omega^2 & \omega & 1 \end{pmatrix} 
	\begin{pmatrix}  0 & -\omega & 1 \\ 
 \omega^2 &  0 & -1 \\
 -\omega^2 & \omega & 0
\end{pmatrix} \begin{pmatrix} \omega^2  \\ 
  \omega   \\
  1
\end{pmatrix}= -3 \sqrt{3}i,
\]
we obtain \eqref{quant2} 
\end{proof}

We continue with the proof of Proposition \ref{prop:JS}.
 
From the lemma and the formula \eqref{Mplus} for $M_{\alpha}^+$, we obtain
\begin{multline} \widetilde N_{\alpha}(z) M_{\alpha}^{+}(n^{1/2}\tau(z)) 
 	\widetilde N_{\alpha}^{-1}(z) \\
 	= \left(3 \frac{n^{1/2} \tau(z) \left( n \tau^2(z) + 9 \alpha - 9\right)}{27}
 		 -3 \sqrt{3} i \frac{i n^{1/2} \tau(z)}{3 \sqrt{3}} \right) 
 		 c^{-4/3} \mathcal M_\alpha z^{-2/3} + \O\left( z^{-1/3 }\right) \\
 		 = n^{1/2} \frac{\tau(z)(n \tau^2(z) + 9\alpha)}{9 c^{4/3}}
 		 \mathcal M_\alpha z^{-2/3} + \O\left( z^{-1/3 }\right)
\end{multline}

Inserting this into \eqref{JSasymptotics1} we get for $|z| = n^{-1/2}$ and $\Im z > 0$
\begin{equation} \label{JSasymptotics2} 
J_S^{-1}(z) =
 I + \frac{\tau(z)(n \tau^2(z) + 9\alpha)}{9 c^{4/3} f(z)^{1/3}} \mathcal M_{\alpha} z^{-2/3} 
 	+ \frac{\O(z^{-1/3})}{n^{1/2} f(z)^{1/3}}  + \O(n^{-1/3})
 \end{equation}
A similar analysis for $\Im z < 0$ will show that the same formula \eqref{JSasymptotics2}
also holds for $\Im z< 0$.

Note that for $|z| = n^{-1/2}$ we have  $\frac{\O(z^{-1/3})}{n^{1/2} f(z)^{1/3}} = \O(n^{-1/6})$,
where we use that $f(z)$ is a conformal map. In fact, by \eqref{f3g3at0} and \eqref{defftau} 
\[ f(z;t) = f'(0;t) z + \O(z^2) =
	\left((c^*)^{-1} + \O(t-t^*)  \right)z + \O(z^2) \]
which implies that in the double scaling limit \eqref{doublescaling}
\[ f(z;t) = (c^*)^{-1} z + \O(n^{-1}) \]
for $|z|= n^{-1/2}$. 	 Using also \eqref{Malpha2}
we then obtain from \eqref{JSasymptotics2} that
\begin{equation} \label{JSasymptotics3} 
J_S(z) = I - \frac{\tau(z)(n \tau^2(z) + 9\alpha)}{9 c^* z^{1/3}} \mathcal M_{\alpha} z^{-2/3} 
 	  + \O(n^{-1/6}), \qquad |z| = n^{-1/2}
 \end{equation}
which implies \eqref{lemjumpS0} in view of \eqref{Malpha3}. This completes the proof of the Proposition \ref{prop:JS}.
 \end{proof}

Now recall that in the double scaling limit \eqref{doublescaling}
we have that $n^{1/2} \tau(z;t)$ remains bounded for $|z| = n^{-1/2}$, see 
also \eqref{taubounded}. Then also
\[  \frac{\tau(z;t)(n \tau^2(z;t) +9\alpha)}{c^* z} = \O(1) \]
and it follows that the term $ \frac{h_n(z;t)}{ z}  \mathcal M_{\alpha}^* $ in \eqref{lemjumpS0}
remains bounded for $|z| = n^{-1/2}$, but does not tend to $0$ as $n \to \infty$.
Therefore the jump matrix $J_S$ on $|z| = n^{-1/2}$ does not tend to the identity
matrix as $n \to \infty$.

\section{Final transformation} \label{final}

We need one more transformation. What will help us in the final transformation
is the identity \eqref{Malpha2} for $\mathcal M_{\alpha}$, which also holds for
the limiting value, namely $\left(\mathcal M_{\alpha}^* \right)^2 = 0$.
The final transformation $S \mapsto R$ is similar to the one in \cite{DKZ}
and is defined as follows.

\begin{definition} \label{defR}
With the notation \eqref{def:hn} we define for $z \in \mathbb C \setminus \Sigma_S$,
\begin{align} \label{defR1}
		R(z)  & =  S(z) \left( I  - \frac{h_n(0;t)}{ z} \,  \mathcal M_\alpha^* \right), 
		&& \quad |z| > n^{-1/2}, \\
	 \label{defR2}
		R(z) & = 	S(z) \left( I  +   \frac{h_n(z;t)-h_n(0;t)}{ z} \, \mathcal M_\alpha^* \right), 
		&&  \quad |z| < n^{-1/2}.
	\end{align}
\end{definition}

Note that the transformation $S \mapsto R$ is a global transformation which modifies
$S$ in every part of the complex plane. Then $R$ satisfies the following
RH problem on the contour $\Sigma_R = \Sigma_S$, see Figure~\ref{contourS}.

\begin{rhproblem} \label{rhpforR}
\begin{enumerate}
\item $R$ is defined and analytic in $\mathbb C \setminus \Sigma_R$.
\item On $\Sigma_R$ we have the jump 
\begin{equation} \label{jumpR}
	R_+ = R_- J_R 
\end{equation}
with
\begin{equation} \label{JR1}
 J_R(z) =  \left( I  + \frac{h_n(0;t)}{ z} \, \mathcal M_\alpha^* \right)  
		J_S(z) \left( I  + \frac{h_n(z;t)-h_n(0;t)}{ z} \, \mathcal M_\alpha^* \right), 
		\end{equation}		
for $|z| = n^{-1/2}$, 
and
\begin{equation} \label{JR2}
	J_R(z) =  \left( I  + \frac{h_n(0;t)}{ z} \, \mathcal M_\alpha^* \right)  
		J_S(z) \left( I  - \frac{h_n(0;t)}{ z} \, \mathcal M_\alpha^* \right),
		\end{equation}
		elsewhere on $\Sigma_R$.
\item $R(z) = I + \O(1/z)$ as $z \to \infty$.
\end{enumerate}
\end{rhproblem}

All properties in the RH problem \ref{rhpforR} follow easily fom the RH problem \ref{rhpforS}
for $S$ and the definition \eqref{defR1}--\eqref{defR2}. For \eqref{JR1} and \eqref{JR2}
one also uses \eqref{Malpha2} and \eqref{Malpha3}  which imply that for every
constant $\gamma$,
\[ \left(I - \gamma \mathcal M_{\alpha}^* \right)^{-1} = I + \gamma M_{\alpha}^*. \]

Under the transformation $S \mapsto R$ the jumps on 
the part of $\Sigma_R$ outside of the circle $|z| = n^{-1/2}$
are not essentially affected. We have the same estimates as
in Lemma~\ref{lem:JumpsS}:
\begin{lemma} \label{lem:JumpsR1}
In the limit \eqref{doublescaling} we have 
\begin{equation} \label{JRgood1}
	J_R (z) = I+\O \left(n^{-1} \right)\quad \text{ uniformly for } |z-q^*| = r_q. 
\end{equation}
and
\begin{equation} \label{JRgood3}
	J_R (z) = I+\O \left(\frac{e^{-cn^{2/3}}}{1+|z|} \right)\quad
 	\text{uniformly for $z \in \Sigma_R$ outside of the two circles.}
\end{equation}
\end{lemma}
\begin{proof}
%To obtain \eqref{JRgood3} it is important that $(\mathcal M_{\alpha})^2 = 0$.
By \eqref{JSgood1} and \eqref{JR2}, for $ |z-q^*| = r_q$,
\begin{align*}
J_R & =\left( I  + \frac{h_n(0;t)}{ z} \, \mathcal M_\alpha^* \right)  
		\left( I+\O \left(n^{-1} \right) \right) \left( I  - \frac{h_n(0;t)}{ z} \, \mathcal M_\alpha^* \right) \\
		& = I + \left( I  + \frac{h_n(0;t)}{ z} \, \mathcal M_\alpha^* \right)    \O \left(n^{-1}   \right) \left( I  - \frac{h_n(0;t)}{ z} \, \mathcal M_\alpha^* \right)\\
		& = I +  \O \left(n^{-1}   \right) .
\end{align*}
Analogous calculations yield \eqref{JRgood3}. In both cases the fact that $(\mathcal M_{\alpha}^*)^2 = 0$ is crucial.
\end{proof}

\begin{lemma} \label{lem:JumpsR2}
In the limit \eqref{doublescaling} we have 
\begin{equation} \label{JRgood2}
	J_R (z) = I+\O \left(n^{-1/6} \right)\quad \text{ uniformly for } |z| = n^{-1/2}. 
\end{equation}
\end{lemma}
\begin{proof}
Using \eqref{JR1} and \eqref{JSasymptotics3} and the fact that $(\mathcal M_{\alpha}^*)^2 = 0$, direct calculation yields
$$
J_R (z) = I+  \left( I  + \frac{h_n(0;t)}{ z} \, \mathcal M_\alpha^* \right)  
		\O \left(n^{-1/6} \right) \left( I  + \frac{h_n(z;t)-h_n(0;t)}{ z} \, \mathcal M_\alpha^* \right),
$$
and \eqref{JRgood2} follows.
\end{proof}

As a result of the estimates on $J_R$ we may now conclude that
in the double scaling limit \eqref{doublescaling}
\begin{equation} \label{asympR}
	R(z) = I + \O \left( \frac{1}{n^{1/6}(1 + |z|)} \right) \qquad \text{as } n \to \infty 
\end{equation}
uniformly for $z \in \mathbb C \setminus \Sigma_R$.
See \cite[Appendix A]{BK4} for arguments that justify this, also in a situation
of varying contours.

\section{The limiting kernel}
\subsection{Expression for the critical kernel}\label{mainproof}
We start from \eqref{defK}, which gives the correlation kernel $K_{n}
(x,y;t)$ in terms of the solution of the RH problem for $Y$. Following
the transformation $Y \mapsto X \mapsto U \mapsto T$, we find that for $x,y>0$ and
$x,y\in (0,q)$, 
\begin{equation}\label{kernelanalysis}
K_{n}(x,y;t) = \frac{1}{2\pi i(x-y)}
    \begin{pmatrix} - e^{-n\lambda_{1,+}(y)} & y^{\alpha} e^{-n \lambda_{2,+}(y)} & 0 \end{pmatrix}
    T_+^{-1}(y) T_+(x) \begin{pmatrix} e^{n \lambda_{1,+}(x)} \\ x^{-\alpha} e^{n \lambda_{2,+}(x)} \\ 0 \end{pmatrix}.
\end{equation}

For $z$ inside the disk of radius $n^{-1/2}$,  we have by 
\eqref{defS}, \eqref{Qform2}, 
\eqref{defLambda2}, and \eqref{defR2}, that
\begin{equation}\label{expT}
\begin{aligned}
	T (z)&=S (z)Q (z)\\
	&=R (z) \left( I  - \ds \frac{\tau(z;t)(n\tau^2(z;t)+9\alpha)- \tau(0;t)(n \tau^2(0;t) + 9\alpha) }{9 c^*  z} \mathcal M_\alpha^* \right) \\
	& \qquad \qquad \times
	E_n (z)\Phi_{\alpha}(n^{3/2}f (z;t);n^{1/2}\tau (z;t)) \Lambda_n(z) e^{\frac{2nz}{3t (1-t)}}, 
\end{aligned}
\end{equation}
Thus, if $0<x,y<n^{-1/2}$, we get by plugging \eqref{expT} into
\eqref{kernelanalysis}, 
\begin{multline} \label{Kn}
K_{n}(x,y;t) = \frac{e^{\frac{2n (x-y)}{3t (1-t)}}}{2\pi i(x-y)}
    \begin{pmatrix} - 1 & 1 & 0 \end{pmatrix}
    \Phi_{\alpha,+}^{-1} (n^{3/2}f (y;t);n^{1/2}\tau (y;t))  \\
   	\times E_n^{-1}(y) \left( I  + \ds \frac{\tau(y;t)(n\tau^2(y;t)+9\alpha)- \tau(0;t)(n \tau^2(0;t) + 9\alpha) }{9 c^*  y} \mathcal M_\alpha^* \right) 
   	R^{-1} (y) \\
   	 \times 
    R (x) \left( I  -  \ds \frac{\tau(x;t)(n\tau^2(x;t)+9\alpha)- \tau(0;t)(n \tau^2(0;t) + 9\alpha) }{9 c^* x} \mathcal M_\alpha^* \right)	
    E_n(x) \\ \times
    \Phi_{\alpha,+}(n^{3/2}f (x;t);n^{1/2}\tau (x;t)) 
	\begin{pmatrix} 1 \\ 1 \\ 0 \end{pmatrix}
\end{multline}
which is an exact formula.

Now we take the double scaling limit $n \to \infty$, $t \to t^*$ such that
\begin{equation} \label{doublescaling2} 
	c^* \tau = n^{1/2}(t^* - t) \qquad \text{remains fixed.} 
	\end{equation}
We also replace $x$ and $y$ in \eqref{Kn} by
\begin{equation} \label{xnyn} 
	x_n = \frac{c^* x}{n^{3/2}}, \qquad y_n = \frac{c^* y}{n^{3/2}}
	\end{equation}
with $x, y > 0$ fixed. For $n$ large enough we then have that $x_n$ and $y_n$ are less than $n^{-1/2}$.
We study how the various factors in \eqref{Kn} behave in this limit.

\begin{lemma}\label{lem-f-tau}
Let $x, y >$ be fixed.  Then we have in the double scaling limit \eqref{doublescaling2} with
$x_n$ and $y_n$ given by \eqref{xnyn} 
\begin{equation}\label{limfxn}
	n^{3/2} f (x_{n};t) = x (1 + \O(n^{-1/2})), \qquad n^{3/2} f(y_n; t) = y(1 + \O(n^{-1/2})),
\end{equation}
and
\begin{equation}\label{limtauxn}
	n^{1/2}\tau (x_{n};t) = \tau + \O(n^{-1/2}), \qquad n^{1/2} \tau(y_n; t) = \tau + \O(n^{-1/2})
\end{equation}
as $n \to \infty$.
\end{lemma}
\begin{proof}
By \eqref{defftau}, \eqref{f3g3t*} and \eqref{expf3g3at0} we have
\[ f(z;t) = z [f_3(z;t)]^{3/2} = z \left[\left(c^{*}\right)^{-2/3} + \O 
(z)+\O (t-t^{*}) \right]^{3/2} =
		\frac{z}{c^*} \left[1 + \O(z) + \O(t-t^*)\right] \]
	as $z \to 0$ and $t \to t^*$.
This readily implies \eqref{limfxn}.	

Again by \eqref{defftau},  \eqref{f3g3t*} and \eqref{expf3g3at0} we have
\[ \tau(z;t) = \frac{g_3(z;t)}{f_3(z;t)^{1/2}} = \left(c^{*}\right)^{1/3} g_3(z;t) 
\left[1 + \O(z) + \O(t-t^*)\right] \]
	as $z \to 0$ and $t \to t^*$.
Then from the definitions in Lemmas \ref{wat0}, \ref{zetaat0},
and \ref{lambdaat0} it is not difficult to verify that $g_3(z;t)$ is analytic in
both arguments with $g_3(0,t^*) = 0$, so that
\[ g_3(z,t) = \left[ \frac{\partial g_3}{\partial t}(0,t^*) \right](t-t^*) + \O(z) + \O(t-t^*)^2 \]
as $z \to 0$ and $t \to t^*$.
By \eqref{f3g3at0} we have $g_{3}(0;t)= 3pc^{-4/3}$ 
and using the dependence of $p$ and $c$ on $t$ as given in \eqref{defcp} 
we find, after some calculations, 
that 
\[ \frac{\partial g_3}{\partial t}(0,t^*) = - \left(c^{*}\right)^{-4/3}. \]
Hence $\tau(z;t) =   \left(c^{*}\right)^{-4/3} (t^*-t) + \O(z) + \O((t-t^*)^2)$
and 
\[ n^{1/2} \tau(x_n;t) = n^{1/2} \left(c^{*}\right)^{-1}(t^*-t) + \O(n^{-1}) + n^{1/2} \O((t-t^*)^2) \]
which by \eqref{doublescaling2} indeed leads to \eqref{limtauxn}.
\end{proof}

\begin{lemma} \label{lem-E-Mstar-R}
Under the same assumptions as in Lemma~\ref{lem-f-tau}
\begin{multline} \label{limRE}
	E_n^{-1} (y_{n}) 	
	 \left( I  + \ds \frac{\tau(y_n;t)(n\tau^2(y_n;t)+9\alpha)- \tau(0;t)(n \tau^2(0;t) + 9\alpha) }{9 c^*  y_n} \mathcal M_\alpha^* \right) 
   	R^{-1} (y_n) \\
   	 \times 
    R (x_n) \left( I  -  \ds \frac{\tau(x_n;t)(n\tau^2(x_n;t)+9\alpha)- \tau(0;t)(n \tau^2(0;t) + 9\alpha) }{9 c^* x_n} \mathcal M_\alpha^* \right)	
    E_n(x_n) \to I.
\end{multline}
\end{lemma}

\begin{proof}
% This follows as in the proof of \cite[Lemma 10.1]{BK4}.
For $z = \O(n^{-3/2})$ we have by Cauchy's theorem and \eqref{asympR}
\[ R'(z) = \frac{1}{2\pi i} \int_{|s| = n^{-1/2}} \frac{R(s) - I}{(s-z)^2} ds
	= \O(n^{1/3}) \qquad \text{as } n \to \infty.  \]
Therefore $R(x_n) - R(y_n) = \O((x_n-y_n) n^{1/3}) = \O(n^{-7/6})$ and so
\begin{equation} \label{Rynxn} 
	R^{-1}(y_n) R(x_n) = I + R^{-1}(y_n)(R(x_n) - R(y_n)) = I + \O(n^{-7/6}) \qquad \text{as } n \to \infty. 
	\end{equation}
where we use that $R^{-1}(y_n)$ remains bounded as $n\to \infty$, which also follows
from \eqref{asympR}.

Let us write 
\begin{equation} \label{rhonxnt} 
	\rho_n(x_n,t) : =  \frac{\tau(x_n;t)(n\tau^2(x_n;t)+9\alpha)- \tau(0;t)(n \tau^2(0;t) + 9\alpha) }{9 c^* x_n} 
	\end{equation}
and similarly for $\rho_n(y_n,t)$. Then explicit calculations 
(done with the help of Maple) show that
\begin{equation} \label{rhonestimate1}
	\rho_n(x_n,t) = \frac{\tau^2 + 3 \alpha}{36 (c^*)^2} -
	\frac{\tau^3 (a-2)}{54(a+1)(c^*)^2} n^{-1/2} + \O(n^{-1}) 
	\end{equation}
and similarly for $\rho_n(y_n,t)$. Thus
\begin{equation} \label{rhonestimate2} 
	\rho_n(y_n,t) - \rho_n(x_n,t) = \O(n^{-1})  \qquad \text{as } n \to \infty.
	\end{equation}

Using \eqref{Rynxn}, \eqref{rhonestimate1}, \eqref{rhonestimate2} 
and the fact that $(\mathcal M_{\alpha}^*)^2 = 0$, we see that \eqref{limRE} will 
follow from the following three estimates
\begin{align} \label{Enestimate1}
	E_n^{-1}(y_n) E_n(x_n) & = I + \O(n^{-1/2}) \\ \label{Enestimate2}
	E_n^{-1}(y_n) \O(n^{-7/6}) E_n(x_n) & = \O(n^{-1/6}) \\ \label{Enestimate3}
	E_n^{-1}(y_n) \mathcal M_{\alpha}^* E_n(x_n) & = \O(n^{1/2})
	\end{align}
	as $n \to \infty$.

By \eqref{defE}, the analytic factor $E_n(z)$ depends on $n$ mainly via the argument $n^{3/2} f(z;t)$ of
$L_{\alpha}^{-1}$. By  \eqref{tildeNalpha}, \eqref{defL}, and \eqref{defE}, 
we can factor out the dependence on $n^{3/2} f(z;t)$, and we obtain
\begin{equation} \label{En2} 
	E_n(z) = \left( \frac{n^{3/2} f(z;t)}{z} \right)^{\alpha/3} F_{\alpha}(z)
	 \diag \begin{pmatrix} \left( \frac{n^{3/2} f(z;t)}{z} \right)^{-1/3} & 1 &  
	 	\left( \frac{n^{3/2} f(z;t)}{z} \right)^{1/3} \end{pmatrix} 
	 	\end{equation}
where
\begin{equation} \label{Falpha} 
	F_{\alpha}(z) = \frac{1}{i \sqrt{3}} \widetilde{N}_{\alpha}(z) 
	\begin{pmatrix} e^{-\alpha \pi i/3} & 0 & 0 \\ 0 & e^{\alpha \pi i/3} & 0 \\ 0 & 0 & 1 \end{pmatrix}
		\begin{pmatrix} \omega^2 & 1 & \omega  \\ \omega & 1 & \omega^2 \\ 1 & 1 & 1 \end{pmatrix}
		\begin{pmatrix} z^{-1/3} & 0 & 0 \\ 0 & 1 & 0 \\ 0 & 0 & z^{1/3} \end{pmatrix} 
		\end{equation}
is analytic around $z=0$ and depends on $n$ in a very mild way only, namely via the
dependence of $\widetilde{N}_{\alpha}(z)$ on the endpoint $q$ which is only slightly moving with $n$.

The scalar factor in \eqref{En2} will appear in the products \eqref{Enestimate1}--\eqref{Enestimate3}
in the form
\begin{equation} \label{ratiofxnfyn} 
	\left( \frac{n^{3/2} f(x_n;t)}{x_n} \frac{y_n}{n^{3/2} f(y_n;t)} \right)^{\alpha/3} = 1 + \O(n^{-1/2}) 
	\end{equation}
where the estimate follows from \eqref{xnyn} and \eqref{limfxn}.
Thus by \eqref{En2}
\begin{multline*} 
	E_n^{-1}(y_n) E_n(x_n) = 
	(1 + \O(n^{-1/2})) 
	 \diag \begin{pmatrix} \left( \frac{n^{3/2} f(y_n;t)}{y_n} \right)^{1/3} & 1 &  
	 	\left( \frac{n^{3/2} f(y_n;t)}{y_n} \right)^{-1/3} \end{pmatrix} \\
	 	\times
		  F_{\alpha}^{-1}(y_n) F_{\alpha}(x_n) 
	 \diag \begin{pmatrix} \left( \frac{n^{3/2} f(x_n;t)}{x_n} \right)^{-1/3} & 1 &  
	 	\left( \frac{n^{3/2} f(x_n;t)}{x_n} \right)^{1/3} \end{pmatrix} 
	 	\end{multline*}
where $F_{\alpha}^{-1}(y_n) F_{\alpha}(x_n) = I + \O(x_n-y_n) = I + \O(n^{-3/2})$.
Since the two entries $\left(\frac{n^{3/2} f(y_n;t)}{y_n} \right)^{1/3}$ and $\left(\frac{n^{3/2} f(x_n;t)}{x_n} \right)^{1/3}$
in the diagonal matrices grow like $\O(n^{1/2})$ we find \eqref{Enestimate1}, where we also use \eqref{ratiofxnfyn}.

We similarly have
\begin{multline*} 
	E_n^{-1}(y_n) \O(n^{-7/6}) E_n(x_n) = 
	(1 + \O(n^{-1/2})) 
	 \diag \begin{pmatrix} \left( \frac{n^{3/2} f(y_n;t)}{y_n} \right)^{1/3} & 1 &  
	 	\left( \frac{n^{3/2} f(y_n;t)}{y_n} \right)^{-1/3} \end{pmatrix} \\
	 	\times
		  F_{\alpha}^{-1}(y_n) \O(n^{-7/6}) F_{\alpha}(x_n) 
	 \diag \begin{pmatrix} \left( \frac{n^{3/2} f(x_n;t)}{x_n} \right)^{-1/3} & 1 &  
	 	\left( \frac{n^{3/2} f(x_n;t)}{x_n} \right)^{1/3} \end{pmatrix}.
	 	\end{multline*}
Since $F_{\alpha}^{-1}(y_n)$ and $F_{\alpha}(x_n)$ remain bounded as $n \to \infty$,
and the two diagonal matrices are $\O(n^{1/2})$ we obtain the estimate \eqref{Enestimate2}.

To prove the final estimate \eqref{Enestimate3} we note that
\begin{multline} \label{EnMalphaEn} 
	E_n^{-1}(y_n) \mathcal M_{\alpha}^* E_n(x_n) = 
	(1 + \O(n^{-1/2})) 
	 \diag \begin{pmatrix} \left( \frac{n^{3/2} f(y_n;t)}{y_n} \right)^{1/3} & 1 &  
	 	\left( \frac{n^{3/2} f(y_n;t)}{y_n} \right)^{-1/3} \end{pmatrix} \\
	 	\times
		  F_{\alpha}^{-1}(y_n) \mathcal M_{\alpha}^* F_{\alpha}(x_n) 
	 \diag \begin{pmatrix} \left( \frac{n^{3/2} f(x_n;t)}{x_n} \right)^{-1/3} & 1 &  
	 	\left( \frac{n^{3/2} f(x_n;t)}{x_n} \right)^{1/3} \end{pmatrix}
	 	\end{multline}
which would lead to $\O(n)$ as $n \to \infty$ if we use the same estimates as above.
However, by the form of the right hand side of \eqref{EnMalphaEn}, we see that it is only the
$(1,3)$ entry $E_n^{-1}(y_n) \mathcal M_{\alpha}^* E_n(x_n)$ that could grow like $\O(n)$.
The other entries are $\O(n^{1/2})$ as claimed in \eqref{Enestimate3}.

We have by \eqref{EnMalphaEn} and \eqref{Falpha}
\begin{multline*} 
	\left[ E_n^{-1}(y_n) \mathcal M_{\alpha}^* E_n(x_n) \right]_{1,3}
		= \left( \frac{n^{3/2} f(x_n;t)}{x_n} \right)^{1/3}  \left( \frac{n^{3/2} f(y_n;t)}{y_n} \right)^{1/3} 
		\left[ F_{\alpha}^{-1}(y_n) \mathcal M_{\alpha}^* F_{\alpha}(x_n) \right]_{1,3}   \\
		= \left( n^{3/2} f(x_n;t)  \right)^{1/3}  \left( n^{3/2} f(y_n;t) \right)^{1/3}  
		\O \left( \left\| \widetilde{N}_{\alpha}^{-1}(y_n) \mathcal M_{\alpha}^* \widetilde{N}_{\alpha}(x_n) \right\| \right)  \\
		= \O \left( \left\| \widetilde{N}_{\alpha}^{-1}(y_n) \mathcal M_{\alpha}^* \widetilde{N}_{\alpha}(x_n) \right\| \right)
			\end{multline*}
where in the last step we used \eqref{limfxn}. Both matrices
$\widetilde{N}_{\alpha}^{-1}(y_n)$ and $\widetilde{N}_{\alpha}(x_n)$ grow like $n^{1/2}$, see \eqref{Nalphaat0}.
However 
\[ \mathcal M_{\alpha}^* \widetilde{N}_{\alpha}(x_n) = \O(1) \qquad \text{as } n \to \infty \]
which follows from \eqref{tildeNalpha} and the fact that  by \eqref{asymptoticsN0at0}, \eqref{Malpha}, 
\begin{align*} 
	\mathcal M_{\alpha}^* C_{\alpha} N_0(z) & =
	C_{\alpha} \begin{pmatrix} K_1 \\ K_2 \\ K_3 \end{pmatrix}
		\begin{pmatrix} K_1 & K_2 & K_3 \end{pmatrix} \begin{pmatrix} 1 & 0 & 0 \\ 0 & 0 & -i \\ 0 & -i & 0 \end{pmatrix}
			c^{-2/3} \begin{pmatrix} K_1 \\ K_2 \\ K_3 \end{pmatrix} \begin{pmatrix} \omega^2 & \omega & 1 \end{pmatrix}
				z^{-1/3} + \O(1) \\
				& = \O(1) \qquad \text{as } z \to 0 
\end{align*}
because of the relation \eqref{relationK} satisfied by the constants $K_j$.
Thus $\widetilde{N}_{\alpha}^{-1}(y_n) \mathcal M_{\alpha}^* \widetilde{N}_{\alpha}(x_n)  = \O(n^{1/2})$
and \eqref{Enestimate3} follows.

This completes the proof of Lemma~\ref{lem-E-Mstar-R}.
\end{proof}

From \eqref{Kn}, \eqref{xnyn} and Lemmas~\ref{lem-f-tau} and \ref{lem-E-Mstar-R}, we obtain
\[ \lim_{n\to \infty} 
	\frac{c^{*}}{n^{3/2}}K_{n}\left(\frac{c^{*}x}{n^{3/2}}, \frac{c^{*}y}{n^{3/2}} ; 
		t^* - \frac{c^* \tau}{n^{1/2}} \right) = K^{crit}_{\alpha} (x,y;\tau),
\]
where $K^{crit}_{\alpha} $ is given by \eqref{KcritRHP}. This proves Theorem~\ref{thm:PhiRHP}.

\subsection{Proof of Theorem~\ref{corrkernel}}

Let us analyze the expression \eqref{KcritRHP} for the critical kernel. 
We define
\[ \Phi_{\alpha}^{up} = \begin{pmatrix} p_2 & p_3 & p_1 \\ p_2' & p_3' &
p_1' \\ p_2'' & p_3'' & p_1'' \end{pmatrix} \] as an analytic
matrix-valued function in $\mathbb C \setminus i \mathbb R_-$. It
is the analytic continuation of the restriction of $\Phi_{\alpha}$ to the
upper sector $\pi/4 < \arg z < 3 \pi/4$ to the cut plane $\mathbb C \setminus i \mathbb R_-$.

Then by the jump relations of $\Phi_{\alpha}$, see \eqref{Phijumps}, 
we can rewrite \eqref{KcritRHP}  as
\begin{equation}\label{critKalter}
K_\alpha^{crit}(x,y;\tau)
    = \frac{1}{2\pi i(x-y)} \begin{pmatrix} 0 & 1 & 0
    \end{pmatrix} \left(\Phi_{\alpha}^{up}\right)^{-1}(y) \Phi_{\alpha}^{up}(x) \begin{pmatrix} 1 \\ 0
    \\ 0 \end{pmatrix},
\end{equation}    
for $x, y > 0$. Clearly,
\begin{equation}\label{critKcol} 
\Phi_{\alpha}^{up}(x) \begin{pmatrix} 1 \\ 0 \\ 0 \end{pmatrix} = \begin{pmatrix} p_2(x) \\ p_2'(x) \\ p_2''(x)
\end{pmatrix}. 
\end{equation}

%\section{The inverse of $\Phi$ and the adjoint equation}
The inverse of $\Phi_{\alpha}^{up}$ is built out of solutions of the
differential equation
\begin{equation} \label{eqadj}
x q''' + (3-\alpha) q'' - \tau q' + q = 0,
\end{equation}
which is, up to a sign, the adjoint of the equation \eqref{ODEnew}.
%\[ x p''' + \alpha p'' - \tau p' + p = 0. \]

Define the pairing
\begin{multline} \label{def:concomitantXY}
	[p(x),q(y) ] =  (y q'' (y) -(\alpha-2) q' (y) -\tau q (y) )p (x) \\ 
	 + (-y q' (y)   + (\alpha-1)q (y) ) p' (x) +yq (y)p'' (x),
\end{multline}
and denote $[p,q](x)=[p(x),q(x)]$ which is the bilinear concomitant. Then
%\begin{align*}
%    [p, q](x) & = x \left( p(x) q''(x) - p'(x) q'(x) + p''(x) q(x) \right) \\
%     &  \qquad - (\alpha -2) p(x) q'(x) + (\alpha-1) p'(x) q(x) - \tau p(x) q(x)
%     \end{align*}
%    which is such that
\begin{multline*}
    \frac{d}{dx} [p, q](x)  =
    p(x) \left(x q'''(x) + (3-\alpha) q''(x) - \tau q'(x) + q(x)\right) \\
            + q(x) \left(x p'''(x) + \alpha p''(x) - \tau p'(x) - p(x) \right),
        \end{multline*}
which shows that if $p$ and $q$ satisfy the respective differential equations,
then the bilinear concomitant $[p,q](x)$ is constant.

To find the inverse of $\Phi_{\alpha}^{up}$ we need solutions, that we call
$q_{1}$, $q_{2}$, and $q_{3}$, dual to $p_1$, $p_2$, $p_3$,
satisfying
\begin{equation}\label{dual}
[p_j ,q_k] = \delta_{j,k}, \qquad j, k = 1, 2, 3.
\end{equation}
The inverse matrix is then given by
\[ \left(\Phi_{\alpha}^{up}\right)^{-1}(z) =\begin{pmatrix}
 zq_2''(z)-(\alpha-2) q_2'(z)-\tau q_2(z) & -zq_2'(z) + (\alpha-1)q_2(z) & zq_2(z) \\
 zq_3''(z)-(\alpha-2) q_3'(z)-\tau q_3(z) & -zq_3'(z) + (\alpha-1)q_3(z) & zq_3(z) \\
 zq_1''(z)-(\alpha-2) q_1'(z)-\tau q_1(z) & -zq_1'(z) + (\alpha-1)q_1(z) & zq_1(z)
\end{pmatrix},
\]
and
\begin{equation} \label{critKrow}
\begin{pmatrix} 0 & 1 & 0
    \end{pmatrix}  \left(\Phi_{\alpha}^{up}\right)^{-1}(y) = 
    \left(y q_3''(y)-(\alpha-2) q_3'(y)-\tau q_3(y) , -yq_3'(y) + (\alpha-1)q_3(y) , yq_3(y)\right).
\end{equation}
Hence, the solution $q_3$ is the relevant one for the critical kernel: by \eqref{critKalter}--\eqref{critKrow}, we get
$$
K^{crit}(x,y)=\frac{[p_2(x),q_3(y)]}{2\pi i(x-y)}.
$$
%which coincides with the first identity in \eqref{Kcritical}.

%In order to construct $q_{j}$'s satisfying \eqref{bilinear} we need the integral expressions both for the solutions of \eqref{eqadj} and for the bilinear concomitant. 
Before continuing, let us build the dual functions for $p_j$. 
The solutions $q_{j} (z)$,
$j=1,2, 3$, of
\eqref{eqadj} admit integral representations
\begin{equation}\label{defv}
q_{j} (z)= C_j \int_{\Sigma_{j}}t^{-\alpha}e^{-\tau/t}e^{-1/ (2t^{2})} e^{-zt}dt.
\end{equation}
%where, unlike in the definition of functions $p_j$, we take the main branch of $t^{-\alpha}$ with a cut along the positive semiaxis. 
When the variable $z$ is positive, the contours $\Sigma_{j}$,
$j=1,2,3$, can be chosen as in Figure~\ref{Sigma}; note that the integrals converge since the
contours $\Sigma_{j}$ approach the origin 
tangentially to
the real axis and go to infinity along the positive real axis. We choose the main branch of $t^\alpha$ in \eqref{defv} with the cut in the $t$-plane along $\R_+$, and allow $\Sigma_1$ to go along the upper side of the cut. Observe that we may take the same branch cut in the definition of $p_2$ in \eqref{integralForp2}.
%%%%%%%%%%%%%%%%%%%%%%%%%%%%%%%%%
%\begin{figure}[htb]
%\hspace{1.2cm} {\includegraphics[scale=0.8]{pathsforq}}
%\caption{The contours of integration $\Sigma_{j}$, $j=1,3, 4$, (in bold) that
%can be
%used in the definition \eqref{defv} of the functions $q_{j} (z)$, when
%$z$ is positive. Also shown the contours $\Gamma_{j}$, $j=1,3, 4$ (
%dashed lines)}
%\label{Sigma}
%\end{figure}
%%%%%%%%%%%%%%%%%%%%%%%%%%%%%%%%%

%%%%%%%%%%%%%%%%%%%%%%%%%%%%%%%%%%%%%%%%%%%%%%%%%%%%%%%%%%
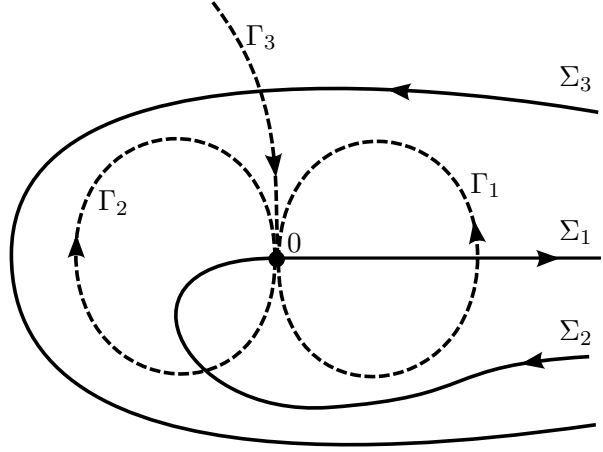
\begin{figure}[t]
\centering 
\begin{overpic}[scale=1.7]% 
{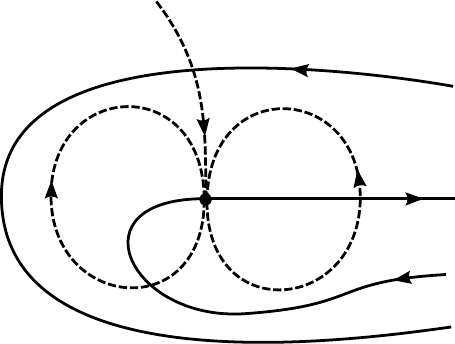}%
%{contourfct}%
\put(47,33){$0 $}
  \put(15,40){$\Gamma_{2} $}
  \put(78,43){$\Gamma_{1}  $}
   \put(40,68){$\Gamma_{3} $}
 \put(93,18){$\Sigma_{2} $}
  \put(93,35){$\Sigma_{1}  $}
   \put(93,61){$\Sigma_{3} $}
    \end{overpic}
\caption{The contours of integration $\Sigma_{j}$, $j=1,2, 3$, (in bold) 
used in the definition \eqref{defv} of the functions $q_{j} (z)$, when
$z$ is positive. Also shown the contours $\Gamma_{j}$, $j=1,2, 3$ (dashed lines).}
\label{Sigma}
\end{figure}
%%%%%%%%%%%%%%%%%%%%%%%%%%%%%%%%%%%%

In order to build an integral expression for the bilinear concomitant we need the following
\begin{lemma}\label{intersection}
Let $p$ and $q$ be solutions of \eqref{ODEnew} and \eqref{eqadj}
respectively, with integral expressions
\[
p (z)=%\frac{1}{2\pi i}
\int_{\Gamma}t^{-1}e^{V(t)} e^{zt}dt,\qquad 
q(z)= \int_{\Sigma}t^{-2}e^{-V(t)}  e^{-zt}dt,
\]
where $\Gamma$ and
$\Sigma$ are one of the contours depicted in Figure~\ref{Sigma}, respectively, and
\begin{equation}\label{defPotential}
V(t)=(\alpha-2) \log(t) +\frac{\tau}{t}+\frac{1}{2t^2}.
\end{equation}
Then
\begin{enumerate}
\item[\rm (a)]
\begin{align*}
\int_{\Gamma}t^{-2}e^{V(t)} e^{zt}dt & = z p''(z) +(\alpha-1)\, p'(z) -\tau p(z), \\
\int_{\Gamma}t^{-3}e^{V(t)} e^{zt}dt & = -\tau z p''(z) +(z-\tau(\alpha-1))\, p'(z) +(\alpha-2+\tau^{2}) p(z), \\
\int_{\Sigma}t^{-3}e^{-V(t)} e^{-zt}dt & = z q''(z) +(2-\alpha)\, q'(z) -\tau q(z).
\end{align*}

\item[\rm (b)] 
%If $\Gamma$ and $\Sigma$ only intersect at their common endpoint $0$,
%then $[p,q]=0$, 
%whereas if $\Gamma$ and $\Sigma$ intersect
%transversally at a non-zero point, and if the contours are oriented so
%that $\Gamma$ meets $\Sigma$ on the $-$-side of $\Sigma$, then
%\[
%[p,q]=2\pi i.
%\]
\begin{equation}
\label{identityconcomitant}
[p (x),q(y) ] =  \int_{\Gamma} \int_{\Sigma} \left(   \frac{V'(t)-V'(s)}{t-s} -\frac{(x-y)(t+s)}{s^2}\right)\, e^{V(t)-V(s)} e^{x t- y s }\,  dt ds.
\end{equation}
In particular,
\begin{equation}
\label{identityconcomitantConfl}
[p ,q ](x)=  \int_{\Gamma} \int_{\Sigma}  \frac{V'(t)-V'(s)}{t-s}\, e^{V(t)-V(s)} e^{x(t-s) }\,  dt ds.
\end{equation}
\item[\rm (c)] If the only point of intersection of $\Gamma$ and $\Sigma$ is the origin, then $[p,q](x)\equiv 0$. 
\item[\rm (d)] If $\Gamma$ and $\Sigma$ intersect transversally at $z_0\neq 0$, and if the contours are oriented so that $\Gamma$ meets $\Sigma$ in $z_0$ on the ``$-$''-side of $\Sigma$, then
$$
[p,q](x)\equiv 2\pi i.
$$
\end{enumerate}
\end{lemma}
\begin{proof}
Let us denote
$$
h(z)=\int_{\Gamma}t^{-2}e^{V(t)} e^{zt}dt.
$$
Then $h'(z)=p(z)= zp'''(z)+(\alpha+2) p''(z)-\tau p'(z)$, where we have used the differential equation \eqref{ODEnew}. Hence, $h(z)=  zp''(z)+(\alpha+1) p'(z)-\tau p(z)+c$, and in order to find the constant $c$ we compute
\begin{align*}
h(z)- & (  zp''(z)+ (\alpha+1) p'(z)-\tau p(z)) =\int_{\Gamma} \left(t^{-2}-zt -(\alpha+1) -\tau t^{-1} \right)e^{V(t)} e^{zt}dt \\
& = - \int_{\Gamma} t \, d\left(  e^{V(t)} e^{zt}\right) - \int_{\Gamma}  e^{V(t)} e^{zt}dt = t  e^{V(t)} e^{zt} \bigg|_{\Gamma}=0,
\end{align*}
due to the selection of the contour. This proves the first identity in (a). We leave the details of the proof of the remaining 
identities in (a) to the reader.
 
Since
$$
 \frac{V'(t)-V'(s)}{t-s}=\frac{1}{t^3 s} + \frac{1}{t^2 s^2}+ \frac{1}{t s^3} + \frac{2-\alpha}{t s} + \tau\, \left(\frac{1}{t^2 s} +\frac{1}{t s^2}  \right),   
$$
the expression in \eqref{identityconcomitant} is obtained by direct substitution of those in (a) into the right hand side and comparison with \eqref{def:concomitantXY}.

Furthermore, if $\Gamma$ and $\Sigma$ do not intersect, then
\begin{align*}
0 & = \int_\Gamma \int_\Sigma \left[ \frac{\partial}{\partial t}\, \frac{1}{t-s} +  \frac{\partial}{\partial s}\, \frac{1}{t-s} \right]e^{V(t)-V(s)} e^{x(t-s)}ds dt \\
& = \int_\Gamma \int_\Sigma e^{-V(s)-x s}   \frac{\partial}{\partial t}  \left(\frac{1}{t-s}\right)  e^{V(t)+x t}  ds dt + \int_\Gamma \int_\Sigma e^{V(t)+x t}   \frac{\partial}{\partial s}  \left(\frac{1}{t-s}\right)  e^{-V(s)-x s} ds dt \\
& = -\int_\Gamma \int_\Sigma e^{-V(s)-x s}    \frac{V'(t)+x}{t-s}   e^{V(t)+x t}  ds dt - \int_\Gamma \int_\Sigma e^{V(t)+x t}    \frac{-V'(s)-x}{t-s}  e^{-V(s)-x s} ds dt\\
&= -[p,q](x),
\end{align*}
where we have used integration by parts and \eqref{identityconcomitant}. 

On the other hand, if the only intersection of $\Gamma$ and $\Sigma$ is at the origin, we can perform the same calculation with $\Gamma$ and $\Sigma_\epsilon= \Sigma\setminus U_\epsilon(0)$, where $U_\epsilon(0)=\{z\in \C:\, |z|<\epsilon\}$. Taking a posteriori $\epsilon \to 0$ and observing that all integrands are strongly vanishing at the origin, we arrive at (c) also in this case.

Assume finally that there exists a point $z_0\neq 0$ such that $\Gamma\cap \Sigma=\{ z_0\}$. Denote now $\Sigma_\epsilon =\Sigma \setminus U_\epsilon(z_0)$. We have
\begin{align*}
[p,q](x)  = & \lim_{\epsilon \to 0} \left[ \int_\Gamma \int_{\Sigma_\epsilon} e^{-V(s)-x s}    \frac{V'(t)+x}{t-s}   e^{V(t)+x t}  ds dt \right. \\
& \left. + \int_\Gamma \int_{\Sigma_\epsilon} e^{V(t)+x t}    \frac{-V'(s)-x}{t-s}  e^{-V(s)-x s} ds dt \right] \\
= & \lim_{\epsilon \to 0} \int_\Gamma \left[      \frac{1}{t-s''}   e^{V(t)-V(s'')+x (t-s'')} -  \frac{1}{t-s'}   e^{V(t)-V(s')+x (t-s')}  \right] dt,
\end{align*}
where $s'$ and $s''$ are the two points of intersection of $\Sigma$ with the circle $|z-z_0|=\epsilon$. We can deform the path of integration $\Gamma$ in such a way that it forms a small loop around $z=s'$, picking up the reside of the integrand, and conclude that
$$
[p,q](x)=2\pi i.
$$
\end{proof}

With account of Lemma~\ref{intersection} we define $q_j$, $j=1,2,3$, as in \eqref{defv}, with paths $\Sigma_j$ specified in Figure~\ref{Sigma} and
\begin{equation}
\label{defC}
C_1=\frac{1}{2\pi i}, \quad C_2=C_3=\frac{e^{\alpha \pi i }}{2\pi i},
\end{equation}
and conclude that condition \eqref{dual} is satisfied.

Let us turn to the equality \eqref{Kcritical}; for that, let us define 
\begin{equation}
\label{Ktilde}
\widetilde{K}(x,y;z)= \int_{t\in\Gamma} \int_{s\in\Sigma}  e^{V(t)-V(s)} e^{xt-ys+(x-y)\log(z) }\, \frac{dt ds}{s-t},
\end{equation}
where the contours are as described in Theorem~\ref{corrkernel}. %Observe that $\widetilde{K}(x,y;1)$ coincides with the right hand side in \eqref{Kcritical}. 
A straightforward computation shows that
$$
\frac{\partial}{\partial z} \widetilde{K}(x,y;z)\bigg|_{z=1}= (x-y) \widetilde{K}(x,y;1)=(x-y) \int_{t\in\Gamma} \int_{s\in\Sigma}  e^{V(t)-V(s)} e^{xt-ys }\, \frac{dt ds}{s-t}.
$$
On the other hand,  the change of variables $t\mapsto t-\log (z)$ and $s\mapsto s-\log (z)$ in  the definition of $\widetilde{K}(x,y;z)$ yields
$$
\widetilde{K}(x,y;z)=  \int_{t\in \Gamma} \int_{s\in \Sigma}  e^{V(t-\log(z))-V(s-\log(z))} e^{xt-ys }\, \frac{dt ds}{s-t},
$$
so that
$$
\frac{\partial}{\partial z} \widetilde{K}(x,y;z)\bigg|_{z=1}=  \int_{t\in \Gamma} \int_{s\in \Sigma}  \frac{V'(t)-V'(s)}{t-s}\, e^{V(t)-V(s)} e^{xt-ys }\,  dt ds.
$$
Thus, we get 
$$
\int_{t\in \Gamma} \int_{s\in \Sigma}  \frac{V'(t)-V'(s)}{t-s}\, e^{V(t)-V(s)} e^{xt-ys }\,  dt ds=(x-y) \int_{t\in\Gamma} \int_{s\in\Sigma}  e^{V(t)-V(s)} e^{xt-ys }\, \frac{dt ds}{s-t}.
$$
From \eqref{identityconcomitant} where we plug in the previous identity, together with the value of $C_3$ in \eqref{defC} and the definition of $p_2$ in \eqref{integralForp2}, we obtain \eqref{Kcritical}.

\section{Acknowledgements}

ABJK and FW acknowledge the support of a Tournesol program for scientific and technological exchanges between Flanders and France, project code 18063PB. 
ABJK is supported by K.U.~Leuven research grant OT/08/33,
FWO-Flanders project G.0427.09, and by the Belgian Interuniversity Attraction Pole P06/02.
A.M.-F. is supported in part by Junta de Andaluc\'{\i}a grants FQM-229, P06-FQM-01735 and P09-FQM-4643. 
ABJK and A.M.-F. are also supported by the Ministry of Science and Innovation of Spain (project code MTM2008-06689-C02-01).

\obeylines
\texttt{
A. B. J. Kuijlaars (arno.kuijlaars@wis.kuleuven.be)
Department of Mathematics
Katholieke Universiteit Leuven
Celestijnenlaan 200B
3001 Leuven, BELGIUM
\medskip
A. Mart\'{\i}nez-Finkelshtein (andrei@ual.es)
Department of Statistics and Applied Mathematics
University of Almer\'{\i}a, SPAIN, and
Instituto Carlos I de F\'{\i}sica Te\'{o}rica y Computacional
Granada University, SPAIN
\medskip
F. Wielonsky (Franck.Wielonsky@cmi.univ-mrs.fr)
Laboratoire d'Analyse, Topologie et Probabilit\'es
Universit\'e de Provence
39 Rue Joliot Curie
F-13453 Marseille Cedex 20, FRANCE
}

\end{document}